\documentclass{amsproc}

\usepackage{amsmath}
\usepackage{amsfonts}
\usepackage{amssymb}
\usepackage{amsthm}
\usepackage{mathrsfs}
\usepackage{amscd}
\usepackage[all]{xy}
\usepackage{amsbsy}
\usepackage{mathtools}

\usepackage[latin1]{inputenc}
\usepackage{tikz}
\usetikzlibrary{backgrounds}
\usetikzlibrary{decorations.pathreplacing}
\usetikzlibrary{decorations.pathmorphing}
\usetikzlibrary{decorations.footprints}
\usetikzlibrary{decorations.markings}
\usetikzlibrary{shapes,arrows}

\usetikzlibrary{positioning}

\usepackage[linktoc=all]{hyperref}
\hypersetup{
    colorlinks,
    citecolor=black,
    filecolor=black,
    linkcolor=black,
    urlcolor=black
}

\newtheorem{theorem}{Theorem}[section]
\newtheorem{prop}[theorem]{Proposition}
\newtheorem{quote_prop}[theorem]{``Proposition''}
\newtheorem{lemma}[theorem]{Lemma}
\newtheorem{cor}[theorem]{Corollary}
\newtheorem*{cor*}{Corollary}

\newtheorem*{theorem*}{Theorem}
\newtheorem*{prop*}{Proposition}
\newtheorem*{cobhyp}{Baez-Dolan Cobordism Hypothesis}
\newtheorem{exercise}{Exercise}[section]

\theoremstyle{definition}
\newtheorem{definition}[theorem]{Definition}
\newtheorem{predef}[theorem]{Preliminary Definition}

\theoremstyle{remark}
\newtheorem{remark}[theorem]{Remark}
\newtheorem{example}[theorem]{Example}

\newtheorem*{example*}{Example}

\DeclareMathOperator{\Set}{Set}

\DeclareMathOperator{\Hom}{Hom}
\DeclareMathOperator{\Aut}{Aut}
\DeclareMathOperator{\Sym}{Sym}

\DeclareMathOperator{\ob}{ob}
\DeclareMathOperator{\Fun}{Fun}

\DeclareMathOperator{\Pre}{Pre}

\DeclareMathOperator{\cat}{Cat}
\DeclareMathOperator{\CSS}{CSS}
\DeclareMathOperator{\Seg}{Seg}
\DeclareMathOperator{\set}{Set}
\DeclareMathOperator{\sSet}{sSet}
\DeclareMathOperator{\Top}{Top}
\newcommand{\Bord}{\mathsf{Bord}}

\DeclareMathOperator*{\colim}{colim}  

\newcommand{\RP}{\mathbb{RP}}

\def\L8{L_\infty}

\def\op{\textrm{op}}

\def\cB{\mathcal B}\def\cC{\mathcal C}\def\cD{\mathcal D}
\def\cF{\mathcal F}
\def\cJ{\mathcal J}\def\cK{\mathcal K}
\def\cM{\mathcal M}\def\cN{\mathcal N}
\def\cR{\mathcal R}\def\cS{\mathcal S}
\def\cW{\mathcal W}

\def\AA{\mathbb A}\def\CC{\mathbb C}
\def\GG{\mathbb G}
\def\KK{\mathbb K}
\def\NN{\mathbb N}
\def\RR{\mathbb R}

\def\ZZ{\mathbb Z}

\def\sA{\mathscr A}\def\sB{\mathscr B}

\def\sK{\mathscr K}

\def\sk{\mathscr k}

\def\bB{\mathbf B}

 \def\bDelta{\mathbf \Delta}

\def\fJ(E){\mathfrak E}
\def\fJ{\mathfrak J}

\def\ev{\mathrm{ev}}
\def\coev{\mathrm{coev}}
\def\Id{\mathrm{Id}}

\begin{document}

\title{Dualizability in Low-Dimensional Higher Category Theory}

\author{Christopher J. Schommer-Pries}

\address{Max-Planck Institute for Mathematics, Vivatsgasse 7, 53111 Bonn, Germany}
\email{schommerpries.chris.math@gmail.com}
\thanks{The author was partially supported by NSF fellowship DMS-0902808.}


\date{August 16, 2013}

\dedicatory{This paper is dedicated to my daughter, Lilith}


\begin{abstract} These lecture notes form an expanded account of a course given at the Summer School on Topology and Field Theories held at the Center for Mathematics at the University of Notre Dame, Indiana during the Summer of 2012. A similar lecture series was given in Hamburg in January 2013. The lecture notes are divided into two parts. 

The first part, consisting of the bulk of these notes, provides an expository account of the author's joint work with Christopher Douglas and Noah Snyder on dualizability in low-dimensional higher categories and the connection to low-dimensional topology. 
The cobordism hypothesis provides bridge between topology and algebra, establishing important connections between these two fields. One example of this is the prediction that the $n$-groupoid of so-called `fully-dualizable' objects in any symmetric monoidal $n$-category inherits an $O(n)$-action. However the proof of the cobordism hypothesis outlined by Lurie is elaborate and inductive. Many consequences of the cobordism hypothesis, such as the precise form of this $O(n)$-action, remain mysterious. The aim of these lectures is to explain how this $O(n)$-action emerges in a range of low category numbers ($n \leq 3$).  

The second part of these lecture notes focuses on the author's joint work with Clark Barwick on the Unicity Theorem, as presented in arXiv:1112.0040. This theorem and the accompanying machinery provide an axiomatization of the theory of $(\infty,n)$-categories and several tools for verifying these axioms. The aim of this portion of the lectures is to provide an introduction to this material.
\end{abstract}

\maketitle





\section*{Introduction}

The cobordism hypothesis \cite{MR2555928} establishes a powerful relationship between extended topological field theories taking values in a symmetric monoidal higher category and objects in that higher category with various kinds of duality. It states that the higher groupoid of ``fully dualizable'' objects in $\cC$ is equivalent to the higher category of fully extended framed $n$-dimensional field theories in $C$, i.e. those where each 
bordism is equipped with a tangential framing. Since this version of the bordism category has a natural action of the orthogonal group (acting by change of framing), such an equivalence induces a 
(homotopically coherent) $O(n)$-action on the groupoid of fully-dualizable objects. 

Understanding this action is fundamental in applications of the cobordism hypothesis, as it provides a bridge to understanding extended topological field theories for bordisms  equipped with a different tangential structure group. 
The groupoid of field theories, say 
with structure group $G$, is obtained by starting with the groupoid of fully-dualizable objects 
and passing to the $G$-homotopy fixed points. 

Unfortunately the Hopkins-Lurie proof of the 
cobordism hypothesis is inductive, and the origin and nature of the induced $O(n)$-action on 
the groupoid of fully-dualizable objects remains largely mysterious and elusive. 

The goal of these lectures is to explain part of the author's joint work with C. Douglas and N. Snyder \cite{DSPS_TC3, DSPS_DTC1, DSPS_DTC2} which explores aspects of the $O(n)$-action on the $n$-groupoid of fully-dualizable objects in a range of ``low'' category numbers ($n\leq 3$). In coming to grips with how a topological group can act on a higher category, we will 
touch upon the beautiful connections between modern homotopy theory and higher category 
theory. 

Following Atiyah and Segal \cite{MR1001453, MR2079383}, a topological field theory may be understood as a symmetric monoidal functor from a geometric or topological category of $d$-manifolds and bordisms to an algebraic category, typically taken to be the category of vector spaces over a fixed field. Both the target category and the source category can be altered. For example the source may be altered  by equipping the manifolds and bordisms with orientations, spin structures, or framings. This provides a rich source of examples and variations on the notion of topological field theory.  

By their very conception topological field theories have the potential for providing a two-way bridge between algebraic structures and topology. Perhaps one of the earliest examples of this is the once-folklore result that 2-dimensional topological field theories (for oriented bordisms) with values in $k$-vector spaces are in natural bijection with commutative Frobenius algebras over $k$. 

Recently there has been dramatic progress developing such a algebraic--topological bridge in the setting of extended topological field theories. An extended topological field theory is a higher categorical extension of the Atiyah-Segal axioms which allows for topological bordisms to be decomposed along submanifolds of arbitrary codimension. This is formalized by replacing the cobordism category with a cobordism $n$-category. More specifically it will be an $n$-category whose objects are 0-manifolds, whose 1-morphisms are 1-dimensional cobordisms, whose 2-morphisms are 2-dimensional cobordisms between the 1-dimensional cobordisms, and so on up to dimension $n$. An extended topological field theory is then defined to be a symmetric monoidal functor from the symmetric monoidal $n$-category $\Bord_n$ to a chosen target symmetric monoidal $n$-category.

The introduction of higher categories adds a new element to the study of topological field theories. Recent developments have shown that there is a homotopy theory of higher categories which in many ways closely mimics the homotopy theory of ordinary spaces. The second part of this manuscript, begining in section \ref{sec:unicity_intro}, provides an expository account of the author's joint work with C. Barwick on this topic \cite{BarSch1112}. One aspect of this is the {\em Homotopy Hypothesis} which states that the homotopy theory of $n$-groupoids should be equivalent to the homotopy theory of $n$-types. This point of view allows for many ideas and techniques from homotopy theory to be imported into the study of extended topological field theories. 

In the past few years the study of extended topological field theories has undergone a dramatic transformation. The proof by Lurie \cite{MR2555928} of the cobordism hypothesis allows the complete classification of these field theories (c.f. also \cite{MR2713992}). In particular one learns that the tangentially framed bordism $n$-category has a particularly nice universal property: it is the free symmetric monoidal $n$-category generated by a `fully-dualizable object'. More precisely the cobordism hypothesis states that the category of extended tangentially framed topological field theories with values in the symmetric monoidal $n$-category $\cC$ is given precisely by the n-groupoid of so-called `fully-dualizable' objects of $\cC$. The property of being fully-dualizable can be thought of as a strong finiteness property.  

The framed bordism category has interesting automorphisms. In particular the group $O(n)$ acts on this symmetric monoidal $n$-category by changing the framings. Thus the cobordism hypothesis predicts that the $n$-groupoid of fully-dualizable objects in any symmetric monoidal $n$-category should also inherit a natural $O(n)$-action. 

It is the purpose of these lectures to try to understand and explain the nature of this action for low values of $n$, namely less than or equal to three. We will start in the first portion (sections~\ref{sec:strictNCats}-\ref{sec:exercises_1}) with a discussion of higher categories. In the second portion (sections~\ref{sec:cats_via_gen_reln}-\ref{sec:exercises_2}) we will describe the $O(1)$-action on 1-dualizable categories, in the third (sections~\ref{sec:Serre_Auto}-\ref{sec:exercises_3}) we describe the $SO(2)$-action on 2-dualizable categories, and in the fourth portion (sections~\ref{sec:3_dual}-\ref{sec:exercises_4}) we describe part of the $SO(3)$-action on 3-dualizable categories together with some applications. In many cases the proofs are only sketched and we refer the reader to the actual papers \cite{DSPS_TC3, DSPS_DTC1, DSPS_DTC2} for complete details. In the final portion (sections~\ref{sec:unicity_intro}-\ref{sec:Unicity}) we delve more thoroughly into the theory of higher categories and in particular into the Unicity theorems \cite{BarSch1112}, which provide a solution to the comparison problem in higher category theory.

\section*{Acknowledgements}

These lectures are based primarily on the author's joint work with Chris Douglas and Noah Snyder. The final section on the unicity theorem is based on joint work with Clark Barwick. Without these individuals, these lectures would not be possible. I would also like to thank Mike Hopkins, Stephan Stolz, and Peter Teichner for many helpful conversations about these ideas. Finally, I would like to extend generous thanks to Ryan Grady, whose careful note-taking and latex skills produced the first (and nearly complete) draft of these notes.

\specialsection*{Higher categories}

Higher categories are higher dimensional versions of categories. Where a category has objects and morphisms passing between these objects, a higher category has objects, morphisms between the objects, 2-morphisms between the morphisms, and possibly still higher morphisms between those. An $n$-category will have morphisms up to dimension $n$. The idea is best conveyed pictorially using pasting diagrams.
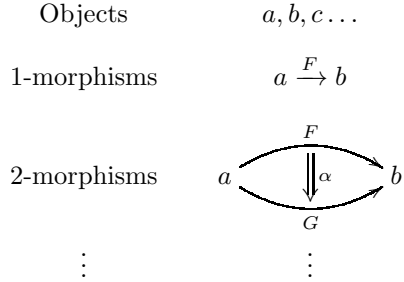
\[
\begin{array}{clc}
\text{Objects} && a,b,c \dotsc \\[2ex]
1\text{-morphisms} && a \xrightarrow{F} b\\[2ex]
2\text{-morphisms} &&\xymatrix{
a\ar@/^1pc/[rr]^{F}_{}="1"\ar@/_1pc/[rr]_G^{}="2"&& b \ar@{=>}^\alpha"1";"2"
}\\[2ex]
\vdots && \vdots

\end{array}
\]
In addition we have various compositions, and coherence equations.  For example, given 2-morphisms
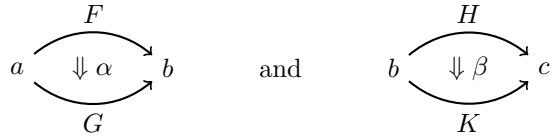
\begin{center}
\begin{tikzpicture}[thick]
\node (a) at (0,0) {$a$};
\node (b) at (2,0) {$b$};
\node (b2) at (5,0) {$b$};
\node (c) at (7,0) {$c$};
\node at (3.5,0) {and};

\draw [->] (a) to [bend left = 40] node [above] {$F$} (b);
\draw [->] (a) to [bend right = 40] node [below] {$G$} (b);
\draw [->] (b2) to [bend left = 40] node [above] {$H$} (c);
\draw [->] (b2) to [bend right = 40] node [below] {$K$} (c);

\node at (1,0) {$\Downarrow \alpha$};
\node at (6,0) {$\Downarrow \beta$};

\end{tikzpicture}
\end{center}

\noindent
we should be able to compose them {\it horizontally} to obtain a new 2-morphism from $H \circ F$ to $K \circ G$
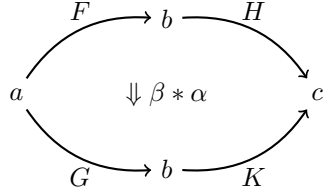
\begin{center}
\begin{tikzpicture}[thick]
\node (a) at (0,0) {$a$};
\node (b) at (2,1) {$b$};
\node (b2) at (2,-1) {$b$};
\node (c) at (4,0) {$c$};

\draw [->] (a) to [bend left = 30] node [above] {$F$} (b);
\draw [->] (a) to [bend right = 30] node [below] {$G$} (b2);
\draw [->] (b) to [bend left = 30] node [above] {$H$} (c);
\draw [->] (b2) to [bend right = 30] node [below] {$K$} (c);

\node at (2,0) {$\Downarrow \beta \ast \alpha$};
\end{tikzpicture}
\end{center}
Now suppose we have 2-morphisms with compatible source and target of the form
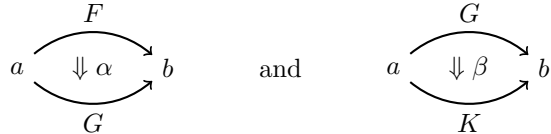
\begin{center}
\begin{tikzpicture}[thick]
\node (a) at (0,0) {$a$};
\node (b) at (2,0) {$b$};
\node (a2) at (5,0) {$a$};
\node (b2) at (7,0) {$b$};
\node at (3.5,0) {and};

\draw [->] (a) to [bend left = 40] node [above] {$F$} (b);
\draw [->] (a) to [bend right = 40] node [below] {$G$} (b);
\draw [->] (a2) to [bend left = 40] node [above] {$G$} (b2);
\draw [->] (a2) to [bend right = 40] node [below] {$K$} (b2);

\node at (1,0) {$\Downarrow \alpha$};
\node at (6,0) {$\Downarrow \beta$};

\end{tikzpicture}
\end{center}
then we should be able to compose them {\it vertically} as follows:
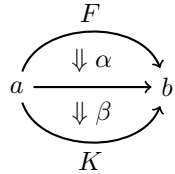
\begin{center}
\begin{tikzpicture}[thick]

\node (a) at (0,0) {$a$};
\node (b) at (2,0) {$b$};

\draw [->] (a) to [bend left = 70] node [above] {$F$} (b);
\draw [->] (a) to (b);
\draw [->] (a) to [bend right=70] node [below] {$K$} (b);

\node at (1,.35) {$\Downarrow \alpha$};
\node at (1,-.35) {$\Downarrow \beta$};
\end{tikzpicture}.
\end{center}
When $n=2$, so that we are thinking about 2-categories, then for each pair of objects, $a$ and $b$, we obtain a category of morphisms between them. The objects of this category will be the morphisms from $a$ to $b$, and the morphisms of this category will be the 2-morphisms between these. The vertical composition is the composition for this {\em hom category}. 

More generally, in an $n$-category, the morphisms between two objects should form an $(n-1)$-category. Thus the theory of higher categories is closely tied to the idea of {\em enriched category theory}. This also leads to the easiest and first attempt at defining higher categories, {\em strict $n$-categories}, which we describe in the next section. However, we will see that strict $n$-categories are not quite adequate for our needs. 

\section{Strict n-Categories} \label{sec:strictNCats}

There are several equivalent ways to define strict $n$-categories. One of the fastest is by induction via the theory of {\em enriched categories}. If $C$ is a category with finite products, then a {\em category enriched in $C$}, say $X$, consists of a set of objects $a, b, c, \dots \in \ob X$ and hom objects $\hom_X(a,b) \in C$ for every pair $a,b \in \ob X$. In addition there are associative composition maps
\begin{equation*}
	\hom_X(b,c) \times \hom_X(a,b) \to \hom_X(a,c)
\end{equation*}
with units $1 \to \hom_X(a,a)$ for each $a$. More generally one may enrich over a monoidal category $(C, \otimes, 1)$; it is not necessary for the monoidal structure to be the categorical product. 
An ordinary (small) category is then the same as a category enriched in sets.

There is an obvious notion of functor between categories enriched in $C$, giving rise to a category of enriched categories. This category will again have finite products, and so we may iterate this process.  

\begin{definition}
	The category of strict $0$-categories is defined as the category of sets. The category of {\em strict $n$-categories} is the category of categories enriched in strict $(n-1)$-categories. 
\end{definition}

\begin{example}
	A strict 1-category is a (small) category in the usual sense.  
\end{example}

\begin{example}
	Any strict $n$-category may be regarded as a strict $(n+1)$-category with only identity $(n+1)$-morphisms. 
\end{example}

	A strict $2$-category consists, in particular, of a set of objects $a,b,c$, and for each pair of objects a category $\hom(a,b)$. The objects of $\hom(a,b)$ are called {\em 1-morphisms} (from $a$ to $b$) and the morphisms of $\hom(a,b)$ are called $2$-morphisms. 

\begin{example}
	We form a (large) strict 2-category with objects (small) categories, 1-morphisms functors between categories, and 2-morphisms given by natural transformations. 
\end{example}

\begin{example}
	The cell $C_k$ is the `free-walking $k$-morphism'. They can be inductively defined as follows. The $0$-cell is the terminal $n$-category, the singleton point. Then $C_k$ is defined to have two objects $0$ and $1$. The $(k-1)$-category of morphisms from $0$ to $1$ consists of $C_{k-1}$. These are the only non-identity morphisms. There is a unique composition making this into a strict $k$-category.
	A functor $C_k \to X$ consists of precisely a $k$-morphism of $X$. 
\end{example}

The category of strict $n$-categories is Cartesian closed; it has finite products and admits an internal hom functor, right adjoint to the cartesian product. Denote this internal hom by $\Fun(X,Y)$ for any two strict $n$-categories $X$ and $Y$. The objects of this $n$-category are the functors $X \to Y$. More generally, the $k$-morphisms are the functors $C_k \times X \to Y$. 

\section{Bicategories}

Most examples of higher categories which `occur in nature' are not strict $n$-categories, but something weaker. One of the earliest examples of such a weaker notion is that of {\em bicategories}, to which we now turn. 

Like a strict 2-category, a bicategory $\cC$ has a collection of objects $a,b,c,\dots$ and for each pair of objects we have a category $\cC(a,b)$, the objects and morphisms of which we refer to as 1-morphisms and 2-morphisms, respectively. There exists a composition functor
\begin{equation*}
	c^{\cC}_{abc}:\cC(b,c) \times \cC(a,b) \to \cC(a,c)
\end{equation*}
and there exist distinguished 1-morphisms $1_a \in \cC(a,a)$ for each object $a \in \ob \cC$. We will also use the notation $f \circ g$ for the composite of 1-morphisms $f$ and $g$, and $\alpha * \beta$ for the `horizontal' composite $c^{\cC}(\alpha, \beta)$ of 2-morphisms $\alpha$ and $\beta$.

Bicategories differ from strict 2-categories in that the composition functor is not required to be associative, nor is the unit object required to be a strict unit for the composition. Instead we are supplied with natural isomorphisms, which in components are:
\begin{align*}
	a:& (h \circ g) \circ f \stackrel{\sim}{\to} h \circ (g \circ f) & \textrm{(associator)} \\
	\ell: & 1_b \circ f \stackrel{\sim}{\to} f & \textrm{(left unitor)} \\
	r: &  f \circ 1_a \stackrel{\sim}{\to} f & \textrm{(right unitor)} 
\end{align*}
for all 1-morphisms $f \in \Hom(a,b)$, $g \in \Hom(b,c)$, and $h \in \Hom(c,d)$. These natural transformations are required to satisfy the {\em pentagon} and {\em triangle} axioms. These say that the following two diagrams commute:
\begin{center}
\begin{tikzpicture}
	\node (MT) at (2.2, 2.5) {$(f\circ g) \circ (h \circ k)$};
	\node (LT) at (0, 1.5) {$((f \circ g) \circ h) \circ k$};
	\node (LB) at (0.5, 0) {$(f \circ (g \circ h)) \circ k$};
	\node (RT) at (5, 1.5) {$f \circ (g \circ (h \circ k))$};
	\node (RB) at (4.5, 0) {$f \circ ((g \circ h) \circ k)$};
	\draw [->] (LT) -- node [left] {$a * 1_k$} (LB);
	\draw [->] (LT) -- node [above left] {$a$} (MT);
	\draw [->] (MT) -- node [above right] {$a$} (RT);
	\draw [->] (RB) -- node [right] {$1_f * a$} (RT);
	\draw [->] (LB) -- node [below] {$a$} (RB);
\end{tikzpicture}\\
\begin{tikzpicture}
	\node (LT) at (0, 1.5) {$(f \circ 1_b) \circ g$};
	\node (MB) at (2, 0) {$f \circ g$};
	\node (RT) at (4, 1.5) {$f \circ (1_b \circ g)$};
	\draw [->] (LT) -- node [left] {$r*1_g$} (MB);
	\draw [->] (LT) -- node [above] {$a$} (RT);
	\draw [->] (RT) -- node [right] {$1_f * \ell$} (MB);
\end{tikzpicture}
\end{center}
whenever the relevant compositions make sense. 

The notion of functor between bicategories is also weaker than the notion for strict 2-categories. We will give the definition in a moment, but first I would like to list some of the examples that we will now have at our fingertips. 

\begin{example}[Strict 2-categories]
	Every strict 2-category will be a bicategory in which the three natural transformations $a$, $\ell$, and $r$ are identities. This includes in particular the bicategory of categories with objects (small) categories, 1-morphisms functors between categories, and 2-morphisms given by natural transformations.
\end{example}


\begin{example}[Monoidal categories $(\cC, \otimes)$] We form a 2-category denoted $\bB (\cC,\otimes)$ with one object $pt$.  The 1-morphisms in $\bB (\cC, \otimes)$ are given by the objects of $\cC$ while the 2-morphisms are the morphisms of $C$.  The horizontal composition is then given by the monoidal structure $\otimes$, i.e.

	\begin{center}
	\begin{tikzpicture}[thick]

	\node (a) at (0,0) {$pt$};
	\node (b) at (2,0) {$pt$};
	\node (c) at (3,0) {$pt$};
	\node (d) at (5,0) {$pt$};

	\node (e) at (7,0) {$pt$};
	\node (f) at (10,0) {$pt$};

	\node at (-1,0) {$\ast :$};
	\node at (2.5,0) {,};
	\node at (6,0) {$\mapsto$};

	\draw [->] (a) to [bend left = 40] node [above] {$a$} (b);
	\draw [->] (a) to [bend right = 40] node [below] {$b$} (b);
	\draw [->] (c) to [bend left = 40] node [above] {$c$} (d);
	\draw [->] (c) to [bend right = 40] node [below] {$d$} (d);
	\draw [->] (e) to [bend left = 40] node [above] {$a \otimes c$} (f);
	\draw [->] (e) to [bend right = 40] node [below] {$b \otimes d$} (f);

	\node at (1,0) {$\Downarrow F$};
	\node at (4,0) {$\Downarrow G$};
	\node at (8.5,0) {$\Downarrow F \otimes G$};

	\end{tikzpicture}
	\end{center}
\end{example}

\begin{example}[The bicategory of algebras]
	Fix a ground ring $\KK$ and form a bicategory with objects $\KK$-algebras, 1-morphisms given by bimodules, and 2-morphisms given by maps of bimodules.  Note that the vertical composition in this bicategory is given by composition of maps of bimodules.  The horizontal composition is given by the tensor product, so given an $A$-$B$ bimodule $M$ and $B$-$C$ bimodule $N$ we have
	\[
	M \ast N = M \otimes_B N
	\]
	viewed as an $A$-$C$ bimodule. The identity bimodule for an algebra $A$ is simply $A$ itself. The unitors and associators are induced by the universal property of the tensor product. 
\end{example}

\begin{example}[The fundamental $2$-groupoid of a space]
	  Given a topological space $X$ we can form a bicategory where all the 2-morphisms are invertible and all the 1-morphisms are weakly invertible (invertible up to a $2$-isomorphism).
We call such bicategories  {\em 2-groupoids}. 
For a space $X$ we denote this groupoid by $\Pi_{\le 2} X$, this $2$-groupoid has objects corresponding to the points of $X$, 1-morphisms are paths in $X$, 2-morphisms are equivalence classes of paths between paths, with the equivalence relation of homotopy relative boundary. Paths are composed in the obvious way, with associators and unitors induced by homotopies which re-parameterize the composite paths. 
\end{example}

\begin{example}[Commutative monoids] \label{exam:commMonoidAsBicat}
	 Consider a 2-category with one object, one 1-morphism, and a set of 2-morphisms $\cS$.  Then the vertical and horizontal compositions give two monoidal structures on $\cS$ such that
	\[
	(p \circ q) \ast (r \circ s) = (p \ast r) \circ (q \ast r) .
	\]
	The Eckmann-Hilton argument (see Exercise \ref{ex:Eckman-Hilton} ) then implies that $\cS$ is a commutative monoid and that $\circ = \ast$. 
	Conversely, every commutative monoid gives rise to a 2-category of this type. 
\end{example}

\begin{definition}
Let $\sA$ and $\sB$ be bicategories. A {\em homomorphism} 
\index{bicategory!homomorphism} 
\index{homomorphism}
\index{homomorphism|see{bicategory}}
$F: \sA \to \sB$ consists of the data:
\begin{enumerate}
\item A function $F: \text{ob } \sA \to \text{ob } \sB$,
\item Functors $F_{ab}: \sA(a,b) \to \sB( F(a), F(b))$,
\item Natural isomorphisms 
\begin{align*}
\phi_{abc}:  c^\sB_{F(a) F(b) F(c) } \circ (F_{bc} \times F_{ab})   & \to   F_{ac} \circ c^\sA_{abc}\\
\phi_a: I^\sB_{F(a)} &\to F_{aa} \circ I_a^\sA 
\end{align*}
(thus invertible  2-morphisms $\phi_{gf}: Fg \circ F f \to F(g \circ f)$ and $\phi_a: I^\sB_{Fa} \to F( I^\sA_a)$ ).
\end{enumerate}
such that the following diagrams commute:
		\begin{center}
		\begin{tikzpicture}[thick]
			\node (A) at (3, 6) {$(Fh \circ Fg) \circ Ff$};
			\node (B) at (0,4) {$Fh \circ (Fg \circ Fh)$};
			\node (C) at (0,2) {$Fh \circ F(g \circ f)$};
			\node (D) at (3,0) {$F(h \circ (g\circ f))$};
			\node (E) at (6,2) {$F((h \circ g) \circ f)$};
			\node (F) at (6,4) {$F(h \circ g) \circ Ff$};
			
			\draw [->] (A) -- node [above left] {$a^\sB$} (B);
			\draw [->] (B) -- node [left] {$1_{Fh} * \phi$} (C);
			\draw [->] (C) -- node [below left] {$\phi$} (D);
			\draw [->] (A) -- node [above right] {$\phi * 1_{Ff}$} (F);
			\draw [->] (F) -- node [right] {$\phi$} (E);
			\draw [->] (E) -- node [below right] {$Fa^\sA$} (D);
		\end{tikzpicture}
		\end{center}
		\begin{center}
		\begin{tikzpicture}[thick]
			\node (A) at (2, 4) {$(Ff) \circ (I_{Fb}^\sB)$};
			\node (B) at (0,2) {$(Ff) \circ (FI^\sA_b)$};
			\node (C) at (2,0) {$F(f \circ I\sA_b)$};
			\node (D) at (4,2) {$Ff$};
			\node (E) at (6,4) {$(I_{Fa}^\sB) \circ (Ff)$};
			\node (F) at (8,2) {$(FI_a^\sA) \circ (Ff)$};
			\node (G) at (6,0) {$F(I_a^\sA \circ f)$};
			
			\draw [->] (A) -- node [above left] {$1_{Ff} * \phi_b$} (B);
			\draw [->] (B) -- node [below left] {$ \phi$} (C);
			\draw [->] (C) -- node [below right] {$Fr^\sA$} (D);
			\draw [->] (A) -- node [above right] {$r^\sB$} (D);
			\draw [->] (E) -- node [above right] {$\phi_{a} * 1_{Ff}$} (F);
			\draw [->] (F) -- node [below right] {$\phi$} (G);
			\draw [->] (G) -- node [below left] {$F\ell^\sA$} (D);
			\draw [->] (E) -- node [above left] {$\ell^\sB$} (D);

		\end{tikzpicture}
		\end{center}
If the natural isomorphisms $\phi_{abc}$ and $\phi_a$ are identities, then the homomorphism $F$ is called a {\em strict homomorphism}. 
\end{definition}

\begin{definition} \label{DefnBicatTransformation}
Let $(F, \phi), (G, \psi): \sA \to \sB$ be two homomorphisms between bicategories. A {\em transformation} 
\index{bicategory!transformation} 
\index{transformation} 
\index{transformation|see{bicategory}}
$\sigma: F \to G$ is given by the data:
\begin{enumerate}
\item 1-morphisms $\sigma_a: Fa \to Ga$ for each object $a \in \sA$,
\item Natural Isomorphisms, $\sigma_{ab}: (\sigma_a)^* \circ G_{ab} \to (\sigma_b)_* \circ F_{ab}$

(thus invertible  2-morphisms $\sigma_f: Gf \circ \sigma_a \to \sigma_b \circ Ff$ for every $f \in \sA_1$).
\end{enumerate}
such that the diagrams in Figure~\ref{FigBicatTransformationAxioms} commute for all 1-morphisms in $\sA$, $f: a \to b$ and $g: b \to c$.
\begin{figure}[h]
\begin{center}
	\begin{center}
		\begin{tikzpicture}[thick]
			\node (A) at (0,6) {$(Gg \circ Gf) \circ \sigma_a$};
			\node (B) at (0,4) {$Gg \circ(Gf \circ \sigma_a)$};
			\node (C) at (0,2) {$ Gg \circ( \sigma_b \circ Ff)$};
			\node (D) at (3,0) {$ (Gg \circ \sigma_b) \circ Ff$};
			\node (E) at (6,2) {$(\sigma_c \circ  Fg) \circ Ff$};
			\node (F) at (6,4) {$\sigma_c \circ (Fg \circ Ff)$};
			\node (G) at (6,6) {$\sigma_c \circ F(g \circ f)$};
			\node (H) at (3,8) {$G(g \circ f) \circ \sigma_a$};
			
			\draw [->] (A) -- node [left] {$a^\sB$} (B);
			\draw [->] (B) -- node [left] {$id_{Gg} * \sigma_f$} (C);
			\draw [->] (C) -- node [below left]  {$(a^\sB)^{-1}$} (D);
			\draw [->] (D) -- node [below right] {$\sigma_g * id_{Ff}$} (E);
			\draw [->] (E) -- node [right] {$a^\sB$} (F);
			\draw [->] (F) -- node [right] {$id_{\sigma_c} * \phi_{g,f}$} (G);
			\draw [->] (A) -- node [above left] {$\psi_{g,f} * id_{\sigma_a}$} (H);
			\draw [->] (H) -- node [above right] {$\sigma_{gf}$} (G);
		\end{tikzpicture}
		\end{center}
		\begin{center}
		\begin{tikzpicture}[thick]
			\node (A) at (0,2) {$I^\sB_{Ga} \circ \sigma_a$};
			\node (B) at (3,3) {$\sigma_a $};
			\node (C) at (6,2) {$ \sigma_a \circ I^\sB_{Fa}$};
			\node (D) at (5,0) {$ \sigma_a \circ (FI^\sA_a)$};
			\node (E) at (1,0) {$(GI^\sA_a) \circ \sigma_a$};
			
			\draw [->] (A) -- node [above left] {$\ell^\sB$} (B);
			\draw [->] (B) -- node [above right]  {$(r^\sB)^{-1}$} (C);
			\draw [->] (A) -- node [below left] {$\psi_a * id_{\sigma_a}$} (E);
			\draw [->] (E) -- node [below] {$\sigma_{I^\sA_a}$} (D);
			\draw [->] (D) -- node [below right] {$id_{\sigma_a} * \phi_a^{-1}$} (C);
		\end{tikzpicture}
		\end{center}
\caption{Transformation Axioms}
\label{FigBicatTransformationAxioms}
\end{center}
\end{figure}

\end{definition}

\begin{definition}
Let $(F, \phi), (G, \psi): \sA \to \sB$ be two homomorphisms between bicategories and let $\sigma, \theta: F \to G$ be two transformations between homomorphisms. A {\em modification} \index{bicategory!modification} 
\index{modification}
\index{modification|see{bicategory}} 
$\Gamma: \sigma \to \theta$ consists of  2-morphisms $\Gamma_a: \sigma_a \to \theta_a$ for every object $a \in \sA$,
such that the following square commutes:
\begin{center}
\begin{tikzpicture}[thick]
	\node (LT) at (0,1.5) 	{$Gf \circ \sigma_a$ };
	\node (LB) at (0,0) 	{$\sigma_b \circ F f$};
	\node (RT) at (3,1.5) 	{$Gf \circ \theta_a$};
	\node (RB) at (3,0)	{$\theta_b \circ Ff$};
	\draw [->] (LT) --  node [left] {$\sigma_f$} (LB);
	\draw [->] (LT) -- node [above] {$id * \Gamma_a$} (RT);
	\draw [->] (RT) -- node [right] {$\theta_f$} (RB);
	\draw [->] (LB) -- node [below] {$\Gamma_b * id$} (RB);
\end{tikzpicture}
\end{center}
for every 1-morphism $f: a \to b$ in $\sA$. 
\end{definition}

\begin{remark}
	We may also form pointed variants of the above notions. A {\em pointed bicategory} $(\sA, p_\sA)$ is a bicategory equipped with a distinguished object $p_\sA \in \sA$. A homomorphism $F$ between pointed bicategories $(\sA, p_\sA)$ and $(\sB, p_\sB)$ is a {\em pointed homomorphism} if $F(p_\sA) = p_\sB$.
	If $(F, \phi)$ and $(G, \psi)$ are two pointed homorphisms from $\sA$ to $\sB$, then a transformation $\sigma$ from $F$ to $G$ is a {\em pointed transformation} if $\sigma_{p_\sA} = I_{p_\sB}$. Finally, if $\sigma$ and $\theta$ are two pointed transformations from $F$ to $G$, and $\Gamma$ is a modification from $\sigma$ to $\theta$, then we say $\Gamma$ is a {\em pointed modification} if $\Gamma_{p_\sA}$ is the identity of $I_{p_\sB}$.
\end{remark}

\section{Higher Categories: Hypotheses of Baez and Dolan}

As we have seen $2$-categories and bicategories already encode a multitude of interesting mathematical structures. There are many approaches to the construction of a theory of weak higher categories, and it is a testament to the creativity and artistry of mathematicians that so many diverse models for higher categories exist. In order to decide which, if any, of these approaches to consider, it will be useful to have some criteria by which we can judge them. 
In \cite{MR1355899}, Baez and Dolan recorded three properties that a reasonable definition of $n$-category should possess.  These properties go under the names of the {\it homotopy}, {\it stabilization}, and {\it cobordism} hypotheses. We will describe the first two properties now. The third property, the cobordism hypothesis, will be the key focus for many of the remaining lectures. We will turn to it in section \ref{sec:cobordism_hypothesis} after introducing a good model of higher categories and discussing duality in higher category theory.  

\subsection{The Homotopy Hypothesis}


As motivation let us recall the fundamental groupoid of a topological space $X$, which we denoted $\Pi_{\le 1} X$, familiar in algebraic topology.  The groupoid $\Pi_{\le 1} X$ has objects the points of $X$ and morphisms given by paths in $X$ up to homotopy.  We have now defined a functor:
\[
\Pi_{\le 1} : \text{Spaces} \to \text{Groupoids} .
\]
Further, we have the classifying space construction which gives a functor:
\[
\lvert - \rvert : \text{Groupoids} \to \text{Spaces}.
\]
Now for a given space $X$, we can compare the homotopy groups of $X$ and $\Pi_{\le 1} X$: 
\begin{align*}
	\pi_0(\lvert \Pi_{\le 1} X \rvert) & \cong \pi_0(X) \\
	\pi_n(\lvert \Pi_{\le 1} X \rvert, x_0) & \cong \begin{cases}
		\pi_1(X, x_0) & n=1 \\
		0 & n > 1
	\end{cases}
\end{align*}
for all basepoints $x_0 \in X$ (such a point is also an object of $\Pi_{\le 1}X$, hence also gives a basepoint in $\lvert \Pi_{\le 1} X \rvert$). 
In short, $\lvert \Pi_{\le 1} X \rvert$ encodes the {\it homotopy} 1-{\it type} of $X$ and loses all higher homotopical information.

We can enhance our motivating example and show that the functors $\Pi_{\le 1}$ and $\lvert - \rvert$ actually form an equivalence of homotopy theories between {\em 1-types} (spaces with trivial higher homotopy groups) and groupoids (see Section \ref{sec:HomotopyTheory} for a discussion of abstract homotopy theories).

 In {\it Pursuing Stacks} \cite{Grothendieck_PS}, Alexander Grothendieck proposed that this equivalence of homotopy theories should be extended to include higher homotopical data. The {\it homotopy hypothesis} states that in a reasonable paradigm of $n$-category, we should have an equivalence of homotopy theories:
\[
\lvert - \rvert : n \text{-groupoids} \rightleftarrows   \left \{\begin{array}{c} \text{Spaces with } \pi_k = 0 \\ \text{for } k>n \end{array} \right \} : \Pi_{\le n}
\]
Spaces $X$ for which $\pi_k(X) = 0$ for all $k>n$ are called homotopy {\em $n$-types}. So the theory of $n$-groupoids and $n$-types should be equivalent. A weaker version of this statement, which is often mentioned, is that we obtain a bijection between equivalence classes of $n$-groupoids and homotopy $n$-types. 

\subsection{The stabilization hypothesis}

We learned above that a 2-category with a single object and single 1-morphism is just a commutative monoid (see example \ref{exam:commMonoidAsBicat}).  A similar analysis shows that this process stabilizes in the sense that an $n$-category with a unique $l$-morphism for each $l<n$ is again a commutative monoid.  We now consider the case of an $(k+n)$-category where there is a unique $l$-morphism for $l<k$, that is, we allow (potentially) interesting higher morphisms.

The {\it stabilization hypothesis} states that any notion of higher category should satisfy the periodicity displayed in Table \ref{tab:periodic_table}: the {\em periodic table of higher category theory}.  The entries describe the structure of a $(k+n)$-category which has $k$ trivial layers, i.e. there is a unique $l$-morphism for $l<k$. Borrowing from topological language, these are $(k-1)$-connected $(k+n)$-categories.  
\begin{table}[htbp]
	\begin{center}
	\begin{tabular}{|l|c|c|c| c|}
	\hline
	&$n=0$&$n=1$&$n=2$ &  $n$ \\
	\hline
	$k=0$ & Set & Category & 2-Category & $n$-Cat\\
	\hline
	$k=1$ & Monoid & Monoidal Cat & Monoidal 2-Cat & \quad$E_1$-$n$-Cat \\
	\hline
	$k=2$ & {\bf Commutative  }& Braided  & Braided  & \quad$E_2$-$n$-Cat\\
	& {\bf Monoid} & Monoidal Cat & Monoidal 2-Cat &\\
	\hline
	$k=3$ & {\bf Commutative } & {\bf Symmetric } & Sylleptic & \quad$E_3$-$n$-Cat\\
	& {\bf Monoid} & {\bf Monoidal Cat} & Monoidal 2-Cat &\\
	\hline
	$k=4$ & {\bf Commutative } & {\bf Symmetric  } &{\bf Symmetric } & \quad$E_4$-$n$-Cat\\
	& {\bf Monoid} & {\bf Monoidal Cat} & \bf Monoidal 2-Cat &\\
	\hline
	$k=5$ & {\bf Commutative } & {\bf Symmetric  } &{\bf Symmetric } & $\cdots$\\
	& {\bf Monoid} & {\bf Monoidal Cat} & \bf Monoidal 2-Cat &\\
	\hline
	\end{tabular}
	\end{center}
	\caption{The periodic table of higher categories.}
	\label{tab:periodic_table}
\end{table}

More compactly, the stabilization hypothesis states that the forgetful functor (which just forgets the unique 0-morphism)
\[
\left \{ \begin{array}{c}
pointed \\
(k)\text{-connected}\\
(n+k+1)\text{-categories}
\end{array}
\right \} 
\to
\left \{ \begin{array}{c}
pointed \\
(k-1)\text{-connected}\\
(n+k)\text{-categories}
\end{array}
\right \}
\]
is an equivalence of homotopy theories if $k \ge n+2$. This is known in several special cases. For $(n+k)$-groupoids, i.e. $(n+k)$-types, this equivalence can be deduced from the homotopy hypothesis and the  Freudenthal suspension theorem. Thus the stabilization hypothesis can be viewed as an extension of this classic result from algebraic topology, an extension which applies to higher categories in which not all morphisms are invertible. 

\subsection{Conclusions about Higher Categories}

The homotopy hypothesis suggests that we should model homotopy $n$-types (spaces with trivial homotopy groups above level $n$) by $n$-groupoids.  At a formal level (at least initially) we could send $n$ to $\infty$ and  model arbitrary spaces  by something we might call $\infty$-groupoids; we take this as a jumping off point to consider $\infty$-categories.  An $\infty$-category is morally a higher category with $l$-morphisms for each $l \in \NN$.

\begin{predef}
An $(\infty, n)$-category is an $\infty$-category where all morphisms are invertible above dimension $n$.
\end{predef}

The homotopy hypothesis then suggests that $(\infty, 0)$-categories ($\infty$-groupoids) should define the same homotopy theory as that of spaces.  There are (again) many models for $(\infty,n)$-categories, but we have the following {\it unicity theorem} which characterizes reasonable models.

\begin{theorem}[Barwick--Schommer-Pries]
There are 4 axioms which characterize the homotopy theory of $(\infty, n)$-categories up to equivalence.
\end{theorem}

The above unicity theorem is discussed in detail in section \ref{sec:Unicity}. Instead of expanding upon the statement of the the theorem here, we present one such model for $(\infty,n)$-categories, namely {\it Segal} $n${\it -categories}.

\section{Segal Categories} \label{sec:SegalCats}

\subsection{Simplicial objects}

Let $\bDelta$ denote the category of ``combinatorial simplices''. The objects of $\bDelta$ consist of the totally ordered sets $[n] = (0, 1, 2, \dots, n)$ and morphisms are the order preserving maps. A {\em simplical object} in a category $C$ is a functor $X: \bDelta^\op \to C$. We will denote the value of $X$ on the object $[n]$ by $X_n$. 

Taking $C = \Set$ we obtain the category of presheaves on $\bDelta$, which we call the category of {\em simplicial sets} $\sSet$.  Any totally ordered set $T$ gives rise to a 
simplical set $\Delta^T$ by the assignment $(\Delta^T)_n = \Hom([n], T)$. For example we have the representable simplicial sets $\Delta^n$, but also simplicial sets such as $\Delta^{\{ 0, 2\}}$. 

A {\em cosimplicial object} of $D$ is a simplical object in the opposite category. In other words it is a functor $c:\bDelta \to D$. If $D$ is cocomplete, then from each cosimplicial object $c$ we get an adjunction:
\begin{equation*}
	|| - ||_c: \sSet \leftrightarrows D: N_c
\end{equation*}
The right adjoint, the {\em $c$-nerve}, is given by the formula $N_c(d)_n = D(c([n]), d)$. The left adjoint is given by the left Kan extension of $c$ along the Yoneda embedding $\bDelta \hookrightarrow \sSet$.  It may be computed as a coend.

For example, we can realize each combinatorial simplex $\Delta^n \in \bDelta \subseteq \sSet$, as the topological simplex $\{ (x_i) \in \RR^{n+1} \; | \; x_i \geq 0, \sum x_i = 1 \}  $. This cosimplicial object induces the classical adjunction
\begin{equation*}
	|-|: \sSet \leftrightarrows \Top: Sing
\end{equation*}
the right adjoint of which associates to a topological space its singular simplicial set. In this case the left adjoint is known as {\em geometric realization}.

\subsection{The categorical nerve} \label{sec:categorical_nerve}

There is a fully-faithful embedding of $\bDelta$ into $\mathsf{Cat}$
which assigns to $[n]$ the corresponding poset category with $n+1$ objects $\{0, 1 , \dotsc , n\}$ and a single morphism from $i$ to $j$ if and only if $i \le j$. When convenient we will identify the object $[n]$ with its image under this functor.

Since $\mathsf{Cat}$ is cocomplete, we obtain an adjunction as above:
\begin{equation*}
	\tau_1: \sSet \leftrightarrows \mathsf{Cat}: N.
\end{equation*}
The left adjoint $\tau_1$ associates to a simplicial set $Y$ its {\em fundamental category} $\tau_1 Y$. The objects of $\tau_1Y$ are the vertices of $Y$, and $\tau_1Y$ is freely generated by the edges of $Y$ modulo relations coming from the 2-simplices. Specifically for each 2-simplex $\sigma \in Y_2$ we impose the relation that $d_0(\sigma)\circ d_2(\sigma) = d_1(\sigma)$. In particular the fundamental category only depends on the 2-skeleton of $Y$ and degenerate edges are identity morphisms. 

In the other direction, given an ordinary category $X$, its {\it nerve} is a simplicial set $NX : \bDelta^{op} \to \Set$.
The set of $n$-simplices of $NX$ consists of the collection of all $n$-tuples of composable morphisms in $X$.  More explicitly, we have
\begin{equation*}
	\begin{array}{ccl}
	[0] &\mapsto& \text{objects of } X \\[2ex]
	[1] & \mapsto &{\displaystyle \text{morphisms   in } X = \coprod_{a,b \in X} X(a,b) = \mathrm{Fun} ([1],X)}\\[2ex]
	[2] & \mapsto &{\displaystyle \text{ composable morphisms in } X = \coprod_{a,b,c \in X} X(a,b) \times X(b,c) = \mathrm{Fun} ([2],X)}\\[2ex]
	\vdots && \quad \vdots\\[2ex]
	[n] & \mapsto & {\displaystyle \mathrm{Fun} ([n] ,X)}
	\end{array}
\end{equation*}

Notice that $NX$ encodes all of the structure of $X$. For instance, the source and target of each morphism can be recovered via the face maps
\begin{equation*}
	d_0, d_1: NX_1 \to NX_0 = \ob X.
\end{equation*}
To recover the composition rule in $X$ we observe that we have three morphisms $[1] \to [2]$ which we picture as:
\begin{center}
\begin{tikzpicture}[thick]

\node (a) at (0,0) {$0$};
\node (b) at (1.5,0) {$1$};
\node (c) at (3,0) {$2$};

\node (d) at (4,0) {$0$};
\node (e) at (5.5,0) {$1$};
\node (f) at (7,0) {$2$};

\node (g) at (8,0) {$0$};
\node (h) at (9.5,0) {$1$};
\node (k) at (11,0) {$2$};

\draw[->] (a) to (b);
\draw[->] (b) to (c);
\draw[->] (d) to (e);
\draw[->] (e) to (f);
\draw[->] (g) to (h);
\draw[->] (h) to (k);

\node (l) at (0,-2) {$0$};
\node (m) at (1.5,-2) {$1$};
\node (n) at (5.5,-2) {$0$};
\node (o) at (7,-2) {$1$};
\node (p) at (8,-2) {$0$};
\node (q) at (11, -2) {$1$};

\draw[->] (l) to (m);
\draw[->] (n) to (o);
\draw[->] (p) to (q);

\draw[->] (0.75,-1.5) -- node [left] {$d_2$} (0.75,-.5);
\draw[->] (6.25,-1.5) -- node [left] {$d_0$}(6.25,-.5);
\draw[->] (9.5,-1.5) -- node [left] {$d_1$} (9.5, -.5);

\end{tikzpicture}
\end{center}
Using the first two maps, $d_2$ and $d_0$, we may express the category $[2]$ as a pushout (of either posets or categories):
\begin{equation*}
	[2] = \{0,1\} \cup^{\{1\}} \{1,2\},
\end{equation*}
which induces an isomorphism for the nerve of any category:
\begin{equation*}
	(d_2, d_0) : NX_2 \stackrel{\cong}{\to} NX_1 \times_{NX_0} NX_1.
\end{equation*}
The composition map is then obtained as the composite:
\[
\circ : NX_1 \times_{NX_0} NX_1 \xleftarrow{\cong} NX_2 \xrightarrow{d_1} NX_1.
\]

From these observations it follows that the counit map $\tau_1 NX \to X$ is an isomorphism of categories. Equivalently, the nerve functor $N$ is fully-faithful. 
%
%
We can ask how to characterize the image of the nerve functor, i.e. for which simplicial sets $Z$ does there exist a category $X$ such that $Z = N X$? 

One characterization is that the $n$-simplices of $NX$ are obtained as iterated pullbacks. The {\em spine} $S_n$ of the simplex $\Delta[n]$ is a sub-simplicial set consisting of the union of all the consecutive 1-simplices. There is the natural inclusion of simplicial sets
\begin{equation*}
	s_n: S_n = \Delta^{\{0,1\}} \cup^{\Delta^{\{1\}}} \Delta^{\{1,2\}} \cup^{\Delta^{\{2\}}} \cdots  \cup^{\Delta^{\{n-1\}}}  \Delta^{\{n-1,n\}} \to \Delta[n],
\end{equation*}
which corepresents the $n^\textrm{th}$ {\em Segal map}:
\begin{equation*}
	s_n: Z_n \to Z(S_n) = Z_1 \times_{Z_0}  Z_1 \times_{Z_0} \cdots \times_{Z_0}  Z_1.
\end{equation*}
A simplicial set is isomorphic to the nerve of a category if and only if each Segal map is a bijection for $n \geq 1$. Morever the full subcategory of simplical sets satisfying this property is equivalent to the category of small categories and functors. 

\subsection{Segal categories} \label{sec:Segal_cats_subsection}
Recall that bicategories are a model for weak 2-categories; that a bicategory is essentially a category  enriched in categories, but where the associativity of 1-morphisms is only required up to higher (coherent) isomorphism.  Hence, we are naturally led to define $(\infty,1)$-categories as a categories enriched in $(\infty,0)$-categories.  Via the homotopy hypothesis we could then say that an $(\infty,1)$-category is a {\it topological category}, i.e. a category enriched in spaces.  Such a definition is perfectly reasonable, however topological categories can be a bit unwieldy (for instance, to permit may examples it is sometimes desirable that composition is associative only up to higher coherent homotopy). Therefore we present a closely related, but more flexible approach.

A {\it Segal category} is a homotopical weakening of the simplicial structure we have just described.  They appear to have first been studied by Dwyer, Kan, and Smith by a different name \cite{MR984042}, but see also \cite{MR1317592}. 


\begin{definition}
A {\it Segal category} is a simplicial space $X$ (i.e. a functor $X : \bDelta^{op} \to \mathsf{Spaces}$) such that
\begin{itemize}
\item {\bf Discreteness.} $X_0$ is discrete;
\item {\bf Segal Condition.} For each $n>0$ the Segal map is a homotopy equivalence
\[
s_n: X_n \xrightarrow{\simeq} \underbrace{ X_1 \times_{X_0} X_1 \times \dotsb \times X_1 \times_{X_0} X_1 }_{n \text{ factors}}.
\]
\end{itemize}
\end{definition}

The Segal condition guarantees that we have a notion of composition which is coherent up to higher homotopy.

In practice, it is extremely useful that Segal categories can be distinguished as simplicial spaces which satisfy a certain lifting condition.  More precisely we have the following.

\begin{theorem}[Bergner \cite{MR2664620}, Hirschowitz-Simpson \cite{math.AG/9807049}]
There is a Quillen model category structure on the category of simplicial spaces with $X_0$ discrete such that Segal categories are the fibrant objects in this model structure.
\end{theorem}

\noindent This theorem is beyond the scope of these notes, however the relevant notion of weak equivalence is important and will play a role in what follows. 

Let $X$ be a Segal category and let $a,b \in X_0$ be a pair of objects. We define $X(a,b)$ as the pullback 
\begin{center}
\begin{tikzpicture}
	\node (LT) at (0, 1.5) {$X(a,b)$};
	\node (LB) at (0, 0) {$pt$};
	\node (RT) at (3, 1.5) {$X_1$};
	\node (RB) at (3, 0) {$X_0 \times X_0$};
	\draw [->] (LT) -- node [left] {$$} (LB);
	\draw [->] (LT) -- node [above] {$$} (RT);
	\draw [->] (RT) -- node [right] {$(d_0 , d_1)$} (RB);
	\draw [->] (LB) -- node [below] {$(a,b)$} (RB);
	\node at (0.5, 1) {$\ulcorner$};
\end{tikzpicture}
\end{center}
Similarly for a triple of objects $a,b,c \in X_0$ we define $X(a,b,c)$ as the fiber of $X_2$ over $(a,b,c) \in X_0^{\times 3}$. Since $X$ satisfies the Segal conditions the Segal maps induce an equivalence
\begin{equation*}
	X(a,b,c) \stackrel{\sim}{\to} X(a,b) \times X(b,c).
\end{equation*}

\begin{definition}
Let $X$ be a Segal category.  The {\it homotopy category} of $X$, denoted $hX$, is the category with objects $X_0$ and morphisms from $a \in X_0$ to $b \in X_0$ given by $\pi_0 X(a,b)$.
\end{definition}
\noindent The composition in this category is induced by the diagram of spaces:
\begin{equation*}
	 X(a,b) \times X(b,c) \stackrel{\sim}{\leftarrow} X(a,b,c) \to X(a,c).
\end{equation*}
One may easily check that it is associative and unital.

\begin{definition}
A map $f: X \to Y$ of Segal categories is an {\it equivalence} if the following conditions are satisfied
\begin{enumerate}
\item The induced map $hf : hX \to hY$ is an equivalence of categories;
\item We have a homotopy equivalence $X(a,b) \xrightarrow{\simeq} Y(f(a), f(b))$ for all $a,b \in X_0$.
\end{enumerate}
\end{definition}



\section{Higher Categories: Segal n-categories}

The construction of Segal categories explained in the last section can be vastly generalized where we replace spaces by an appropriate category with weak equivalences. In good cases this construction can be iterated. Iterating, begining with Segal categories and the above describe equivalences yields Hirschowitz and Simpson's notion of Segal n-category, which is closely related to Tamsamani's notion of weak $n$-category. This notion has many good properties and gives a robust model of $(\infty,n)$-categories. We will now describe some general properties of categories with weak equivalences and the general Segal construction. 

\subsection{Relative Categories}

\begin{definition}
	A {\em relative category} consists of a pair $(\cC, \cW)$ of a category $\cC$ and a subcategory $\cW$ containing all the identities (hence $\cW$ has the same objects as $\cC$). The morphisms of $\cW$ we call {\em weak equivalences}.  We will often abuse notation a write $\cC$ for the pair $(\cC, \cW)$. A {\em relative functor} is a functor of categories which sends weak equivalences to weak equivalences. 
	
	A relative category $(\cC, \cW)$ is called {\em homotopical} if $\cW$ satisfies the {\em 2-out-of-6 property}: If $h:w \to x$, $g:x\to y$, and $f:y \to z$ are three morphisms of $\cC$ and the composites $fg$ and $gh$ belong to $\cW$, then so do the remaining four $f$, $g$, $h$, and $fgh$. 
\end{definition}

\begin{example}
	Let $\cC$ be an ordinary category. We have relative categories: 
	\begin{enumerate}
		\item $\check \cC = (\cC, ob\cC)$ is the minimal relative category structure; the only weak equivalences are identities.
		\item $\hat \cC = (\cC, \cC)$ is the maximal relative category structure; all morphisms are weak equivalences. 
		\item $\underline{\cC} = (\cC, Isom(\cC))$ is the homotopical relative category in which the weak equivalences are the isomorphisms. 
	\end{enumerate}
\end{example}

\begin{example}
	The category of topological spaces or the category of simplicial sets, each with the weak homotopy equivalences form important relative categories.  
\end{example}

\begin{example}
	If $(\cC, \cW)$ is a relative category and $f:\cD \to \cC$ is a functor let $f^{-1}\cW$ denote the subcategory of $\cD$ consisting of all morphisms which map to weak equivalences in $\cC$. Then $(\cD, f^{-1}\cW)$ is a relative category, which is homotopical if $\cC$ is.   
\end{example}

\begin{definition}
	If $(\cC, \cW)$ is a relative category, then the {\em homotopy category of $\cC$} is defined to be $h \cC = \cW^{-1} \cC$, the category obtained by formally inverting the morphisms of $\cW$. It is equipped with the canonical localization $\ell: \cC \to h\cC$.  A relative category $(\cC, \cW)$ is {\em saturated} if $\cW = \ell^{-1}(Isom(h\cC))$. In otherwords $\cC$ is saturated if every morphism which becomes invertible in the homotopy category was already a weak equivalence. 
\end{definition}

Saturated relative categories are homotopical. Every model category gives a saturated relative category. For example the category of sets and bijections, the category of simplicial sets and the weak homotopy equivalences, and the category of categories and equivalences of categories form saturated relative categories. 

\begin{lemma} \label{lma:homotopycatcommutesproduct}
	The homotopy category of a product of relative categories is the product of the homotopy categories, $h(\cM_1 \times \cM_2) \simeq h\cM_1 \times h \cM_2$.
\end{lemma}

\begin{proof}
	Let $(\cM_1, \cW_1)$ and $(\cM_2, \cW_2)$ be relative categories. 
	The category $h\cM_1 \times h \cM_2$ can be obtained as a sequence of two localizations. First consider the relative category $\cM_1 \times \check{\cM_2}$ whose weak equivalences consist of those which are weak equivalences in $\cM_1$ and identities in $\cM_2$. A direct calculation shows $h(\cM_1 \times \check{\cM_2}) = h(\cM_1) \times \check{\cM_2}$. Next we view $h(\cM_1) \times \cM_2$ as the relative category $\check{h(\cM_1)} \times {\cM_2}$, whose weak equivalences are identities in $h(\cM_1)$ and weak equivalences in $\cM_2$. An identical calculation gives $h(\check{h(\cM_1)} \times {\cM_2}) = h(\cM_1) \times h(\cM_2)$. This shows that the product of homotopy categories $h(\cM_1) \times h(\cM_2)$ is universal for functors from $\cM_1 \times \cM_2$ which localize the classes $\cW_1 \times ident(\cM_2)$ and $ident(\cM_1) \times \cW_2$ (these classes commute so the order of localization is irrelevant). But this is precisely the same universal property shared by $h(\cM_1 \times \cM_2)$.
\end{proof}

%
%
%

\subsection{The Segal Category Construction.}

\begin{definition}
	A {\em Segalic relative category} is a triple $\cM = (\cM, \cW, \pi_0)$ consisting of a relative category $(\cM, \cW)$ and a relative functor $\pi_0: (\cM, \cW) \to \underline{\Set} = (\Set, Isom)$ such that
	\begin{enumerate}
		\item the category $\cM$ admits all finite products;
		\item the class $\cW$ is closed under finite products, i.e., for every $(f:x \to y) \in \cW$ and every object $z \in \cM$, the map $(f \times id: x \times z \to y \times z) \in \cW$; and
		\item $\pi_0$ preserves finite products.   
	\end{enumerate}
A {\em Segalic functor} from $(\cM, \cW, \pi^{\cM})$ to $(\cM', \cW', \pi^{\cM'})$ consists a functor $F: \cM \to \cM'$ and a natural isomorphism $\pi_0^{\cM'} \circ F \cong \pi_0^{\cM}$ such that
\begin{enumerate}
	\item $F: (\cM, \cW) \to (\cM', \cW')$ is a relative functor; 
	\item $F$ preserves terminal objects; and  
	\item $F$ is {\em weakly product preserving} in the sense that for all object $x, y \in \cM$ the canonical map (induced by the projections) $F(x \times y) \to F(x) \times F(y)$ is a weak equivalence in $\cW'$.
\end{enumerate}
\end{definition}

Segalic relative categories and Segalic functors form a category. The Segal construction will take as input a Segalic relative category $\cM$ and produce a new Segalic relative category $\Seg(\cM)$, and take Segalic functors to Segalic functors.

The objects of $\Seg(\cM)$ consist of the {\em $\cM$-enriched Segal categories}:

\begin{definition}
	Given a set $S$ let $\Delta_S$ denote the category $(\Delta \downarrow S)$ whose objects consist of pairs $([n], \phi:[n] \to S)$ of an object of $\Delta$ and a set-theoretic map to $S$. The morphisms of $\Delta_S$ are morphisms $[m] \to [n]$ in $\Delta$ which lie over the maps to $S$. We may denote objects of $\Delta_S$ by their ordered image in $S$, e.g. $(s_0, s_1, \dots s_n)$. 
	
	An {\em $\cM$-enriched precategory} is a pair $(S, X)$ consisting of a set $S$ (the set of {\em objects}) and a functor $X: (\Delta_S)^\op \to \cM$ such that $X( s_0 ) = 1 \in \cM$, the terminal object. An $\cM$-enriched precategory satisfies the {\em Segal condition} if for all $(s_0, s_1, \dots, s_n) \in \Delta_S$ the Segal maps
	\begin{equation*}
		X(s_0, s_1, \dots, s_n) \to X(s_0, s_1) \times X(s_1, s_2) \times \dots \times X(s_{n-1}, s_n)
	\end{equation*}
	are weak equivalences. These $\cM$-enriched precategories will be call  {\em $\cM$-enriched Segal categories}.
	
	A morphism of $\cM$-enriched precategories $f:(S, X) \to (T,Y)$ consist of a set map $f: S \to T$ (which induces a functor $f: \Delta_S \to \Delta_T$) and a natural transformation $X \to f^*Y$ of functors $\Delta_S^{\op} \to \cM$. Denote the category of all $\cM$-enriched Segal categories by $\Seg(\cM)$.
\end{definition}

\begin{example}
	If $(\cM, \cW) = (\Set, Isom)$, then an $\cM$-enriched precategory $(S,X)$ is the same as a simplicial set with vertex set $S$.
\end{example}

\begin{example}
	If $\cM = \underline{\cC}$ is an ordinary category with finite products, viewed as a relative category whose weak equivalences are the isomorphisms, then $\Seg(\underline{\cC})$ is equivalent the usual category of $\cC$-enriched categories and enriched functors. 
\end{example}

\noindent The definition of $\cM$-enriched Segal category only uses the structure of the relative category $(\cM, \cW)$ and the existence of finite products in $\cM$. However, while it is clear that under these weaker assumptions the category of $\cM$-enriched precategories admits finite products, it is not clear that the Segal condition is preserved under these products. Property (2), namely that $\cW$ is closed under products, ensures that $\Seg(\cM)$ again has finite products. 

Moreover, a weakly product preserving relative functor induces a functor between categories of enriched Segal objects. To define the weak equivalences of $\Seg(\cM)$ we make use of the product preserving relative functor $\pi_0^\cM: \cM \to \underline{\Set}$. 
\begin{definition}
	Applying $\pi_0: \cM \to \underline{\Set}$ levelwise induces the {\em homotopy category functor}: \begin{equation*}
		h:\Seg(\cM) \to \Seg(\underline{\Set}) \cong \cat.
	\end{equation*}
	In otherwords $hX$ is the category whose objects are the same as those of $X$ and for which the morphisms from $a$ to $b$ consist of the set $\pi_0X(a,b)$. 
\end{definition}

\begin{definition}
	Let $(\cM, \cW, \pi_0^{\cM})$ be a Segalic relative category. 
	A morphism $f: X \to Y$ in $\Seg(\cM)$ is a {\em weak equivalence} if 
	\begin{enumerate}
		\item it induces an equivalence of homotopy categories $hf: hX \to hY$, and
			\item for each pair of objects $a, b$ in $X$, we have an induced weak equivalence:
			\begin{equation*}
				X(a,b) \to Y(fa, fb).
			\end{equation*}
	\end{enumerate}
\end{definition}

\noindent With these weak equivalences $\Seg(\cM)$ becomes relative category which is homotopical or saturated if $\cM$ is. Moreover it is immediate that the class of weak equivalences $\cW_{\Seg(\cM)}$ is closed under products. We may define the product preserving relative functor $\pi_0: \Seg(\cM) \to \underline{\Set}$ as the functor which sends $X$ to the set of isomorphism classes of objects in $hX$. The triple $(\Seg(\cM), \cW_{\Seg(\cM)}, \pi_0)$ is again Segalic, hence we have arrived at:

\begin{theorem}
	The Segal construction $(\cM, \cW, \pi_0^{\cM}) \mapsto (\Seg(\cM), \cW_{\Seg(\cM)}, \pi_0^{\Seg(\cM)})$ defines an endofunctor on the category of Segalic relative categories and Segalic functors. 
\end{theorem}

\begin{example}
	{\em Weak $n$-categories} (in the sense of Tamsamani) are obtained by iterating the above Segal category construction $n$-times with the base case $\cM = \Set$. {\em Segal $n$-categories} $\Seg_n$ are obtained by iterating this construction $n$-times with the base case $\cM = \sSet$. The $n=1$ case of the later is the category of Segal categories, also denoted $\Seg$. These examples are saturated. 
\end{example}

%

\begin{example}
	The inclusion $\underline{\Set} \to \sSet$ induces an inclusion of weak $n$-categories among all Segal $n$-categories. Likewise, $\pi_0: \sSet \to \underline{\Set}$ induces relative functors $\tau_{\leq n}: \Seg_n \to n\cat$. These functors are both compatible with the projection $\pi_0$ to $\Set$. 
\end{example}


\subsection{Segal $n$-categories}

Spelling out the definition of Segal $n$-category we arrive at:

\begin{definition}
A {\it Segal} $n${\it -category} is a simplicial Segal $(n-1)$-category $X_\bullet$, such that
\begin{itemize}
\item {\bf Discreteness.} $X_0$ is discrete;
\item {\bf Segal Condition.} For each $k \geq 0$ we have an equivalence of Segal $(n-1)$-categories
\[
\underbrace{ X_1 \times_{X_0} X_1 \times \dotsb \times X_1 \times_{X_0} X_1 }_{k \text{ factors}} \xleftarrow{\simeq} X_k .
\]
\end{itemize}
\end{definition}

It is a result of Hirschowitz, Simpson, and Pellissier, building on Bergner's work, that Segal $n$-categories can be characterized as the fibrant objects in a certain model category, see \cite{math.AG/9807049, Pellissier_WEC, MR2341955, MR2883823}. This is a model category whose objects are $n$-fold simplical spaces with appropriate discreteness conditions imposed. Simpson's book \cite{MR2883823}, among other things, proves the existence of this model structure. There it is also shown that Segal n-categories satisfy the homotopy hypothesis and also a version of the stabilization hypothesis. 

Many examples of higher categories can be expressed in the language of Segal $n$-categories with less difficulty than other models. For example Lack and Paoli \cite{MR2366560} construct an explicit {\em 2-nerve} which takes a bicategory and produces a weak 2-category in the sense of Tamsamani (which is a special kind of Segal 2-category). They also construct a {\em realization} functor the other way and prove that these functors form a weak equivalence. Segal $n$-categories are a model with a rich supply of examples. 

We may also speak about $k$-morphisms in a Segal $n$-category, which we may define inductive. A $0$-morphism in a Segal $n$-category $X$ is defined to be an object of $X$. A $k$-morphism of $X$ is inductively defined to be a $(k-1)$-morphism of $X(a,b)$ for some pair of objects $a, b \in X$.

\section{Dualizability in 2-categories}

One essential ingredient of the cobordism hypothesis is the notion of {\it dualizability}.  We begin by discussing dualizability in the setting of 2-categories. Later we will compare our notion of dualizable with the more classical notion in the setting of monoidal categories.

\begin{definition}
Let $f: a \to b$ be a 1-morphism in a 2-category and let $\alpha$ be the associator in this 2-category.  We say $f$ is {\it left dualizable} if there exists a 1-morphism $g: b \to a$ and 2-morphisms $\ev : f \circ g \to \mathbf{1}_b$ and $\coev : \mathbf{1}_a \to g \circ f$ such that the following compositions of 2-morphisms are identities
\begin{align*}
	f = f \circ \mathbf{1}_a \stackrel{\Id_f * \coev}{\to} f \circ (g \circ f) \stackrel{\alpha^{-1}}{\to} (f \circ g) \circ f \stackrel{\ev * \Id_f}{\to} \mathbf{1}_b \circ f = f \\
	g = \mathbf{1}_a \circ g \stackrel{\coev * \Id_g}{\to} (g \circ f) \circ g \stackrel{\alpha}{\to} g \circ (f \circ g) \stackrel{\Id_g * \ev}{\to} g \circ \mathbf{1}_b = g .
\end{align*}

\end{definition}

Similarly, we have the notion of {\it right dualizable}.  In terms of pasting diagrams we express the {\it dualizability data} (i.e. the collection $\{f, g, \ev , \coev\}$) by the following:

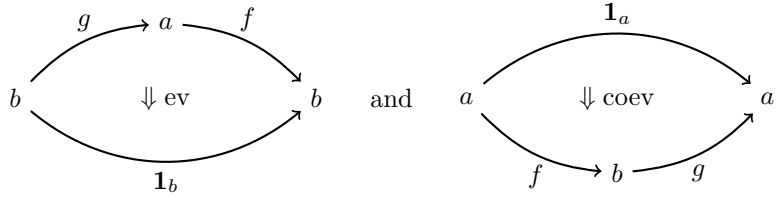
\begin{center}
\begin{tikzpicture}[thick]

\node (a) at (0,0) {$b$};
\node (b) at (2,1) {$a$};
\node (c) at (4,0) {$b$};
\draw[->] (a) to [bend left = 20] node [above] {$g$} (b); 
\draw[->] (b) to [bend left=20] node [above] {$f$} (c);
\draw[->] (a) to [bend right=40] node [below] {$\mathbf{1}_b$} (c);
\node at (2,0) {$\Downarrow \ev$};

\node at (5,0) {and};

\node (d) at (6,0) {$a$};
\node (e) at (8,-1) {$b$};
\node (f) at (10,0) {$a$};
\node at (8,0) {$\Downarrow \coev$};
\draw[->] (d) to [bend left =40] node [above] {$\mathbf{1}_a$} (f);
\draw[->] (d) to [bend right =20] node [below] {$f$} (e);
\draw[->] (e) to [bend right =20] node [below] {$g$} (f);

\end{tikzpicture}
\end{center}

In order to write down the condition that certain compositions are identities, we switch to the Poincar\'e dual picture and use  {\it string diagrams}, which are read from left to right and from top to bottom. In the following examples, the darker region should be thought of representing $a$ and the lighter region $b$. A morphism $f: a \to b$ is given by the following picture:
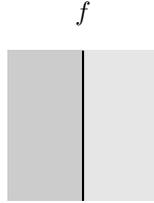
\begin{center}
\begin{tikzpicture}[thick]
\fill [black!20] (4, 2) rectangle (5, 0);
	\fill [black!10] (5, 2) rectangle (6, 0);
	\draw (5, 2) -- (5, 0);
	
	\node at (5,2.5) {$f$};
\end{tikzpicture}
\end{center}

\noindent While the identity $\mathbf{1}_a :a \to a$ corresponds to the following diagram:

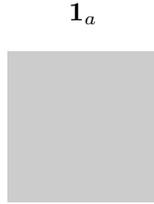
\begin{center}
\begin{tikzpicture}[thick]

\fill [black!20] (4, 2) rectangle (6, 0);
		
	\node at (5,2.5) {$\mathbf{1}_a$};
	
\end{tikzpicture}
\end{center}

\noindent We represent the dualizability data in the following diagrams:

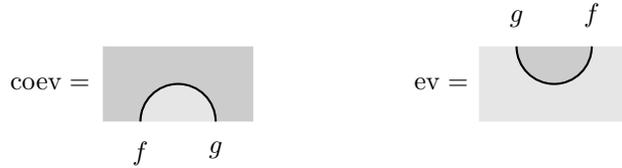
\begin{center}
\begin{tikzpicture}[thick]

\begin{scope}
	\node (LT) at (0, 1) {};
	\node (f) [label=below:$f$] at (.5, 0) {}; 
	\node (g) [label=below:$g$]  at (1.5, 0) {};
	\node (RB) at (2, 0) {};
	\node at (-.7, .5) {$\coev =$};
	
	\draw (f.center) arc (180: 0: 0.5 cm);
	\begin{pgfonlayer}{background}
		\fill [black!20] (LT.center) -- (LT.center |- RB.center) -- 
		(f.center) arc (180: 0: 0.5 cm) -- (RB.center) -- (RB.center |- LT.center) -- cycle;
		\fill [black!10]  (f.center) arc (180: 0: 0.5 cm) -- cycle;
	\end{pgfonlayer}	
\end{scope}

\begin{scope}[xshift = 5cm]
	\node (LT) at (0, 1) {};
	\node (g) [label=above:$g$] at (.5, 1) {}; 
	\node (f) [label=above:$f$]  at (1.5, 1) {};
	\node (RB) at (2, 0) {};
	\node at (-.5, .5) {$\ev =$};
	
	\draw (g.center) arc (180: 360: 0.5 cm);
	\begin{pgfonlayer}{background}
		\fill [black!20] (g.center) arc (180: 360: 0.5 cm) -- cycle;
		\fill [black!10] (LT.center) -- (g.center) arc (180: 360: 0.5 cm) -- (LT.center -| RB.center) -- (RB.center) -| (LT.center);
	\end{pgfonlayer}	
\end{scope}
\end{tikzpicture}
\end{center}

\noindent That the compositions above are identities can be encoded by the following diagrams:

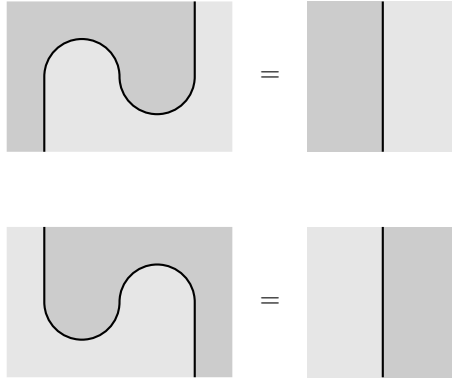
\begin{center}
\begin{tikzpicture}[thick]
\begin{scope}
	\node (LB) at (0,0) {};
	\node (f) at (.5, 1) {};
	\node (g) at (2.5, 2) {};
	\node (RT) at (3, 2) {};
	\draw (f.center |- LB.center) -- (f.center) arc (180:0: 0.5 cm) arc (180: 360: 0.5 cm) -- (g.center);

	\node at (3.5, 1) {=};
	
	\fill [black!20] (4, 2) rectangle (5, 0);
	\fill [black!10] (5, 2) rectangle (6, 0);
	\draw (5, 2) -- (5, 0);
	
	\begin{pgfonlayer}{background}
		\fill [black!20] (LB.center) -- (f.center |- LB.center) -- (f.center) arc (180:0: 0.5 cm) arc (180: 360: 0.5 cm) -- (g.center) -- (g.center -| LB.center) -- cycle;
		\fill [black!10] (f.center |- LB.center) -- (f.center) arc (180:0: 0.5 cm) arc (180: 360: 0.5 cm) -- (g.center) -- (RT.center) -- (RT.center |- LB.center) -- cycle;	
	\end{pgfonlayer}	

\end{scope}

\begin{scope}[yshift = -3cm]
	\node (LT) at (0,2) {};
	\node (f) at (.5, 1) {};
	\node (g) at (2.5, 0) {};
	\node (RB) at (3, 0) {};
	\draw (f.center |- LT.center) -- (f.center) arc (180:360: 0.5 cm) arc (180: 0: 0.5 cm) -- (g.center);

	\node at (3.5, 1) {=};
	
	\fill [black!10] (4, 2) rectangle (5, 0);
	\fill [black!20] (5, 2) rectangle (6, 0);
	\draw (5, 2) -- (5, 0);
	
	\begin{pgfonlayer}{background}
		\fill [black!10] (LT.center) -- (f.center |- LT.center) -- (f.center) arc (180:360: 0.5 cm) arc (180: 0: 0.5 cm) -- (g.center) -- (g.center -| LT.center) -- cycle;
		\fill [black!20] (f.center |- LT.center) -- (f.center) arc (180:360: 0.5 cm) arc (180: 0: 0.5 cm) -- (g.center) -- (RB.center) -- (RB.center |- LT.center) -- cycle;	
	\end{pgfonlayer}	
\end{scope}

\end{tikzpicture}

\end{center}

\begin{example}
To explicate dualizability, let's consider the symmetric monoidal category of vector spaces (say over $\CC$) with monoidal structure given by tensor product.  To this data we have an associated bicategory $\bB \mathsf{Vect}$ with one object and 1-morphisms given by vector spaces.  Let $V$ be a vector space, then $V$ is left dualizable if there exists a vector space $V^\vee$ and linear maps $\ev : V \otimes V^\vee \to \CC$ and $\coev: \CC \to V^\vee \otimes V$ such that the appropriate compositions are identities.  For any vector space we can find its linear dual which plays the role of $V^\vee$ and we always have an evaluation map $\ev$, however we have a coevaluation map $\coev$ satisfying the duality equations if and only if $V$ is finite dimensional.  We deduce that a vector space is left (or right) dualizable precisely when it is finite dimensional.
\end{example}

As suggested by the previous example, one can see that if $(\cC, \otimes)$ is a symmetric monoidal category, then left dualizability of 1-morphisms in $\bB C$ is the same as right dualizability.  In general, left and right dualizability may be distinct as the next example illustrates.

\begin{example}
Consider the 2-category of categories, functors, and natural transformations.  A functor $F: \cC \to \cD$ is left dualizable precisely if it is a left adjoint.  In this case the dual $G: \cD \to \cC$ is right adjoint to $F$ and the 2-morphisms $\ev$ and $\coev$ are the unit and counit for the adjunction.  Of course the same holds for right dualizability and right adjoints.
\end{example}

\section{Duality in Higher Categories}

Before discussing dualizability in Segal $n$-categories, let's give an advertisement for our use of Segal $n$-categories as a model for $(\infty,n)$-categories.  We are principally interested in three aspects of Segal $n$-categories:
\begin{enumerate}
\item It is relatively easy to construct examples of Segal $n$-categories;
\item Given two Segal $n$-categories we can form a Segal $n$-category of functors between them;
\item For a fixed Segal $n$-category we can extract various $k$-categories for $k \le n$; in particular by (carefully) considering three consecutive layers we can associate a 2-category.
\end{enumerate}
We will utilize property (3) in defining dualizability for higher morphisms in an $n$-category.

\begin{definition}
Let $X$ be a Segal $n$-category, with $n\geq2$. The category of Segal $n$-categories may be written as a two-fold iteration of the Segal category construction: $\Seg_n = \Seg(\Seg(\Seg_{n-2}))$. The relative functor $\pi_0:\Seg_{n-2} \to \set$, sending a Segal $(n-2)$-category to its set of isomorphism classes of objects induces a relative functor
\begin{equation*}
	h_2: \Seg_n = \Seg(\Seg(\Seg_{n-2})) \to \Seg(\Seg(\set)) \simeq \cat_2.
\end{equation*}
Let's describe $h_2X$.  The objects of $h_2X$ are just same as those of $X$, the discrete set $X_0$.  For any two objects $a,b \in X_0$ we have a Segal $(n-1)$-category $X(a,b)$ and the category of morphisms in the bicategory $h_2X$ is given by the homotopy category $hX(a,b)$.  That is, 2-morphisms are identified up to homotopy/isomorphism. 

We will say that $h_2X$ is the {\em homotopy bicategory } of $X$, although that is slightly abusive; the horizontal composition in $h_2X$ is not specified but can be chosen up to unique natural isomorphism. Alternatively we may also just was well apply the Lack-Paoli realization to obtain an actual bicategory \cite{MR2366560}.
\end{definition}

\begin{example}
	There is a 3-category $\cat_3$ whose objects are 2-categories, 1-morphisms are functors, 2-morphisms are natural transformations, and 3-morphisms are natural `modifications' between the transformations. Then $h_2 \cat_3$ has the same objects and 1-morphisms, but the 2-morphisms are now isomorphism classes of natural transformations. 
\end{example}

\begin{definition}
Now for each $0 \le k \le n-2$ we will inductively define a 2-category, the $k${\it th level bicategory} $h^{(k)}_2 X$. First we set $h^{(0)}_2 X = h_2 X$. Now notice that if $X$ is a Segal $n$-category, then we may form the union $\sqcup_{a,b \in \ob X} X(a,b)$, which is a Segal $(n-1)$-category. For $k>1$ we set 
 $h^{(k)}_2 X = h^{(k-1)}_2 (\sqcup_{a,b \in \ob X} X(a,b))$.
\end{definition}

Thus $h^{(k)}_X$ has objects the $k$-morphisms of $X$, morphisms the $(k+1)$-morphisms of $X$, and 2-morphisms the equivalence classes of $(k+2)$-morphisms of $X$. 

\begin{definition}
Let $X$ be a Segal $n$-category.  A $(k+1)$-morphism $f$ of $X$ is {\it left dualizable} if $f$ is left dualizable in the $k$th level bicategory $h_2^{(k)}X$. Similarly, a $(k+1)$-morphism $g$ is {\it right dualizable} if $g$ is right dualizable in $h_2^{(k)}X$.
\end{definition}

\subsection{Symmetric monoidal $n$-categories}

In what follows, where we write $n$-category we will implicitly be working with $(\infty,n)$-categories modeled on Segal $n$-categories. We will be interested in symmetric monoidal $n$-categories.  There are several approaches to defining what a symmetric monoidal $n$-category $(\cC, \otimes)$ actually is; here we list three:
\begin{enumerate}
\item The category $\cC$ is an algebra over the $E_\infty$-operad;
\item The category $\cC$ is a $\Gamma$ object, where $\Gamma$ is Segal's category of finite pointed sets;
\item In accordance with the stabilization hypothesis, for all $k$, $\cC$ can be realized as the higher morphism $(\infty,n)$-category of an $(\infty,n+k)$-category with $k$ trivial layers. Thus $\cC$ is taken as a compatible family of such $(\infty,n+k)$-categories.
\end{enumerate}

The first and second approaches can be compared using the same techniques employed by May and Thomason \cite{MR508885}. For concreteness we will always consider symmetric monoidal $(\infty,n)$-categories to be `special $\Gamma$-objects' in Segal $n$-categories. Some work on the third approach has been initiated by Simpson \cite{MR2883823}. In any event these structures are sufficient to endow the homotopy category $hX$ of a Segal $n$-category with the structure of a symmetric monoidal category. An object of $X$ will be said to have a dual if it has a dual in $hX$.

%
%

\subsection{Full-Dualizability}

Let $(\cC, \otimes)$ be a symmetric monoidal $n$-category.

\begin{definition}
The symmetric monoidal $n$-category $\cC$ is $k${\it -fully dualizable} if all $l$-morphisms are both left and right dualizable for $0 \le l \le k-1$ (in the case $l =0$ we mean that the objects of $X$ are dualizable). 
\end{definition}
\noindent In every symmetric monoidal $n$-category there exists a maximal $k$-fully dualizable subcategory. When $k=n$ we denote this as $\cC^\textrm{fd}$. In these notes we will be interested in 2-full dualizability and 3-full dualizability.

\section{The Cobordism Hypothesis} \label{sec:cobordism_hypothesis}

A key player in the remainder will be the $n$-dimensional bordism category $\mathsf{Bord}_n$.  At the 1-categorical level, $\mathsf{Bord}_n$ has as objects closed $(n-1)$-manifolds and morphisms are bordisms between them. Composition in the category is given by gluing bordisms.  $\mathsf{Bord}_n$ is a symmetric monoidal category with respect to disjoint union.  We will often insist that our manifolds come equipped with a certain tangential structure, e.g. framings or orientations.

As an introduction the cobordism hypothesis and its consequences, let's consider the 1-dimensional oriented bordism category $\mathsf{Bord}^{or}_1$, which is generated by the following objects and 1-morphisms (see section \ref{sect:presentation}).

\begin{center}
\begin{tikzpicture}[thick]

\node at (0,0) {Generating objects:};
\node at (3,0) {$+$};
\node at (4,0) {and};
\node at (5,0) {$-$};

\node at (0,-1) {Generating 1-morphisms:};
\draw (3.5,-1) arc (90:-90:0.7cm and 0.4cm);
\node at (5.8,-1.4) {and};
\draw (8,-1) arc (90:270:0.7cm and 0.4cm);
\node at (3.2,-1) {$+$};
\node at (3.2,-1.8) {$-$};
\node at (8.3,-1) {$-$};
\node at (8.3,-1.8) {$+$};

\end{tikzpicture}
\end{center}

\noindent
Additionally, we have the following generating relations.

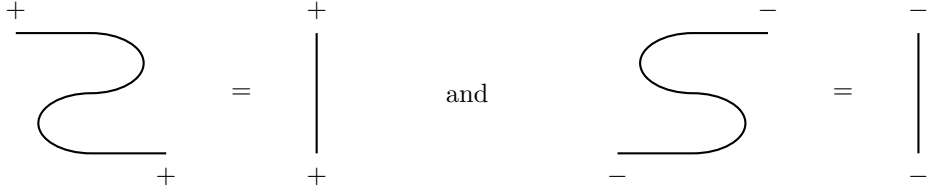
\begin{center}
\begin{tikzpicture}[thick]

\draw (0,0) --(1,0) arc(90:-90:0.7cm and 0.4cm) arc(90:270:0.7cm and 0.4cm) -- (2,-1.6);
\draw (4,0) --(4,-1.6);

\node at (0,.3) {$+$};
\node at (2,-1.9) {$+$};
\node at (4,.3) {$+$};
\node at (4,-1.9) {$+$};
\node at (3,-.8) {$=$};

\node at (6,-.8) {and};

\draw (10,0) --(9,0) arc(90:270:0.7cm and 0.4cm) arc(90:-90:0.7cm and 0.4cm) -- (8,-1.6);
\draw (12,0) --(12,-1.6);

\node at (10,.3) {$-$};
\node at (8,-1.9) {$-$};
\node at (12,.3) {$-$};
\node at (12,-1.9) {$-$};
\node at (11,-.8) {$=$};

\end{tikzpicture}
\end{center}

We are interested in studying symmetric monoidal functors $Z: \Bord_1^{or} \to (\cC , \otimes)$, where $(\cC, \otimes)$ is any symmetric monoidal category. Such a functor $Z$ is called a {\it 1-dimensional topological field theory}.  The 1-dimensional cobordism hypothesis can be stated as follows.

\begin{prop}
Given a symmetric monoidal category $(\cC, \otimes)$, the evaluation $Z \mapsto Z (+)$ determines am equivalence between the groupoid of one dimensional oriented topological field theories with values in $\cC$ and the groupoid of fully dualizable objects in $\cC$.
\end{prop}

We can restate the cobordism hypothesis as an equivalence of categories
\[
\mathrm{Fun}^\otimes (\Bord_1^{or} , \cC) \simeq \sK (\cC^{fd}),
\]
where $\sK (\cC^{fd})$ is the {\it core} of $\cC$ which is the maximal groupoid on dualizable objects. 
 Using the presentation of $\Bord_1^{or}$ an object of $\mathrm{Fun}^\otimes (\mathsf{Bord}_1^{or} , \cC)$ can be given explicitly by a quadruple $(X, X^\vee, \ev, \coev)$ for $X \in \cC$ an object, $X^\vee$ its dual, and $\ev$ and $\coev$ the morphisms exhibiting $X^\vee$ as the dual of $X$. The map to $\sK (\cC^{fd})$ just remembers the object $X$. Exercise \ref{Ex:uniquenessofduals} may be used to show this forgetful functor is indeed an equivalence.

The preceding discussion can be ``beefed up'' to $n$-dimensions.  The $n$-dimensional framed bordism category $\mathsf{Bord}_n^{fr}$ is an $n$-category (really an $(\infty,n)$-category) with objects given by finite disjoint unions of points, 1-morphisms are bordisms between points, 2-morphisms are bordisms between bordisms (so manifolds with corners), and so on.  All the manifolds are compact and framed, that is, every manifold is equipped with a trivialization of its tangent bundle, stabilized to dimension $n$ if necessary. Note that in one dimension framings and orientations are the same notion.  The following was established by Mike Hopkins and Jacob Lurie in dimension 2 and by Lurie for all dimensions $n$ \cite{MR2555928} (compare also \cite{MR2713992} for a different approach).

\begin{cobhyp}
The framed bordism category $\mathsf{Bord}_n^{fr}$ is the free symmetric monoidal $(\infty,n)$-category generated by a single $n$-fully dualizable object.
\end{cobhyp}

As a corollary we have an equivalence of $(\infty,n)$-categories for any symmetric monoidal $(\infty,n)$-category $\cC$:
\[
\mathrm{Fun}^\otimes (\mathsf{Bord}_n^{fr}, \cC) \simeq \sK (\cC^{fd})
\]
where $\sK (\cC^{fd} )$ is the maximal $\infty$-groupoid generated by the maximal $n$-fully dualizable subcategory of $\cC$.  This equivalence is again induced by the evaluation functor which evaluates a topological field theory on a point.

There is a useful and very general principle in mathematics.
Given two categories (spaces, $n$-categories, etc.) $\cB$ and $\cC$, the automorphisms of $\cB$ act on the mapping object $\mathrm{Maps}(\cB, \cC)$.  In the case where $\cB = \mathsf{Bord}_n^{fr}$ we have an action of the orthogonal group $O(n)$ by automorphisms; this action is realized geometrically by the $O(n)$ action on the choice of framings for an $n$-manifold.  Hence, we have an action of $O(n)$ on 
\[
\mathrm{Fun}^{\otimes} (\mathsf{Bord}_n^{fr} , \cC) \simeq \sK (\cC^{fd}).
\]
In what follows, we will explore this $O(n)$ action for $n=1,2,3$.

There is also a version of the cobordism hypothesis for other sorts of topological field theories with different {\em tangential structure}.  Let $M$ be an $n$-manifold and let $\tau : M \to BO(n)$ be the classifying map of its tangent bundle, i.e. $\tau^\ast EO(n) \simeq TM$.  Now given a topological group $G$ and a homomorphism $G \to O(n)$ we can build a bordism category $\mathsf{Bord}_n^G$ where objects are now equipped with a lift of $\tau$ to a map $\widetilde{\tau} : M \to BG$, so the following commutes
\[
\xymatrix{
& BG \ar[d]\\
M \ar[ur]^{\widetilde{\tau}} \ar[r]^\tau & BO(n) .}
\]
Interesting examples are the cases where $G$ is trivial, $O(n)$, $SO(n)$, or $Spin(n)$ which correspond to the tangent bundle of $M$ being framed, no condition, oriented, or spin. This gives the following identifications:
\begin{align*}
	\mathsf{Bord}_n^{\{1\}} \simeq \mathsf{Bord}_n^{fr}, & & \mathsf{Bord}_n^{O(n)} \simeq \mathsf{Bord}_n, \\
	\mathsf{Bord}_n^{SO(n)} \simeq \mathsf{Bord}_n^{or}, &  & \mathsf{Bord}_n^{Spin(n)} \simeq \mathsf{Bord}_n^{spin}.
\end{align*}
The cobordism hypothesis then asserts the following equivalence of categories
\[
\mathrm{Fun}^\otimes (\mathsf{Bord}_n^G , \cC ) \simeq \left [ \sK (\cC^{fd}) \right ]^{hG},
\]
where the superscript indicates passing to homotopy fixed points with respect to the action of $G$.

\section{Exercises} \label{sec:exercises_1}

\begin{exercise}
	Show that the category of strict $n$-categories is Cartesian closed and that the internal hom is as described at the end of section \ref{sec:strictNCats}. Describe the composition of $k$-morphisms.  
\end{exercise}

\begin{exercise}[Eckman-Hilton] \label{ex:Eckman-Hilton}
	Let $S$ be a set with two unital binary operations $\circ$ and $*$ which satisfy the following distributivity law
	\begin{equation*}
		(a \circ b) * (c \circ d) = (a * c) \circ (b * d).
	\end{equation*}
	Show that both $*$ and $\circ$ are commutative and associative, and in fact that they agree identically ($* = \circ$). In particular the units of $*$ and $\circ$ agree.
\end{exercise}

\begin{exercise}
	Let $F: \cC \to \cD$ and $G: \cD \to \cC$ be given functors. Show the following additional structures are equivalent:
	\begin{enumerate}
		\item a bijection $\Hom_\cD(Fc, d) \cong \Hom_\cC(c, Gd)$, natural in both $c \in \cC$ and  $d \in \cD$. 
		\item natural transformations $\varepsilon:FG \to id_\cD$ (counit) and $\eta: id_\cC \to GF$ (unit) satisfying the equations
		\begin{align*}
			 (\varepsilon * 1_F) \circ (1_F * \eta) &= 1_F \\
			 (1_G * \varepsilon ) \circ (\eta * 1_G) &= 1_G.
		\end{align*}
	\end{enumerate}
\end{exercise}

\begin{exercise} \label{Ex:uniquenessofduals}
	Let $F: \cC \to \cD$ be a fixed functor. Define a category of {\em dualizability data} for $F$ as follows. The objects consist of triples $(G, \varepsilon, \eta)$ which witness $G$ as a right adjoint to $F$. A morphism from $(G, \varepsilon, \eta)$ to $(G', \varepsilon', \eta')$ consists of a natural transformation $\phi: G \to G'$ compatible with the unit and counit in the sense that  
	\begin{equation*}
		\varepsilon' \circ (1_F * \phi) = \varepsilon \quad \quad \eta' = (\phi * 1_F) \circ \eta.
	\end{equation*}
	Show that the resulting category is either empty or {\em contractible} (equivalent to the terminal category $pt$). Does a similar result hold if any of the data of $(G, \varepsilon, \eta)$ is removed?

		 Allow $F$ to vary by considering an analogous category of quadruples $(F,G, \varepsilon, \eta)$. Show that the forgetful functor $(F, G, \varepsilon, \eta) \mapsto F$ induces an equivalence of this category with the groupoid consisting of left adjoint functors and natural isomorphisms. Deduce that dualizability data may be ``chosen in families''.   
\end{exercise}

\begin{exercise}[$\star$]
	Let $G$ be a group and $A$ an abelian group. Calculate and compare the maps from $BG$ to $B^2A$ as
	\begin{enumerate}
		\item spaces,
		\item strict 2-categories,
		\item bicategories, and
		\item Segal 2-categories. (using the 2-nerve from \cite{MR2366560})
	\end{enumerate}
\end{exercise}

\begin{exercise}[$\star$]
	Prove that ``left duals = right duals'' for objects in any symmetric monoidal category. To what extent does this hold for braided monoidal categories? Find a monoidal category with duals, but where (some) left duals fail to be isomorphic to the corresponding right dual. 
\end{exercise}

\begin{exercise}
	Explore dualizability in the bicategory of algebras, bimodules, and bimodule maps. Which algebras are dualizable? Which bimodules admit left or right duals? Demonstrate a class of fully 2-dualizable algebras. 
\end{exercise}

%
%

\specialsection*{Understanding the $O(1)$-action}

\section{Defining categories via generators and relations} \label{sec:cats_via_gen_reln}

There are many equivalent ways to define categories. Traditionally a (small) category $C$ consists of a set $C_0$ of objects and for each pair of objects $a$ and $b$, a set of morphism from $a$ to $b$. Taking the union over all pairs of objects yields a global set of morphisms $C_1$. This set comes equipped with source and target maps $s,t: C_1 \to C_0$ as well as associative and unital compositions. 

By a {\em graph} we will mean what is also commonly called a {\em directed multigraph}. It consists of a set of vertices $C_0$ and a set of arrows $C_1$ which have sources and targets which are vertices. In other words a graph consists of two sets $C_0$ and $C_1$, and a pair of maps $s,t: C_1 \to C_0$. The evident category of graphs is a presheaf category on the quiver $* \rightrightarrows *$.  

There is a forgetful functor from categories to graphs and a corresponding free functor (left adjoint to the forgetful functor) which takes a graph and constructs the free category built from that graph. The adjunction
\begin{equation*}
	F: \textrm{Graphs} \rightleftarrows \textrm{Cat}: U
\end{equation*}
is monadic, that is, a category is exactly the same thing as a graph which is an algebra for the monad $UF$. This has several consequences, for example the category of small categories has all small limits and colimits. 

More importantly for any category $X$, the following is a coequalizer diagram:
\begin{equation*}
	FUFU(X) \rightrightarrows FU(X) \to X.
\end{equation*}
As $FU(X) \to X$ is a bijection on objects, it follows that the functor $FU(X) \to X$ is necessarily surjective on morphisms. In otherwords $X$ can be obtained from $FU(X)$ by identifying certain pairs of arrows. This is a special case of a presentation of a category by generators and relations.  

\begin{definition}
	A {\em presentation} of a category $X$ consists of a set of generating objects and arrows, i.e., a graph $G = (G_1 \rightrightarrows G_0)$, together with a set $R$ of pairs of parallel arrows in the free category $F(G)$, and an equivalence of categories between $X$ and the resulting quotient
	\begin{equation*}
		X \simeq (\sqcup_R  C_1) \cup_{\sqcup_R \partial C_2} F(G).
	\end{equation*}
\end{definition}

In the above $C_1$ denotes the free-walking arrow, the category with two objects $0$ and $1$ and a single non-identity arrow which goes from the former to the latter. $\partial C_2$ denotes the free-walking pair of parallel arrows (the reason for this notation will hopefully become clear later). The category $ \partial C_2$ has two objects $0$ and $1$ and there are precisely two non-identity morphisms which both go from $0$ to $1$. A map $r:\partial C_2 \to Y$ consists of precisely a pair of parallel morphisms. There is exactly one functor $\partial C_2 \to C_1$ which is the identity on objects. It collapses the non-trivial morphisms together. The above pushout is formed using this map.  

Given a presentation $(G,R)$ for a category $X$, we may easily describe the category of functors out of $X$ into a target category $Y$. Up to (unique) equivalence such a functor is just a functor out of $F(G)$ such that each pair of arrows in $R$ have the same image in $Y$. A natural transformation of such functors is just a natural transformation of functors $F(G) \to Y$; the relations $R$ play no role for natural transformations. Furthermore functors $F(G) \to Y$ are equivalent to maps of graphs $G \to U(Y)$. That is, such a functor is equivalent to specifying the images of the generating objects and generating morphisms of $X$. Natural transformations also have a simple description in terms of the images of generators. Such a presentation is analogous to other more familiar algebraic presentations, such as those for groups or rings. 

A presentation for a category is also similar to a CW structure for certain spaces, but with extra care given to take account of the fact that the cells of a category are directed. 
There are similar presentations for symmetric monoidal categories.

\section{Presentations for low-dimensional bordism categories}\label{sect:presentation}

Standard Morse-theoretic techniques allow us to obtain presentations for low-dimensional bordism categories. In the 2-dimensional non-extended case this is discussed in \cite{MR2037238}, and the 1-dimensional and extended 2-dimensional cases are covered thoroughly in \cite{MR2713992}. We refer the interested reader to these sources for details. An alternative method would be to use the classification of manifolds of small dimension. 

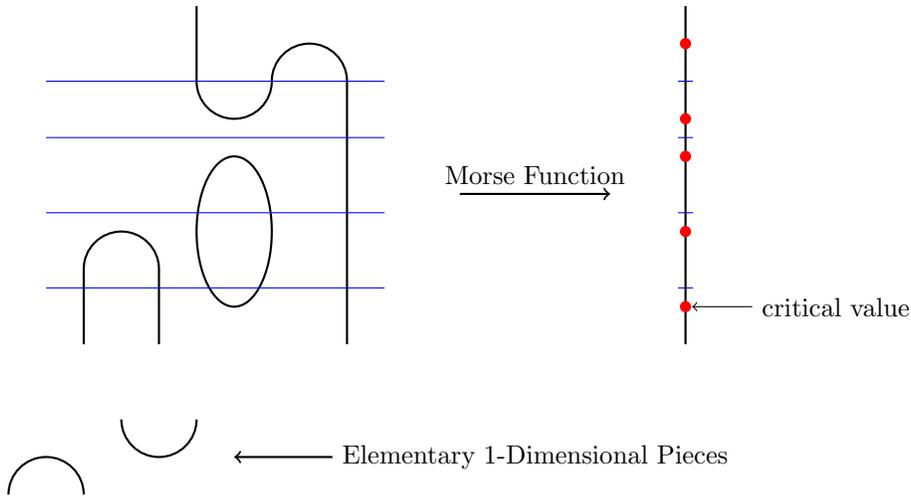
\begin{figure}[ht]
\begin{center}
\begin{tikzpicture}[thick]
	\draw (0,0) -- (0,1) arc (180: 0: 0.5cm) -- (1,0);
	\draw (2, 1.5) ellipse (0.5 cm and 1cm);
	\draw (1.5, 4.5) -- (1.5, 3.5) arc (180: 360: 0.5cm) arc (180: 0: 0.5cm) -- (3.5, 0);
	\draw [thin, blue] (-0.5, 0.75) -- +(4.5,0) +(8.4, 0) -- +(8.6, 0);
	\draw [thin, blue] (-0.5, 1.75) -- +(4.5,0) +(8.4, 0) -- +(8.6, 0);
	\draw [thin, blue] (-0.5, 2.75) -- +(4.5,0) +(8.4, 0) -- +(8.6, 0);
	\draw [thin, blue] (-0.5, 3.5) -- +(4.5,0) +(8.4, 0) -- +(8.6, 0);

	\draw[->] (5, 2) -- node [above] {Morse Function} +(2, 0);
	
	\begin{scope}[xshift=-1cm]
		\draw (9, 4.5) -- +(0, -4.5);

		\node (A) [circle, fill=red,inner sep=1.5pt] at (9, 0.5) {};
		\node [circle, fill=red,inner sep=1.5pt] at (9, 1.5) {};
		\node [circle, fill=red,inner sep=1.5pt] at (9, 2.5) {};
		\node [circle, fill=red,inner sep=1.5pt] at (9, 3) {};
		\node [circle, fill=red,inner sep=1.5pt] at (9, 4) {};

		\node (B) at (11, 0.5)  {critical value};	
		\draw [->, thin] (B) -- (A);
	\end{scope}

	\draw (-1, -2) arc (180: 0: 0.5cm);
	\draw (0.5, -1) arc (180: 360: 0.5cm);
	\node (C) at (6, -1.5) {Elementary 1-Dimensional Pieces};
	\draw [->] (C) -- (2, -1.5);
\end{tikzpicture}
\caption{A one-dimensional decomposition induced by a Morse function}
\label{fig:1DDecompFromMorse}
\end{center}
\end{figure} 

The essential idea is that a Morse function provides a way to decompose any bordism into a composite of elementary bordisms, see Figure \ref{fig:1DDecompFromMorse}. Thus we see that the 1-dimensional bordism category has generating objects the points, and generating morphisms given by the left and right `elbows'. These left and right elbows are precisely the connected bordisms which have exactly one Morse critical point. The relations may similarly be obtained by considering families of Morse functions. In other words by the use of elementary Cerf theory. 

\section{The $O(1)$-action via Presentations}

In this section we dissect the action of $O(1)$ on the core of a fully-dualizable symmetric monoidal category $\cC$.  In the process we describe the unoriented bordism category $\mathsf{Bord}_1$, homotopy quotients, and set the stage for analyzing both the higher dimensional case (which we pursue later) and general tangential structures (a topic which we do not pursue).

Given a symmetirc monoial category $\cC$ we are interested in analyzing the action of $O(1)$ on $\sK (\cC^{fd})$.  In particular, we are interested in the homotopy fixed points of this action as the cobordism hypothesis tells us we have an equivalence
\[
\mathrm{Fun}^\otimes (\mathsf{Bord}_1^{O(1)} , \cC) \simeq \left [ \sK (\cC^{fd}) \right]^{hO(1)}.
\]

One way to understand the action of $O(1)$ on $\sK (\cC^{fd})$ is to understand the $O(1)$ action on $\mathsf{Bord}_1^{or} = \mathsf{Bord}_1^{fr}$, indeed, this is the way in which the action arrises. Now in general for a topological group $G$ (and a pointed map $BG \to BO(1)$) we will have
\[
\mathsf{Bord}_1^{G} = (\mathsf{Bord}_1^{fr})_{hG},
\]
where the subscript denotes the homotopy quotient/coinvariants.  This identification continues to hold in arbitrary dimension where we replace $1$ by $n$.  Hence, we reduce our study to finding a nice presentation for the homotopy quotient of $\mathsf{Bord}_1^{or}$ under the action of $O(1)$.

Recall from above the presentation of $\mathsf{Bord}_1^{or}$.  The action of $O(1) = \ZZ/2$ on this category is, at a first pass, described on generating objects and morphisms by the following assignments.

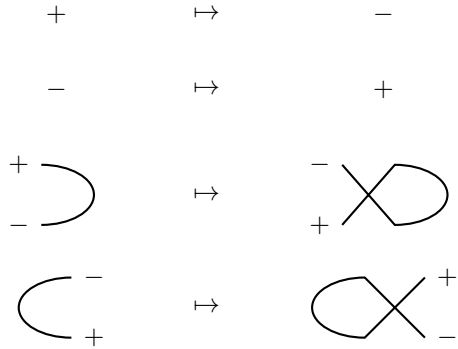
\begin{center}
\begin{tikzpicture}[thick]

\node at (0.2,1) {$+$};
\node at (4.55,1) {$-$};
\node at (2.2,1) {$\mapsto$};

\node at (0.2,0) {$-$};
\node at (4.55,0) {$+$};
\node at (2.2,0) {$\mapsto$};

\draw (0,-1) arc (90:-90:0.7cm and 0.4cm);
\draw (.4,-2.5) arc (90:270:0.7cm and 0.4cm);
\node at (-.3,-1) {$+$};
\node at (-.3,-1.8) {$-$};
\node at (.7,-2.5) {$-$};
\node at (.7,-3.3) {$+$};

\draw (4,-1) -- (4.7,-1.8) arc (-90:90:0.7cm and 0.4cm) -- (4,-1.8);
\node at (3.7,-1) {$-$};
\node at (3.7,-1.8) {$+$};
\draw (5.1,-2.5)--(4.3,-3.3) arc (270:90:0.7cm and 0.4cm) --(5.1,-3.3);
\node at (5.4,-2.5) {$+$};
\node at (5.4,-3.3) {$-$};

\node at (2.2,-1.4) {$\mapsto$};
\node at (2.2,-2.9) {$\mapsto$};

\end{tikzpicture}
\end{center}

The following proposition gives a presentation of noriented bordism category $\mathsf{Bord}_1$, which in light of the above action, identifies $\mathsf{Bord}_1$ with the quotient of $\mathsf{Bord}_1^{or}$ by the action of $O(1)$.

\begin{prop}
The unoriented bordism category $\mathsf{Bord}_1$ has the following presentation.

\begin{center}
\begin{tikzpicture}[thick]

\node at (0,0) {Generating objects:};
\node at (3,0) {$\bullet$};

\node at (0,-1) {Generating 1-morphisms:};
\draw (3.5,-1) arc (90:-90:0.7cm and 0.4cm);
\node at (5.8,-1.4) {and};
\draw (8,-1) arc (90:270:0.7cm and 0.4cm);
\node at (3.5,-1) {$\bullet$};
\node at (3.5,-1.8) {$\bullet$};
\node at (8,-1) {$\bullet$};
\node at (8,-1.8) {$\bullet$};

\end{tikzpicture}
\end{center}

\noindent
Subject to the following relations.

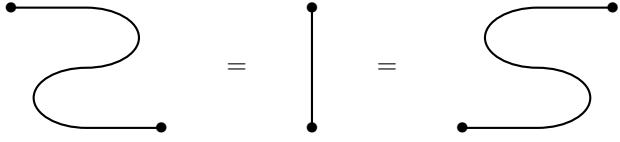
\begin{center}
\begin{tikzpicture}[thick]

\draw (0,0) --(1,0) arc(90:-90:0.7cm and 0.4cm) arc(90:270:0.7cm and 0.4cm) -- (2,-1.6);
\draw (4,0) --(4,-1.6);

\node at (0,0) {$\bullet$};
\node at (2,-1.6) {$\bullet$};
\node at (4,0) {$\bullet$};
\node at (4,-1.6) {$\bullet$};
\node at (3,-.8) {$=$};
\node at (5,-.8) {$=$};

\draw (8,0) --(7,0) arc(90:270:0.7cm and 0.4cm) arc(90:-90:0.7cm and 0.4cm) -- (6,-1.6);

\node at (8,0) {$\bullet$};
\node at (6,-1.6) {$\bullet$};

\end{tikzpicture}
\end{center}

\noindent
and

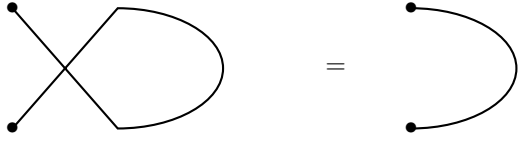
\begin{center}
\begin{tikzpicture}[thick]

\draw (0,0) -- (1.4,-1.6) arc (-90:90:1.4cm and 0.8cm) -- (0,-1.6);
\node at (0,0) {$\bullet$};
\node at (0,-1.6) {$\bullet$};

\node at (4.3,-.8) {$=$};
\draw (5.3,0) arc(90:-90:1.4cm and 0.8cm);
\node at (5.3,0) {$\bullet$};
\node at (5.3,-1.6) {$\bullet$};

\end{tikzpicture}
\end{center}
\end{prop}

So far we have written down the action of $O(1)$ on the category $\mathsf{Bord}_1^{or}$ and identified the quotient category as $\mathsf{Bord}_1$.  However, we have been a bit cavalier and not completely rigorous in our analysis, as we never expressed in detail what it means to have an action of a group on a category, nor have we sufficiently explained the notion of homotopy quotient in the categorical setting.  We rectify our sins in the next section.

\section{Unoriented Bordism as a Homotopy Orbit}

We are interested in a better understanding of the homotopy quotient of the oriented bordism category $\mathsf{Bord}_1^{or}$ under an action of $O(1) = \ZZ/2$. As a preliminary let us consider the situation in spaces.  Let $X$ be a space equipped with an action of a group $G$.  If the action of $G$ on $X$ is not free then the resulting quotient space can be quite badly behaved, homotopically. Instead we free up the action by considering the diagonal action of $G$ on $X \times EG$, where $EG$ is a contractible free $G$ space.  The {\it homotopy quotient} of $X$ by $G$, denoted by $X_{hG}$, is given by
\[
X_{hG} \overset{def}{=} (X \times EG)/G .
\]
If the action of $G$ on $X$ was actually free then we have a homotopy equivalence $X/G \sim X_{hG}$.  For example, if $G$ is the trivial group then $X_{hG} \sim X$.

In spaces we also have the notion of {\it homotopy fixed point sets}.  If $Y$ is a space with a $G$ action, then we can define the space of homotopy fixed points $Y^{hG}$ by
\[
Y^{hG} \overset{def}{=} \mathrm{Maps} (EG, Y).
\]
At this point we could define the homotopy quotient of the core $\sK (\cC^{fd})$ by making use of the homotopy hypothesis to find a space $Y$ such that $\sK (\cC^{fd})$ is the fundamental groupoid of $Y$.  We could then consider the homotopy fixed points of our $O(1)$ action on $Y$.

 Instead we will construct the homotopy quotient of the bordism categories.  The main tool is finding an analog of the universal $G$ space $EG$.  Let's restrict to the case where $G= \ZZ/2$.  We define a category $\cJ$ which is the free-walking isomorphism. That is $\cJ$ has two objects $j, \overline{j}$ and in addition to identities, two morphisms which are inverse to each other. There is an clear free $\ZZ/2$-action on $\cJ$, given by swapping the two objects. Moreover $\cJ$ is contractible and hence it serves as a categorical model of $E\ZZ/2$.
\begin{prop}
Given any symmetric monoidal category $(\cC, \otimes)$ there exists a symmetric monoidal category $\cC \boxtimes \cJ$ characterized by the following equivalence of categories
\[
\mathrm{Fun}^\otimes (\cC \boxtimes \cJ , \cD) \simeq \mathrm{Fun} (\cJ , \mathrm{Fun}^\otimes (\cC, \cD)).
\]
$\cC \boxtimes \cJ$ is unique up to unique equivalence
\end{prop}
\noindent In light of this proposition, we define $\cC_{h\ZZ/2}$ to be the quotient of $\cC \boxtimes \cJ$ by the diagonal (level-wise) action of $\ZZ/2$.

In the case of the oriented bordism category we have the following presentation of $\mathsf{Bord}_1^{or} \boxtimes \cJ$.  The objects of $\mathsf{Bord}_1^{or} \boxtimes \cJ$ are generated by pairs $(b,k)$ where $b \in \mathsf{Bord}_1^{or}$ and $k \in \{j, \overline{j}\}$.  The morphisms are generated by pairs of the form $(f: b \to b' , \mathbf{1}_j)$, $(f: b \to b' , \mathbf{1}_{\overline{j}})$, and $(\mathbf{1}_b, k \to k')$ for $k,k' \in \{j, \overline{j}\}$. Further, morphisms of the form $(f: b \to b', \mathbf{1}_j)$ and $(\mathbf{1}_b, k \to k')$ commute.  We also have isomorphisms of the form
\[
(b,j) \otimes (b',j) \cong (b \otimes b' , j).
\]
Lastly, we have the relations coming from $\mathsf{Bord}_1^{or}$, $\cJ$, monoidal identities, and relations given by commutative diagrams of the following type:
\[
\xymatrix{ (b,j) \otimes (b',j) \ar[r]^{\cong} \ar[d] & (b\otimes b', j) \ar[d] \\(b',j) \otimes (b,j) \ar[r]^{\cong} & (b' \otimes b,j)}
\].

We can now give a presentation for the homotopy quotient of $\mathsf{Bord}_1^{or}$
\[
\left ( \mathsf{Bord}_1^{or} \right )_{h\ZZ/2} =
 \left ( \mathsf{Bord}_1^{or} \boxtimes \cJ \right )_{\ZZ/2}.
\]

\begin{center}
\begin{tikzpicture}[thick]

\node at (0,0) {Generating objects:};
\node at (3,0) {$+$};
\node at (4,0) {and};
\node at (5,0) {$-$};

\node at (0,-1) {Generating 1-morphisms:};
\draw[->] (3.5,-1) arc (90:-90:0.7cm and 0.4cm);
\node at (5.8,-2.4) {and};
\draw[->] (8,-1) arc (90:270:0.7cm and 0.4cm);
\node at (3.2,-1) {$+$};
\node at (3.2,-1.8) {$-$};
\node at (8.3,-1) {$-$};
\node at (8.3,-1.8) {$+$};

\draw(2,-3.4)--(2.5,-3.4);
\draw (4,-3.4)--(4.5,-3.4);
\draw[->] (3.5,-3.4) -- (2.5,-3.4);
\draw[->] (3.5,-3.4) -- (4,-3.4);
\node at (1.7,-3.4) {$-$};
\node at (4.8,-3.4) {$+$};

\draw[->](7,-3.4)--(7.5,-3.4);
\draw[->](9.5,-3.4)--(9,-3.4);
\draw (7.5,-3.4)--(8.5,-3.4);
\draw (9,-3.4)--(8.5,-3.4);
\node at (6.7,-3.4) {$+$};
\node at (9.8,-3.4) {$-$};

\end{tikzpicture}
\end{center}

\noindent
Here we have conflated notation in the sense that the object $+$ represents the orbit of $(+,j)$ which is the same as the orbit of $(-,\overline{j})$ and similarly for the object $- \sim (-, j) \sim (+, \overline{j})$.  The second set of morphisms correspond to the maps $(-,j) \to (-,\overline{j})$ and $(+,j) \to (+,\overline{j})$ respectively.  As far as generating relations, we have the ones coming from the bordism category

\begin{center}
\begin{tikzpicture}[thick]
\draw[->] (0,0) --(.5,0);
\draw(.5,0)-- (1,0) arc(90:-90:0.7cm and 0.4cm) arc(90:270:0.7cm and 0.4cm);
\draw[->] (1,-1.6) --(1.5,-1.6);
\draw(1.5,-1.6)-- (2,-1.6);

\draw[->] (4.3,-.8) --(4.8,-.8);
\draw (4.8,-.8)--(5.3,-.8);
\draw[->] (5.3,-.8)--(5.8,-.8);
\draw (5.8,-.8)--(6.3,-.8);
\node at (0,.3) {$+$};
\node at (2,-1.9) {$+$};
\node at (4,-.8) {$+$};
\node at (6.6,-.8) {$+$};
\node at (3,-.8) {$=$};

\end{tikzpicture}
\end{center}

\noindent
and a similar picture for the object $-$.  Furthermore, we have that the following two morphisms are inverses:
\begin{center}
\begin{tikzpicture}[thick]

\draw(2,-3.4)--(2.5,-3.4);
\draw (4,-3.4)--(4.5,-3.4);
\draw[->] (3.5,-3.4) -- (2.5,-3.4);
\draw[->] (3.5,-3.4) -- (4,-3.4);
\node at (1.7,-3.4) {$-$};
\node at (4.8,-3.4) {$+$};

\draw[->](7,-3.4)--(7.5,-3.4);
\draw[->](9.5,-3.4)--(9,-3.4);
\draw (7.5,-3.4)--(8.5,-3.4);
\draw (9,-3.4)--(8.5,-3.4);
\node at (6.7,-3.4) {$+$};
\node at (9.8,-3.4) {$-$};

\end{tikzpicture}
\end{center}

\noindent
Finally, we have the following relation (and the corresponding one for the object $+$)

\begin{center}
\begin{tikzpicture}[thick]

\draw(-2.5,0)--(-2,0);
\draw (-.5,0)--(0,0);
\draw[->] (-1.5,0) -- (-2,0);
\draw[->] (-1.5,0) -- (-.5,0);
\draw (0,0) -- (1.4,-1.6) arc (-90:90:1.4cm and 0.8cm) -- (0,-1.6);
\draw[->] (0,-1.6)--(-2,-1.6);
\draw (-2,-1.6)--(-2.5,-1.6);

\node at (-2.8,0) {$-$};
\node at (-2.8,-1.6) {$-$};

\node at (3.8,-.8) {$=$};

\draw[->] (7.5,0)--(5.5,0);
\draw (5,0)--(5.5,0);

\draw (7.5,0) arc(90:-90:1.4cm and 0.8cm);

\draw(5,-1.6)--(5.5,-1.6);
\draw (7,-1.6)--(7.5,-1.6);
\draw[->] (6,-1.6) -- (5.5,-1.6);
\draw[->] (6,-1.6) -- (7,-1.6);

\node at (4.7,0) {$-$};
\node at (4.7,-1.6) {$-$};

\end{tikzpicture}
\end{center}
\noindent It is a nice exercise for the reader (exercise \ref{ex:verify_pres_homotopy_quotient}) to verify that these relations come precisely from the ones induced from $\mathsf{Bord}_1^{or} \boxtimes \cJ$.

Given these presentations, we have a functor
\[
\left ( \mathsf{Bord}_1^{or} \right )_{h\ZZ/2} \to \mathsf{Bord}_1^{un}.
\]
This functor send both objects $+$ and $-$ to the generating object of $\mathsf{Bord}_1^{un}$ (i.e. the point). This functor is an equivalence of symmetric monoidal categories (exercise \ref{ex:verify_equivalence}). Thus we have proven the following 1-categorical version of the cobordism hypothesis:

\begin{theorem}
	The unoriented bordism category is the $\ZZ/2$-homotopy quotient of the oriented bordism category $\mathsf{Bord}_1^{or}$ by the natural $\ZZ/2$-action given by reflecting orientations. This is precisely the action which sends an object to its dual. Consequently for any symmetric monoidal category $\cC$ we have an natural equivalence
	\begin{equation*}
		\mathrm{Fun}^\otimes (\mathsf{Bord}_1^{un} , \cC) \simeq \left [ \sK (\cC^{fd}) \right]^{hO(1)}.
	\end{equation*}
	given by evaluating the TFT on the point.  
\end{theorem}

In the above, the $O(1) = \ZZ/2$-action on $\sK (\cC^{fd})$ is not precisely strict but homotopically coherent, meaning there is a monoidal functor from the discrete monoidal category $\ZZ/2$ to the monoidal category $Aut(\sK (\cC^{fd}))$, where the latter is the monoidal category of self-equivalences and natural isomorphisms. This action is induced from the equivalence $\sK (\cC^{fd}) \simeq \mathrm{Fun}^\otimes (\mathsf{Bord}_1^{or} , \cC)$. As we saw there is a strict $\ZZ/2$-action on the latter which exchanges the values of the positively and negatively oriented points. In short the $O(1)$-action on  $\sK (\cC^{fd})$ is given by sending an object to its dual.  

\section{Exercises} \label{sec:exercises_2}

\begin{exercise}
	Let $J$ be the {\em free walking isomorphism}, i.e. the groupoid with exactly two isomorphic objects and no non-trivial automorphisms. What familiar space is the CW-complex $|NJ|$?
\end{exercise}

\begin{exercise} \label{ex:verify_pres_homotopy_quotient}
	Using the presentation of $\Bord_1^{or}$, verify that $(\Bord_1^{or} \boxtimes \cJ)_{\ZZ/2}$ has the claimed presentation. 
\end{exercise}

\begin{exercise}[``Whitehead's Theorem'' for symmetric monoidal categories]
	Let $F:(\cC, \otimes) \to (\cD, \otimes)$ be a symmetric monoidal functor. Show that $F$ is a symmetric monoidal equivalence precisely if it is fully-faithful and essentially surjective (i.e. it is an equivalence after forgetting about symmetric monoidal structures). 
	\end{exercise}

\begin{exercise} \label{ex:verify_equivalence}
	Use the previous exercise to verify that the functor $(\Bord_1^{or} \boxtimes \cJ)_{\ZZ/2} \to \Bord_1^{un}$ described above is an equivalence of symmetric monoidal categories. 
\end{exercise}

%
%

\specialsection*{Understanding the $O(2)$-action}

We now move to dimension two and consider the cobordism hypothesis in this dimension. The material of this Chapter draws from many sources. Sections \ref{sec:Serre_Auto} and \ref{sec:2-Full-Dual} follow \cite{MR2555928}. Section \ref{sec:SimpConn3Types} is essentially standard material in topology. Section \ref{sec:ApplyWhiteheadToCats} establishes an explicit connection between higher categories and topology. This material first appeared in a preprint of A. Joyal and R. Street \cite{MR1250465} as part of their development of braided monoidal categories. A similar analysis for the symmetric case appears as appendix B.2 of \cite{MR2192936} (c.f. \cite{Sinh_thesis, Sinh_1982} for an even earlier treatment). The connection to higher categorical group actions is established in \cite{DSPS_DTC2}.

\section{The Serre automorphism} \label{sec:Serre_Auto}

To illustrate the extra structure imparted by full dualizability we define the Serre automorphism (see \cite[Rk. 4.2.4]{MR2555928} for a discussion of the naming convention for this automorphism).  The Serre automorphism is an automorphism for each object in a 2-fully dualizable symmetric monoidal $n$-category (c.f. Exercise \ref{Ex:SerreNaturality}).

Let $X \in \cC$ be an object and assume $\cC$ is at least 2-fully dualizable.  By assumption $X$ is dualizable, so let $X^\vee$ denote its dual and $\ev : X \otimes X^\vee \to \mathbf{1}$ the evaluation map. Now $\ev$ is a 1-morphism which is itself dualizable. Let $\ev^R$ denote its right dual (so $\ev^R : \mathbf{1} \to X \otimes X^\vee$).  Letting $\tau$ denote the braiding isomorphism in $(\cC, \otimes)$, then the {\it Serre automorphism} of the object $X$, denoted $S_X$, is given by the composition:
\[
S_X : X \to X \otimes \mathbf{1} \xrightarrow{\mathbf{1}_X \otimes \ev^R} X \otimes X \otimes X^\vee \xrightarrow{\tau \otimes \mathbf{1}_{X^\vee}} X \otimes X \otimes X^\vee \xrightarrow{\mathbf{1}_X \otimes \ev} X \otimes \mathbf{1} \to X.
\]
The Serre automorphism is described by the string diagram in Figure \ref{fig:Serre_auto}.
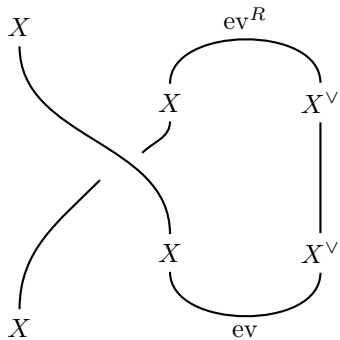
\begin{figure}[htpb]
\begin{center}
\begin{tikzpicture}[thick]
\node (A) at (0,-2) {$X$};
\node (B) at (0,2) {$X$};
\node (C) at (4,-1) {$X^\vee$};
\node (D) at (2,-1) {$X$};
\node (E) at (2,1) {$X$};
\node (F) at (4,1) {$X^\vee$};
\node (G) at (1.5,0.2) {};
\node (H) at (1.2,0.1) {};
\draw (B) to [out = 270, in = 90] (D);
\draw (E) to [out = 270, in =45] (G);
\draw (H) to [out = 225, in =90] (A);
\draw (E) to [out =90,in=90] node [above] {$\ev^R$} (F);
\draw (F) to [out = 270, in =90] (C);
\draw (C) to [out = 270, in = 270] node [below] {$\ev$} (D);
\end{tikzpicture}
\end{center}
\caption{The Serre automorphism.}
\label{fig:Serre_auto}
\end{figure}

Let us explore this in the case $\cC = \mathsf{Bord}^{fr}_2$ is the 2-dimensional tangentially framed bordism category. For the purposes of this discussion there is no harm in viewing $\cC$ as a symmetric monoidal bicategory, rather than a symmetric monoidal $(\infty,2)$-category. The 2-morphisms of this symmetric monoidal bicategory are equivalence classes of 2-dimensional bordisms between 1-dimensional bordisms, equipped with a framing of the 2-dimensional tangent space. 
To make sense of this structure on the lower dimensional bordisms it is best to equip all of our manifolds with a germ of a higher dimensional manifolds surrounding it. Thus the 1-morphisms are equipped with a germ of a surfaces surrounding them, together with a 2-dimensional framing on that germ of a surface. The objects are equipped with a germ of a 1-manifold contained in a 2-manifold, again equipped with a 2-dimensional framing. 

\begin{figure}[htbp]
	\begin{center}
	\begin{tikzpicture}[smooth, tension = 0.25, baseline = 0]
		\fill[fill=black!30] plot  coordinates 
				{(0,1) (0,1.25) (-0.1,1.5) (-1,2.5) (-2.5,1.5) 
				(-2.25,1.25) (-2,1.5) (-3,2.5) (-4,1) (-1,-1) (-0.1,-0.5) (0,-0.25) (0,0)} -- cycle
				plot  coordinates 
				{(1,0) (1, -0.25) (0.9, -0.5) (-1,-2) (-4.5, 1) (-3,3) (-1.5, 1.5) (-2.25, 0.75) (-3, 1.5) (-1,3) (0.9, 1.5) (1, 1.25) (1,1)} -- cycle;
		\fill[fill=black!30] plot coordinates 
				{(0,0) (-0.25, 0) (-0.5,-0.1) (-3,-2) (1,-3) (3,-2) (1.5, -0.1) (1.25, 0) (1,0)} -- cycle
				plot coordinates 
				{(1,1) (1.25, 1) (1.5, 0.9) (4,-2) (1,-4) (-4,-2) (-0.5, 0.9) (-0.25, 1) (0,1)} -- cycle;
		\draw plot  coordinates 
			{(0,1) (0,1.25) (-0.1,1.5) (-1,2.5) (-2.5,1.5) 
			(-2.25,1.25) (-2,1.5) (-3,2.5) (-4,1) (-1,-1) (-0.1,-0.5) (0,-0.25) (0,0)};
		\draw plot  coordinates 
			{(1,1) (1,1.25) (0.9,1.5) (-1,3) (-3,1.5) 
			(-2.25,0.75) (-1.5,1.5) (-3,3) (-4.5,1) (-1,-2) (0.9,-0.5) (1,-0.25) (1,0)};
		\draw plot  coordinates
			{(0,0) (-0.25, 0) (-0.5,-0.1) (-3,-2) (1,-3) (3,-2) (1.5, -0.1) (1.25, 0) (1,0)};
		\draw plot coordinates
			{(1,1) (1.25, 1) (1.5, 0.9) (4,-2) (1,-4) (-4,-2) (-0.5, 0.9) (-0.25, 1) (0,1)};
	\end{tikzpicture}
	\begin{tikzpicture}[domain=0:2, thick, baseline = 1.5, samples=50]
		\draw plot ({\x + 0.5*sin(3.14159*\x r)}, {3-0.25*cos(3.14159*\x r)});
		\draw plot (\x, 2);
		\draw plot ({\x + 0.5*sin(3.14159*\x r)}, {1+0.25*cos(3.14159*\x r)});
		\draw plot ({\x + 0.5*sin(6.28318*\x r)}, {0.25*cos(6.28318*\x r)});
	\end{tikzpicture}
	\end{center}
	\caption[An immersion of the punctured torus]{An immersion of the punctured torus where the blackboard framing induces an interesting tangential framing. To the right several immersed arcs are depicted with distinct 2-dimensional blackboard framings, rel. boundary.}
	\label{fig:immersed_torus}
\end{figure}
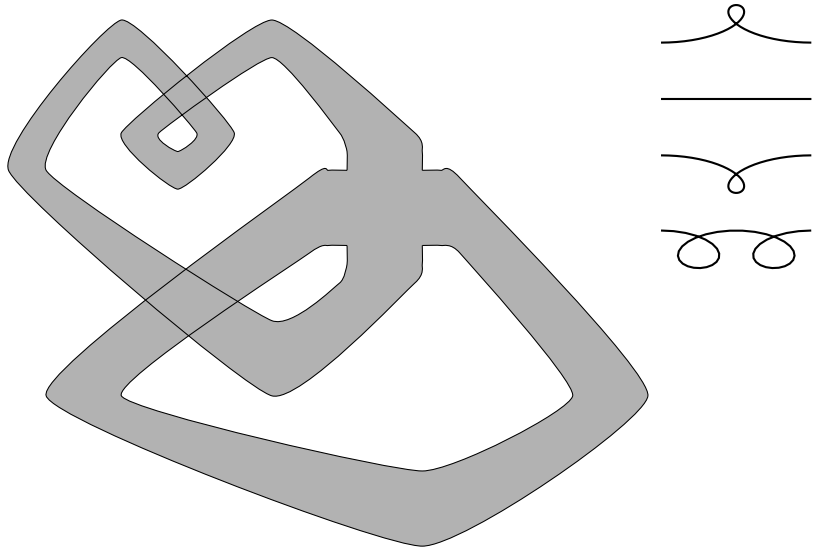

We may obtain a rich supply of easily visualized 2-framed manifolds by using embeddings, and more generally immersions, into the plane. The plane has a standard `blackboard' framing and so any surface immersed into the plane inherits this tangential framing. Of course not every surface immerses into the plane, but it is a consequence of Hirsch-Smale immersion theory that every tangential framing on a connected surface with non-empty boundary may be realized up to isotopy as the blackboard framing induced from an immersion into the plane. Hence many interesting examples arise this way.

Figure \ref{fig:immersed_torus} depicts an immersed punctured torus with an interesting induced tangential framing, as well as several immersed arcs. The isotopy classes of framings on an arc, relative to a fixed framing on the boundary are either empty (if there is no framing on the arc compatible with the framing on the boundary) or a torsor for $\pi_0 \Omega GL_n(\RR) \cong \ZZ$. Under this identification, the immersed arcs in Figure \ref{fig:immersed_torus} differ by consecutive integers.  

In addition to a tangential framing, every bordism has a decomposition of its boundary into incoming and outgoing segments. This decomposition induces, and is equivalent to, a {\em co-orientation} of each boundary segment, i.e., an orientation of its normal bundle. Specifically we will make the convention that the co-orientation for the incoming boundary segments agrees with the inward pointing normal vector, while the co-orientation for the outward boundary segment agrees with the outward pointing normal vector. 
\begin{figure}[htbp]
	\begin{center}
		\begin{tikzpicture}[align=center, decoration={markings,
			mark=between positions 0 and 1 step 2pt with 
				{ \draw [help lines] (0,0) -- (0,0.1); },
				mark=at position -0.1pt with 
				{ \draw [help lines] (0,0) -- (0,0.1); }}]
			
			\fill [fill=black] (1, 4) circle (2pt);
			\draw [dotted] (1, 4) circle (0.5);
			\draw [postaction={decorate}] (1,4) -- (1,3.5);
			\draw [thick, ->] (1,4) -- (1, 3.5);
			\node at (1,3) {positive\\ point};
			
			\fill [fill=black] (3, 4) circle (2pt);
			\draw [dotted] (3, 4) circle (0.5);
			\draw [postaction={decorate}] (3,3.5) -- (3,4);
			\draw [thick, ->] (3,4) -- (3, 3.5);
			\node at (3,3) {negative\\ point};
			
			\draw [postaction={decorate}] (5,4) arc (180:360:0.75cm);
			\draw [->] (5,4) arc (180:195:0.75cm);
			\draw [->] (6.5,4) arc (0:-15: 0.75cm);
			\node at (5.75, 2.5) {$ev: pt_+ \sqcup pt_- \to \emptyset$};
			
			\draw [postaction={decorate}] (7.5,3.5) arc (180:0:0.75cm);
			\draw [->] (7.5,3.5) arc (180:195:0.75cm);
			\draw [->] (9,3.5) arc (0:-15: 0.75cm);
			\node at (8.25, 3) {$coev$};
			
			\draw [postaction={decorate}] (1.5,0.5) arc (0:180:0.75cm);
			\draw [->] (0,0.5) arc (180:195:0.75cm);
			\draw [->] (1.5,0.5) arc (0:-15:0.75cm);
			\node at (0.75,0) {$ev^R$};
			
			\filldraw [postaction={decorate}, fill=black!20] (3.5,0.5) circle (0.75cm);
			\node at (3.5, -0.5) {$\varepsilon: ev \circ ev^R \to id_{\emptyset}$};
			
			\fill [fill=black!20] 
				(5.5 , 1.5) -- (5.5, -0.5) arc (180:0:0.75cm)
				-- (7, 1.5) arc (360:180:0.75cm);
			\draw [postaction={decorate}] 
				(5.5 , 1.5) -- (5.5, -0.5)
				(7,-0.5) -- (7,1.5)
				(5.5, 1.5) arc (180:360:0.75cm)
				(7,-0.5) arc (0:180:0.75cm);
			\draw [->] (5.5, 1.5) -- +(0, -5pt);
			\draw [->] (5.5, -0.5) -- +(0, -5pt);
			\draw [->] (7, 1.5) -- +(0, -5pt);
			\draw [->] (7, -0.5) -- +(0, -5pt);
			\node at (6.25, -1) {$\eta: id_{pt_+ \sqcup pt_-} \to ev^R \circ ev$};
			
			\draw [postaction={decorate}] 
				(8.25,1.5) -- (8.25, 1) to [out = -90, in = 90] (8.75, 0.25) arc (180:360:0.5cm)
				 -- (9.75, 0.75) arc (0:180:0.5cm) to [out = -90, in = 90] (8.25, 0) -- (8.25, -0.5);
			\draw [->] (8.25, 1.5) -- +(0, -5pt);
			\draw [->] (8.25, -0.5) -- +(0, -5pt);
			
			\node (S) at (4, -2) {Serre automorphism of the positive point};
			\draw [->] (S) -| (9, -0.75);
		\end{tikzpicture}
	\end{center}
	\caption{Some tangentially framed bordisms.}
	\label{fig:framed-2-bordisms}
\end{figure}
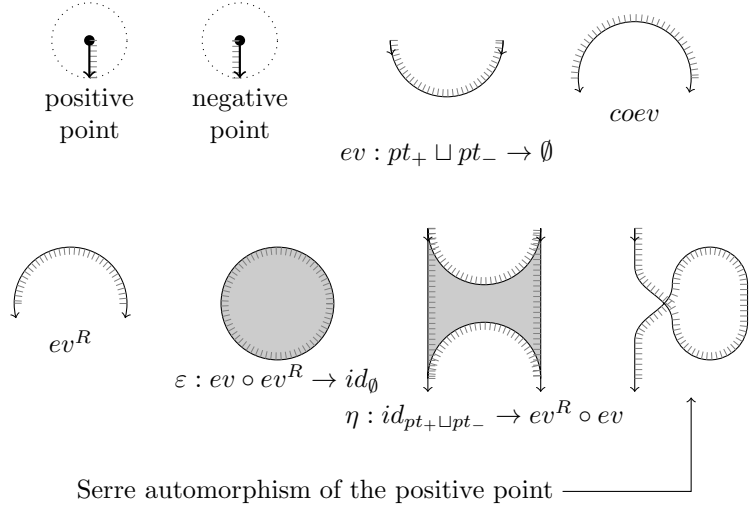

This convention is very useful in illustrating examples of 2-framed bordisms, and several are shown in Figure~\ref{fig:framed-2-bordisms}. In particular the evaluation pairing between the positively framed point and the negatively framed point, as well as its right adjoint are depicted. This allows us to calculate the Serre automorphism of the positive point, which is also depicted in Figure~\ref{fig:framed-2-bordisms}. We find that the framing of the Serre automorphim (a framing on the arc) differs from the framing on the identity morphism of the positive point by a unit in $\pi_0 \Omega GL_2(\RR) \cong \pi_1 GL_2(\RR) \cong \ZZ$. In particular it is a non-trivial automorphism.

\section{2-full dualizability and the action of $O(2)$} \label{sec:2-Full-Dual}

Let $(\cC, \otimes)$ be a symmetric monoidal 2-category which is 2-fully dualizable.  Recall that the core of $\cC$ which we continue to denote by $\sK (\cC^{fd})$ is a 2-groupoid which, via the cobordism hypothesis,  carries an action of $O(2)$ induced by the identification
\[
\sK(\cC^{fd}) \simeq \mathrm{Fun}^\otimes (\mathsf{Bord}_2^{fr} , \cC).
\]
Now we have the splitting $O(2) = SO(2) \rtimes \ZZ/2$.  We analyzed the action of $\ZZ/2$ on categories in previous sections, so here we will focus on the action of $SO(2)$.

What does it mean for $SO(2)$ to act on $\sK(\cC^{fd})$? Loosely, it means that we have something like a group homomorphism
\[
SO(2) \to \mathrm{Aut} \left ( \sK (\cC^{fd}) \right ).
\]
However, the right hand side is a higher categorical object, so we must use more than just the group structure on $SO(2)$; we also use its topology!  More specifically an $SO(2)$ action includes the following (a priori infinite) list of assignments:
\[
\begin{array}{rcl}
\text{a point } x \in SO(2)&\mapsto &\text{a functor } \left (\sK(\cC^{fd}) \to \sK (\cC^{fd}) \right )\\[2ex]
\text{a path } \gamma \subset SO(2)& \mapsto & \text{a natural isomorphism}\\[2ex]
\text{a path of paths} & \mapsto & \text{higher natural isomorphisms}\\
\vdots \hspace{36pt} && \hspace{48pt} \vdots
\end{array}
\]
\noindent These must also have some sort of respect for the the group structure. In particular as $SO(2)$ is connected, the action map 
\[
SO(2) \to \mathrm{Aut} \left ( \sK (\cC^{fd}) \right )
\]
must land in the identity component, which we denote $\mathrm{Aut}^0 \left ( \sK (\cC^{fd}) \right )$.  Hence, we should consider the action of $SO(2)$ as a map of spaces 
\[
SO(2) \to \mathrm{Aut}^0 \left (\sK (\cC^{fd}) \right ).
\]
with some additional structure and properties showing compatibility with the group structure. The map will at least be well-defined up to homotopy. 
 
Now since $\sK(\cC^{fd})$ is a 2-groupoid, $\mathrm{Aut}^0 \left (\sK (\cC^{fd}) \right )$ is a homotopy 2-type (more on this below) and we know that $SO(2)$ has only one interesting homotopy group $\pi_1$ which is $\ZZ$.  So the generator of $\pi_1 SO(2)$ gets sent to a loop at the identity of $\mathrm{Aut}^0 \left (\sK (\cC^{fd}) \right )$.

In the case of the 2-dimensional framed bordism category, there is an $SO(2)$-action given by change of framing. The loop corresponding to $\pi_1(SO(2))$ gives us an invertible 1-dimensional bordism from the point to itself, which, as we just saw in the last section, corresponds to the Serre automorphism of the positive point in the bordism category.  The cobordism hypothesis then tells us that for the $SO(2)$ action on $\sK (\cC^{fd})$, the generator of $\pi_1 SO(2)$ also gets sent to the Serre automorphism:
\[
S: \Id_{\sK (\cC^{fd})} \to \Id_{\sK (\cC^{fd})} .
\]

This is not the complete story, however. The map to $\mathrm{Aut}^0 \left (\sK (\cC^{fd}) \right )$ is not determined by just what it does on homotopy groups. Moreover, we also need to ask that the map behaves like a group homomorphism, that it is compatible with composition.  We address these issues by passing to classifying spaces.

\section{Reducing to the study of simply connected 3-types} \label{sec:SimpConn3Types}

A {\em homotopical action} of $SO(2)$ on $\sK (\cC^{fd})$ may equivalently be described as a pointed map of classifying spaces:
\[
BSO(2) \to B\mathrm{Aut}^0 \left (\sK (\cC^{fd}) \right ).
\]

By assumption $\cC$ is a 2-category, so $\sK (\cC^{fd})$ is a 2-groupoid.  By the homotopy hypothesis, $\sK(\cC^{fd})$ is a homotopy 2-type.  It then follows that $\mathrm{Aut}^0 \left (\sK (\cC^{fd}) \right )$ is also a 2-type and hence $B \mathrm{Aut}^0 \left (\sK (\cC^{fd}) \right )$ is a 3-type. Moreover, as a 3-groupoid $B\mathrm{Aut}^0 \left (\sK (\cC^{fd}) \right )$  has layers
\[
B \mathrm{Aut}^0 \left (\sK (\cC^{fd}) \right ) = \left ( \begin{array}{c} \pi_3 \\ \pi_2 \\ \{\Id_{\sK(\cC^{fd})} \} = \pi_1 \\ \{pt\} = \pi_0 \end{array} \right ).
\]
We see immediately that $B\mathrm{Aut}^0 \left (\sK (\cC^{fd}) \right )$ is actually a simply connected 3-type.

Recall that $BSO(2) \simeq \CC P^\infty$ and that we have a filtration
\[
S^2 \simeq \CC P^1 \subset \CC P^2 \subset \dotsb \subset \CC P^n \subset \dotsb \subset \CC P^\infty \simeq BSO(2).
\]
A (pointed) map $S^2 \to Y$ from the first stage of this filtration is given by an element $\pi_2Y$. 

So far this is exactly what we have constructed. We have a map 
\begin{equation*}
	S^2 \to B\mathrm{Aut}^0 \left (\sK (\cC^{fd}) \right )
\end{equation*}
 given by the Serre automorphism (as $\pi_2 B\mathrm{Aut}^0 \left (\sK (\cC^{fd}) \right )\cong \pi_1 \mathrm{Aut}^0 \left (\sK (\cC^{fd}) \right )$).  We are interested in lifting this to a map \begin{equation*}
 	BSO(2) \to B\mathrm{Aut}^0 \left (\sK (\cC^{fd}) \right )
 \end{equation*} 
 which is a non-trivial problem as $S^2$ and $\CC P^\infty$ do not have the same homotopy 3-type.  Note that we need only lift our map $S^2 \to B\mathrm{Aut}^0 \left (\sK (\cC^{fd}) \right )$ to $\CC P^2$ as $\CC P^2 , \CC P^3 ,\dotsc , \CC P^\infty$ do have the same 3-type.  We summarize this in the following diagram.

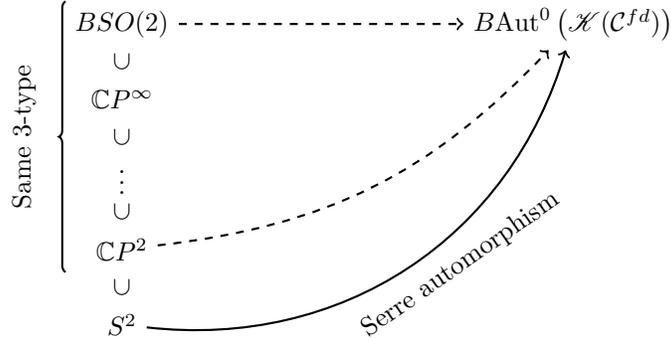
\begin{center}
\begin{tikzpicture}[thick]
\node (A) at (0,0) {$BSO(2)$};
\node at (0,-1) {$\CC P^\infty$};
\node at (0,-2) {$\vdots$};
\node (B) at (0,-3) {$\CC P^2$};
\node (C) at (0,-4) {$S^2$};
\node (D) at (6,0) {$B\mathrm{Aut}^0 \left (\sK (\cC^{fd}) \right )$};
\node[rotate=90] at (0,-.5) {$\subset$};
\node[rotate=90] at (0,-1.5) {$\subset$};
\node[rotate=90] at (0,-2.5) {$\subset$};
\node[rotate=90] at (0,-3.5) {$\subset$};

\draw[->,dashed] (B) to [bend right =20] (D);
\draw [->] (C) to [bend right = 40]  (D);
\draw [->,dashed] (A) to (D);

\draw [
    decoration={
        brace,mirror,
        raise=.75cm
    },
    decorate
] (0,.3) -- (0,-3.3);

\node[rotate=90] at (-1.3,-1.5) {Same $3$-type};
\node[rotate=35] at (4.5,-3.2) {Serre automorphism};

\end{tikzpicture}
\end{center}

\subsection{Whitehead's certain exact sequence and the $\Gamma$-functor}

Let $X$ be a pointed, simply connected, homotopy 3-type.  Following Whitehead, we will construct a complete invariant of such spaces.

Given a space $X$ as above, we form the infinite symmetric product $\Sym^\infty X$ and consider the fibration
\[
F=\mathrm{hofib} (i) \to X \xrightarrow{i} \Sym^\infty X.
\]
It is a theorem of Dold and Thom \cite{MR0097062} that $\pi_k (\Sym^\infty X) = \tilde{H}_k (X)$.  Furthermore, the map $i$ represents the Hurewicz homomorphism. From the Hurewicz Theorem we know that $\pi_2 X \cong H_2 (X)$ and $\pi_3X \to H_3(X)$ is surjective.  Combining with the long exact sequence in homotopy we obtain {\it Whitehead's Certain Exact Sequence}:
\[
0 \to H_4 (X) \to \pi_3 F \xrightarrow{q} \pi_3 X \to H_3 (X) \to 0 .
\]
This short exact sequence is functorial in $X$ and can be regarded as an invariant of the simply-connected 3-type. 

\begin{theorem}[Whitehead \cite{MR0035997}]
The above exact sequence together with the homotopy group $\pi_2$ is a complete invariant of simply connected homotopy 3-types.
\end{theorem}

We now describe how we can calculate the group $\pi_3 F$; we describe an endofunctor $\Gamma$ of abelian groups such that $\pi_3 F = \Gamma (\pi_2 X)$, which for simplicity we simply denote $\Gamma_X$.

\begin{definition}
Let $A$ be an abelian group.  The abelian group $\Gamma (A)$ is the abelian group generated by symbols $\gamma (a)$ for $a \in A$ subject to the relations:
\begin{enumerate}
\item $\gamma (a) = \gamma (-a)$;
\item $\gamma (a) + \gamma (b) + \gamma (c) + \gamma (a+b+c) = \gamma (a+b) + \gamma (b+c) + \gamma (c+a)$.
\end{enumerate}
\end{definition}

The map $A \to \Gamma (A)$ sending $a$ to $\gamma(a)$ is the {\it universal quadratic map}. That is we have a bijection of sets
\[
\mathrm{Hom}_{Ab} (\Gamma(A),B) \leftrightarrow \left\{ 
\text{ quadratic maps } f: A \to B
\right\}.
\]
Here a quadratic map $f:A \to B$ is a map satisfying: 
\begin{enumerate}
\item $f (a) = f (-a)$;
\item $f (a) + f (b) + f (c) + f (a+b+c) = f (a+b) + f (b+c) + f (c+a)$.
\end{enumerate}
One may check that for such maps $f(na) = n^2 f(a)$, and that \begin{equation*}
	B(a_1,a_2)\overset{def}{=}f(a_1 +a_2) - f(a_1) -f(a_2)
\end{equation*} is a symmetric bilinear map. If $2$ is invertible in $A$ and $B$, then such quadratic maps are equivalent to symmetric bilinear maps. 

\begin{theorem}[Whitehead \cite{MR0035997}]
	 $\Gamma_X := \Gamma(\pi_2 X) \cong \pi_3 F$.
\end{theorem}

So far we have described the groups in Whitehead's Certain Exact Sequence; what remains is to describe the map $q$ explicitly in terms of $\pi_2 X$.  To accomplish this we use the Postnikov tower of $X$ (for an introduction to Postnikov towers see for instance \cite[Sect. 4.3]{MR1867354}).  By assumption $X$ is a simply connected 3-type, so its Postnikov tower is given by
\[
\xymatrix{K(\pi_3 X , 3) \ar[r] & X \ar[d] \\ &K(\pi_2 X , 2) \ar[r]^{k_2} & K(\pi_3 X , 4)}
\]
Now by the Hurewicz theorem the map $\pi_3 K(\pi_2 X ,2) \to H_3 (K(\pi_2 X,2))$ is surjective, so $H_3 (K(\pi_2 X  ,2)) =0$.  Hence, by the universal coefficients theorem we have that
\begin{align*}
k_2 \in \mathrm{Map} (K(\pi_2 X , 2) , K(\pi_3 X, 4)) &\cong H^4 (K(\pi_2 X , 2) ; \pi_3)\\
& \cong \mathrm{Hom} (H_4 (K(\pi_2 X , 2)) , \pi_3 X).
	\end{align*}

From a further appliction of the Hurewicz and universal coefficients theorems we have an isomorphishm of abelian groups
\[
H_4(X) \cong H_4 (K(\pi_2 X ,2)),
\]
which by Whitehead's theorem is $\Gamma_X = \Gamma(\pi_2X)$.
The following proposition then relates the quadratic map $q$ and the $k$-invariant $k_2$.

\begin{prop}
Let $X$ be a simply connected 3-type.  Then we have
\[
q=k_2 : \Gamma_X \to \pi_3 X ,
\]
where $q$ is the quadratic map from Whitehead's Certain Exact Sequence and $k_2$ is the second $k$-invariant in the Postnikov tower for $X$. The corresponding quadratic map $\pi_2X \to \pi_3 X$ is given by pre-composition with the Hopf map $S^3 \to S^2$. 
\end{prop}



While both $\CC P^2$ and $S^2$ are simply connected, they are not 3-types. However, for our current purposes we disregard the higher homotopical information and record their homotopical 3-types (we could be more pedantic and truncate them at this stage).

\begin{center}
\begin{tabular}{c|c|c}

&$\CC P^2$& $S^2$  \\
\hline
$\pi_2$& $\ZZ$ & $\ZZ$\\
\hline
$\pi_3$ & 0 & $\ZZ$\\
\hline
$q$ &0&$n \mapsto n^2$
\end{tabular}
\end{center}

Equivalently, the attaching map of the 4-cell in $\CC P^2$ is precisely the Hopf map $q(s) \in \pi_3(S^2)$, where $s \in \pi_2S^2$ is the generator. From these considerations we have the following observation:

\begin{prop}
	Let $X$ be a simply connected 3-type. Then homotopy classes of maps $\CC P^2 \to X$ are naturally in bijection with an elements $s \in \pi_2(X)$ such that $q(s) = 0 \in \pi_3X$. 
\end{prop}

\section{Applying Whitehead's construction to higher categories} \label{sec:ApplyWhiteheadToCats}

Via the homotopy and stablization hypotheses we translate our work on simply connected 3-types to the setting of braided 2-groups and then to the case of $B\mathrm{Aut}^0 \left (\sK (\cC^{fd})\right)$. 
The fundamental $n$-groupoid assigns a pointed, simply connected, 3-groupoid $\Pi_{\le3} X$ to a pointed, simply connected 3-type $X$ (recall that the homotopy hypothesis is that this association is an equivalence).  The stablization hypothesis implies that $\Pi_{\le 3} X$ is then equivalent to a braided monoidal category $\cB$ where all the objects and morphisms are invertible; we will call such categories {\it braided 2-groups}.

The question is how to use our discussion of Whitehead's Certain Exact Sequence to determine our braided 2-group $(\cB,\otimes)$ up to equivalence. More precisely, what are the groups $\pi_2 \cB$ and $\pi_3 \cB$ and what is the quadratic map $q: \pi_2 \cB \to \pi_3 \cB$? This was first solved by A. Joyal and R. Street \cite{MR1250465}. 

Note that by universality, a quadratic map $q: \pi_2 \cB \to \pi_3 \cB$ is the same as a homomorphism $\Gamma (\pi_2 \cB) \to \pi_3 \cB$.  Let $\mathbf{1}$ denote the monoidal unit in $\cB$, then we define our homotopy groups as follows
\[
\begin{array}{l}
\pi_2 \cB = \text{ isomorphism classes of objects of } \cB;\\[2ex]
\pi_3 \cB = \mathrm{Aut} (\mathbf{1}).
\end{array}
\]
We define the map $q: \pi_2 \cB \to \pi_3 \cB$ on objects of $\cB$ and leave it as an exercise to verify that it is well defined and only depends on the isomorphism class of the object (Exercise \ref{ex:q_well_defined}).
Let $b \in \cB$ be an object and $\overline{b} \in \cB$ its $\otimes$-inverse. This is also its (left) dual, so we have isomorphisms
\[
\ev : \overline{b} \otimes b \xrightarrow{\cong} \mathbf{1}  \text{\hspace{16pt} and \hspace{16pt}} \coev: \mathbf{1} \xrightarrow{\cong} b \otimes \overline{b}.
\]
Then define $q : \pi_2 \cB \to \pi_3 \cB$ by
\[
q \overset{def}{=} \ev \circ \tau \circ \coev,
\]
where $\tau$ is the braiding isomorphism. We can visualize the map $q$ evaluated on an object $b \in \cB$ as follows:

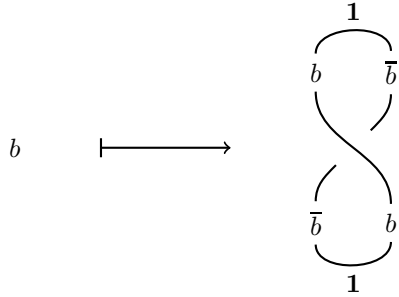
\begin{center}
\begin{tikzpicture}[thick]

\node (A) at (0,0) {$b$};
\node (B) at (1,0) {$\overline{b}$};
\node (C) at (0,-2) {$\overline{b}$};
\node (D) at (1,-2) {$b$};
\node (E) at (0.6,-.9) {};
\node (F) at (0.4,-1.1) {};

\node at (-4,-1) {$b$};
\node (G) at (-3,-1) {};
\node (H) at (-1,-1) {};

\draw (A) to [out =90,in=90] node [above] {$\mathbf{1}$} (B);
\draw (C) to [out=-90,in=-90] node [below] {$\mathbf{1}$} (D);
\draw (A) to [out=-90,in=90] (D);
\draw (B) to [out=-90,in=45] (E);
\draw (F) to [out=225,in=90] (C);
\draw[|->] (G) to (H);

\end{tikzpicture}
\end{center}

\begin{prop}
The map $q: \pi_2 \cB \to \pi_3 \cB$ is quadratic.
\end{prop}

Applying our discussion to the simply connected 3-type $B\mathrm{Aut}^0 \left (\sK (\cC^{fd})\right)$ we have
\begin{align*}
	\pi_2 &= \text{ Natural automorphisms of } \Id_{\sK(\cC^{fd})}, \; \textrm{ and } \\
	\pi_3 &= \text{ Natural automorphisms of } \Id_{\Id_{\sK(\cC^{fd})}}.
\end{align*}
The Serre automorphism is a natural automorphism of the identity functor on $\sK(\cC^{fd})$ (see Exercise \ref{Ex:SerreNaturality}), so for a morphism $f: x \to y$ we have an induced 2-isomorphism
\[
S_f : f \circ S_x \Rightarrow S_y \circ f
\]
Witnessing the commutativity of the following square:
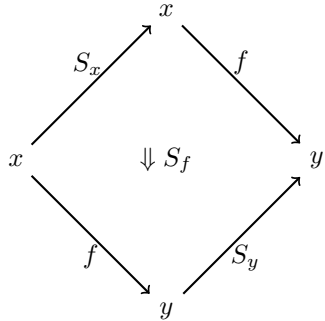
\begin{center}
\begin{tikzpicture}[thick]
\node (A) at (0,0) {$x$};
\node (B) at (2,2) {$x$};
\node (C) at (4,0) {$y$};
\node (D) at (2,-2) {$y$};
\draw[->] (A) to node [above] {$S_x \;$} (B);
\draw[->] (B) to node [above] {$f$} (C);
\draw[->] (A) to node [below] {$f$} (D);
\draw[->] (D) to node [below] {$\; S_y$} (C);
\node at (2,0) {$\Downarrow S_f$};
\end{tikzpicture}
\end{center}
This 2-isomorphism implements the {\em naturality} of the Serre automorphism. 

Let $S \in \pi_2 B\mathrm{Aut}^0 \left (\sK (\cC^{fd})\right)$ be the Serre automorphism. We need to compute $q(S)$. Let $x \in \sK(\cC^{fd})$ be an object and $S_x$ the associated Serre automorphism. Note that $S_x$ is itself invertible and we denote its inverse by $S_x^{-1}$. We have that $q(S_x)$ is given by the assignment depicted in Figure \ref{fig:q_of_serre}. The evaluation and coevaluation maps in this diagram make $S_x^{-1}$ into the left adjoint of $S_x$; they form an adjoint equivalence. 

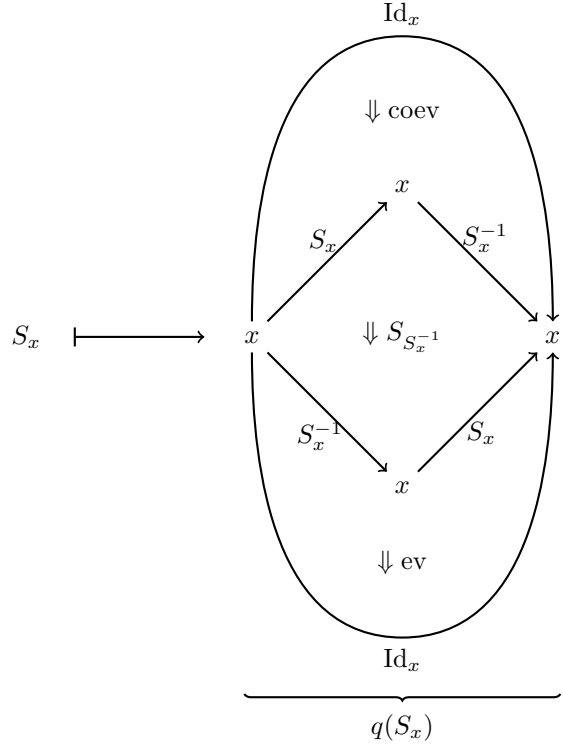
\begin{figure}[htbp]
	\begin{center}
	\begin{tikzpicture}[thick]

	\node (A) at (0,0) {$x$};
	\node (B) at (2,2) {$x$};
	\node (C) at (4,0) {$x$};
	\node (D) at (2,-2) {$x$};

	\draw[->] (A) to node [above] {$S_x \;$} (B);
	\draw[->] (B) to node [above] {$\; \; S_x^{-1}$} (C);
	\draw[->] (A) to node [below] {$S_x^{-1} \; \;$} (D);
	\draw[->] (D) to node [below] {$\;S_x$} (C);
	\node at (2,0) {$\Downarrow S_{S_x^{-1}}$};

	\draw[->] (A) to [out=90,in=180] (2,4) to [out=0,in=90] (C);
	\draw[->] (A) to [out=-90,in=180] (2,-4) to [out=0,in=-90] (C);

	\node at (2,4.3) {$\Id_x$};
	\node at (2,-4.3){$\Id_x$};
	\node at (2,3) {$\Downarrow \coev$};
	\node at (2,-3) {$\Downarrow \ev$};

	\node at (-3,0) {$S_x$};
	\node (F) at (-2.5,0) {};
	\node (G) at (-.5,0) {};
	\draw[|->] (F) to (G);

	\draw [
	    decoration={
	        brace,mirror,
	        raise=.75cm
	    },
	    decorate
	] (-.1,-4) -- (4.1,-4);

	\node at (2,-5.2) {$q(S_x)$};

	\end{tikzpicture}
	\end{center}
	\caption{The Serre automorphism, precomposed with the Hopf map.}
	\label{fig:q_of_serre}
\end{figure}

We computed the 3-types of $\CC P^2$ (which is the same as $\CC P^\infty$) and $S^2$ in the previous section.  It follows that in order to lift our map given by the Serre automorphism $S^2 \to B\mathrm{Aut}^0 \left (\sK (\cC^{fd})\right)$ we must have that $q(S)$ is the identity.

\begin{prop}[\cite{DSPS_DTC2} and \cite{MR2555928}]
Let $S \in \pi_2 B\mathrm{Aut}^0 \left (\sK (\cC^{fd})\right)$ be the Serre automorphism, then $q(S) \in \pi_3 B\mathrm{Aut}^0 \left (\sK (\cC^{fd})\right)$ is the identity.
\end{prop}


\subsection{Conclusion}

Let $(\cC, \otimes)$ be a symmetric monoidal n-category which is 2-fully dualizable.  Then the action of $SO(2)$ on $\sK(\cC^{fd})$ is given by the Serre automorphism
\[
S: \Id_{\sK(\cC^{fd})} \to \Id_{\sK(\cC^{fd})}
\]
subject to the condition that the quadratic map $q$ with 
\[
q(S) : \Id_{\Id_{\sK(\cC^{fd})}} \to \Id_{\Id_{\sK(\cC^{fd})}}
\]
satisfies the identity:
\[
q(S) = \Id_{\Id_{\Id_{\sK(\cC^{fd})}}}.
\]

\section{Exercises} \label{sec:exercises_3}

\begin{exercise}[$\star$] \label{Ex:SerreNaturality}
	To what extent does the Serre automorphism depend on the choice of dualizablity data? Is the Serre automorphism a natural transformation?
\end{exercise}

\begin{exercise}
	Show that if you replace $ev^R$ be the left adjoint $ev^L$ in the formula for the Serre automorphism $S$, then you obtain its inverse $S^{-1}$. 
\end{exercise}

\begin{exercise}
	Let $(\cC, \otimes)$ be the monoidal category of $\ZZ/2$-graded complex vector spaces. Show that, up to equivalence, there are four braided monoidal structures on $(\cC, \otimes)$ which distribute over direct sums. Which of these are symmetric monoidal?
\end{exercise}

\begin{exercise}
	If $X$ is a simply connected 3-type, relate Whitehead's Certain Exact Sequence of $X$ to the Serre Spectral sequence of the Postnikov Tower of $X$. 
\end{exercise}

\begin{exercise}
	Compute $\Gamma(\ZZ/n) \cong \ZZ/n$ for odd $n$, and $\Gamma(\ZZ/2^i) \cong \ZZ/2^{2i}$. 
\end{exercise}

\begin{exercise} \label{ex:q_well_defined}
	Recall the map $q$ defined in lecture, which maps objects of a braided 2-group to automorphisms of the unit object. Use Exercise \ref{Ex:uniquenessofduals} to show that $q$ is well defined and in fact only depends on the isomorphism class of the object. 
\end{exercise}

%
%
%

\specialsection*{Understanding the $O(3)$-action}

\section{3-full dualizability and the action of $O(3)$} \label{sec:3_dual}

Let $(\cC, \otimes)$ be a 3-fully dualizable symmetric monoidal 3-category.  The cobordism hypothesis implies that the core (the maximal 3-groupoid) $\sK (\cC^{fd})$ carries an action of $O(3) = SO(3) \rtimes \ZZ/2$.  In this section we unwind the data of the $SO(3)$ action and as an application we recover a result of Etingof, Nikshych, and Ostrik on fusion categories.

As 3-full dualizability implies 2-full dualizability, the $SO(3)$ action induces an $SO(2)$ action on $\sK(\cC^{fd})$.  The action of $SO(2)$ has the special property that the Serre automorphism is of order 2.  Let's recall the first three homotopy groups of $SO(2)$ and $SO(3)$.

\begin{center}
\begin{tabular}{c|c|c}

&$SO(2)$& $SO(3)$  \\
\hline
$\pi_1$& $\ZZ$ & $\ZZ/2$\\
\hline
$\pi_2$ & 0 & $0$\\
\hline
$\pi_3$ &0&$\ZZ$
\end{tabular}
\end{center}

In order to actually prove that the Serre autormorphism is of order 2 we need the following lemma.

\begin{lemma}[{\cite{DSPS_DTC1} c.f. \cite[Rk 3.4.22]{MR2555928}}] \label{3catlemma}
Let $\cC$ be a symmetric monoidal 3-category and $f: x \to y$ a 1-morphism.  Suppose that $f$ admits a right dual, i.e. there is a quadruple (satisfying the duality relations)
\[
(f,f^R , \ev: f \circ f^R \Rightarrow \Id_y , \coev : \Id_x \Rightarrow f^R \circ f).
\]
Further, suppose that $\ev$ and $\coev$ admit left duals.  Then, the quadruple
\[
(f^R , f , \coev^L , \ev^L)
\]
exhibits $f^R$ as a left dual of $f$.
\end{lemma}

\begin{cor}[\cite{DSPS_DTC1}]
In a 3-fully dualizable symmetric monoidal 3-category, we have a canonical natural isomorphism
\[
R: S^2 \cong \Id_{\Id_{\sk(\cC^{fd})}},
\]
where $S$ is the Serre automorphism.
\end{cor}

\begin{proof}
Recall the Serre automorphism and its inverse which are given by the following diagrams.

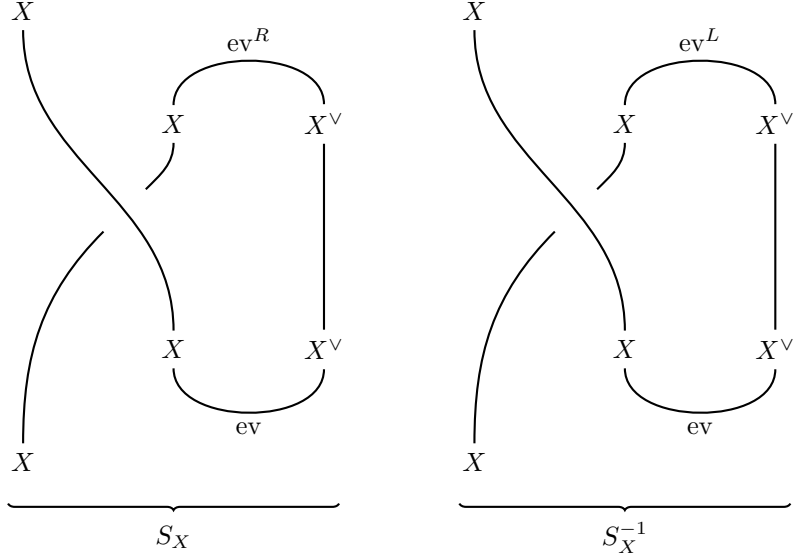
\begin{center}
\begin{tikzpicture}[thick]
\node (A) at (0,-3) {$X$};
\node (B) at (0,3) {$X$};
\node (C) at (4,-1.5) {$X^\vee$};
\node (D) at (2,-1.5) {$X$};
\node (E) at (2,1.5) {$X$};
\node (F) at (4,1.5) {$X^\vee$};
\node (G) at (1.5,.5) {};
\node (H) at (1.2,.2) {};

\draw (B) to [out = 270, in = 90] (D);
\draw (E) to [out = 270, in =45] (G);
\draw (H) to [out = 225, in =90] (A);
\draw (E) to [out =90,in=90] node [above] {$\ev^R$} (F);
\draw (F) to [out = 270, in =90] (C);
\draw (C) to [out = 270, in = 270] node [below] {$\ev$} (D);

\draw [
    decoration={
        brace,mirror,
        raise=.55cm
    },
    decorate
] (-.2,-3) -- (4.2,-3);

\node at (2,-4) {$S_X$};

\node (A1) at (6,-3) {$X$};
\node (B1) at (6,3) {$X$};
\node (C1) at (10,-1.5) {$X^\vee$};
\node (D1) at (8,-1.5) {$X$};
\node (E1) at (8,1.5) {$X$};
\node (F1) at (10,1.5) {$X^\vee$};
\node (G1) at (7.5,.5) {};
\node (H1) at (7.2,.2) {};

\draw (B1) to [out = 270, in = 90] (D1);
\draw (E1) to [out = 270, in =45] (G1);
\draw (H1) to [out = 225, in =90] (A1);
\draw (E1) to [out =90,in=90] node [above] {$\ev^L$} (F1);
\draw (F1) to [out = 270, in =90] (C1);
\draw (C1) to [out = 270, in = 270] node [below] {$\ev$} (D1);

\draw [
    decoration={
        brace,mirror,
        raise=.55cm
    },
    decorate
] (5.8,-3) -- (10.2,-3);

\node at (8,-4) {$S_X^{-1}$};
	
\end{tikzpicture}
\end{center}

Since we are in a 3-fully dualizable symmetric monoidal category, we can apply the lemma to deduce that  $\ev^R \cong \ev^L$ canonically, so we have a natural (in $X$) and canonical isomorphism $S_X \cong S_X^{-1}$.

\end{proof}

The isomorphism $R:S^2 \cong \Id_{\Id_{\sk(\cC^{fd})}}$ is called the {\it Radford isomorphism}. Notice that this Lemma \ref{3catlemma} also implies that in a  3-fully dualizable category, duality of 1-morphisms is ambidexterous, left and right duals canonically agree. This is the first hint that there is something really magical happening when we pass to dimension three and above.

\begin{proof}[Sketch of proof of Lemma \ref{3catlemma}]
We utilize string diagrams to outline the proof.  By assumption the morphism $f$ admits a right dual, so we have

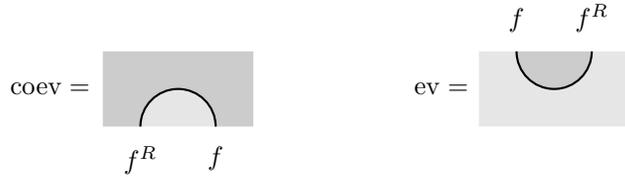
\begin{center}
\begin{tikzpicture}[thick]

\begin{scope}
	\node (LT) at (0, 1) {};
	\node (fR) [label=below:$f^R$] at (.5, 0) {}; 
	\node (f) [label=below:$f$]  at (1.5, 0) {};
	\node (RB) at (2, 0) {};
	\node at (-.7, .5) {$\coev =$};
	
	\draw (fR.center) arc (180: 0: 0.5 cm);
	\begin{pgfonlayer}{background}
		\fill [black!20] (LT.center) -- (LT.center |- RB.center) -- 
		(fR.center) arc (180: 0: 0.5 cm) -- (RB.center) -- (RB.center |- LT.center) -- cycle;
		\fill [black!10]  (fR.center) arc (180: 0: 0.5 cm) -- cycle;
	\end{pgfonlayer}	
\end{scope}

\begin{scope}[xshift = 5cm]
	\node (LT) at (0, 1) {};
	\node (f) [label=above:$f$] at (.5, 1) {}; 
	\node (fR) [label=above:$f^R$]  at (1.5, 1) {};
	\node (RB) at (2, 0) {};
	\node at (-.5, .5) {$\ev =$};
	
	\draw (f.center) arc (180: 360: 0.5 cm);
	\begin{pgfonlayer}{background}
		\fill [black!20] (f.center) arc (180: 360: 0.5 cm) -- cycle;
		\fill [black!10] (LT.center) -- (f.center) arc (180: 360: 0.5 cm) -- (LT.center -| RB.center) -- (RB.center) -| (LT.center);
	\end{pgfonlayer}	
\end{scope}
\end{tikzpicture}
\end{center}

\noindent
The evaluation and coevaluation satisfy the following identities.

\begin{center}
\begin{tikzpicture}[thick]
\begin{scope}
	\node (LB) at (0,0) {};
	\node (f) at (.5, 1) {};
	\node (g) at (2.5, 2) {};
	\node (RT) at (3, 2) {};
	\draw (f.center |- LB.center) -- (f.center) arc (180:0: 0.5 cm) arc (180: 360: 0.5 cm) -- (g.center);

	\node at (3.5, 1) {=};
	
	\fill [black!20] (4, 2) rectangle (5, 0);
	\fill [black!10] (5, 2) rectangle (6, 0);
	\draw (5, 2) -- (5, 0);
	\node at (5,2.5) {$f^R$};
	
	\begin{pgfonlayer}{background}
		\fill [black!20] (LB.center) -- (f.center |- LB.center) -- (f.center) arc (180:0: 0.5 cm) arc (180: 360: 0.5 cm) -- (g.center) -- (g.center -| LB.center) -- cycle;
		\fill [black!10] (f.center |- LB.center) -- (f.center) arc (180:0: 0.5 cm) arc (180: 360: 0.5 cm) -- (g.center) -- (RT.center) -- (RT.center |- LB.center) -- cycle;	
	\end{pgfonlayer}	

\end{scope}

\begin{scope}[yshift = -3cm]
	\node (LT) at (0,2) {};
	\node (f) at (.5, 1) {};
	\node (g) at (2.5, 0) {};
	\node (RB) at (3, 0) {};
	\draw (f.center |- LT.center) -- (f.center) arc (180:360: 0.5 cm) arc (180: 0: 0.5 cm) -- (g.center);

	\node at (3.5, 1) {=};
	
	\fill [black!10] (4, 2) rectangle (5, 0);
	\fill [black!20] (5, 2) rectangle (6, 0);
	\draw (5, 2) -- (5, 0);
	\node at (5,2.5) {$f$};
	
	\begin{pgfonlayer}{background}
		\fill [black!10] (LT.center) -- (f.center |- LT.center) -- (f.center) arc (180:360: 0.5 cm) arc (180: 0: 0.5 cm) -- (g.center) -- (g.center -| LT.center) -- cycle;
		\fill [black!20] (f.center |- LT.center) -- (f.center) arc (180:360: 0.5 cm) arc (180: 0: 0.5 cm) -- (g.center) -- (RB.center) -- (RB.center |- LT.center) -- cycle;	
	\end{pgfonlayer}	
\end{scope}

\end{tikzpicture}

\end{center}

\noindent
Further, by assumption $\ev$ and $\coev$ admit left duals, so we have

\begin{center}
\begin{tikzpicture}[thick]

\begin{scope}
	\node (LT) at (0, 1) {};
	\node (f) [label=below:$f$] at (.5, 0) {}; 
	\node (fR) [label=below:$f^R$]  at (1.5, 0) {};
	\node (RB) at (2, 0) {};
	\node at (-.7, .5) {$\ev^L =$};
	
	\draw (f.center) arc (180: 0: 0.5 cm);
	\begin{pgfonlayer}{background}
		\fill [black!10] (LT.center) -- (LT.center |- RB.center) -- 
		(f.center) arc (180: 0: 0.5 cm) -- (RB.center) -- (RB.center |- LT.center) -- cycle;
		\fill [black!20]  (f.center) arc (180: 0: 0.5 cm) -- cycle;
	\end{pgfonlayer}	
\end{scope}

\begin{scope}[xshift = 5cm]
	\node (LT) at (0, 1) {};
	\node (fR) [label=above:$f^R$] at (.5, 1) {}; 
	\node (f) [label=above:$f$]  at (1.5, 1) {};
	\node (RB) at (2, 0) {};
	\node at (-.85, .5) {$\coev^L =$};
	
	\draw (fR.center) arc (180: 360: 0.5 cm);
	\begin{pgfonlayer}{background}
		\fill [black!10] (fR.center) arc (180: 360: 0.5 cm) -- cycle;
		\fill [black!20] (LT.center) -- (fR.center) arc (180: 360: 0.5 cm) -- (LT.center -| RB.center) -- (RB.center) -| (LT.center);
	\end{pgfonlayer}	
\end{scope}
\end{tikzpicture}
\end{center}

\noindent
By dualizability, we have morphisms $\ev^L \circ \ev \to \Id$ and $\Id \to \ev \circ \ev^L$, that is

\begin{center}
\begin{tikzpicture}[thick]

\begin{scope}
		
	\draw (-1,0) arc (180: 360: 0.5 cm);
	\draw (-1,-3) arc (180:0:0.5 cm);
	\draw (3.5,0) -- (3.5,-3);
	\draw (4.5,0)--(4.5,-3);
	\node at (1.75,-1.5) {$\to$};
	\begin{pgfonlayer}{background}
		\fill [black!10] (-2,0)--(-2,-3)--(1,-3)--(1,0) -- cycle;
		\fill [black!10] (2.5,0)--(2.5,-3)--(5.5,-3)--(5.5,0)--cycle;
		\fill [black!20] (3.5,0)--(3.5,-3)--(4.5,-3)--(4.5,0)--cycle;
		\fill [black!20] (-1,0) arc (180:360:0.5 cm);
		\fill [black!20] (-1,-3) arc (180:0:0.5cm);
	\end{pgfonlayer}	
\end{scope}

\begin{scope}[yshift = -5cm]
	\draw (3.5,0) circle (0.5 cm);
	\begin{pgfonlayer}{background}
		\fill [black!10] (-1,1)--(-1,-1)--(1,-1)--(1,1) -- cycle;	
		\fill [black!10] (2.5,1)--(2.5,-1)--(4.5,-1)--(4.5,1) -- cycle;
		\fill [black!20]  (3.5,0) circle (0.5 cm);
	\end{pgfonlayer}	
	\node at (1.75,0) {$\to$};
\end{scope}
\end{tikzpicture}
\end{center}

\noindent
Similarly, by dualizability of $\coev$ we have

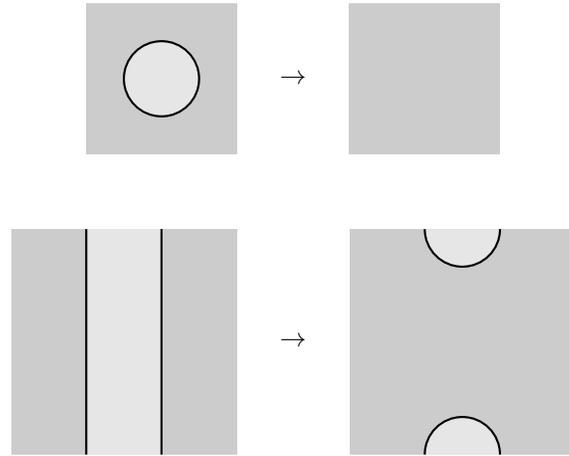
\begin{center}
\begin{tikzpicture}[thick]

\begin{scope}	
	\draw (0,0) circle (0.5 cm);
	\begin{pgfonlayer}{background}
		\fill [black!20] (-1,1)--(-1,-1)--(1,-1)--(1,1) -- cycle;
		\fill [black!10]  (0,0) circle (0.5 cm);
		\fill [black!20] (2.5,1)--(2.5,-1)--(4.5,-1)--(4.5,1) -- cycle;
	\end{pgfonlayer}	
	\node at (1.75,0) {$\to$};
\end{scope}

\begin{scope}[yshift = -2cm]
		
	\draw (3.50,0) arc (180: 360: 0.5 cm);
	\draw (3.5,-3) arc (180:0:0.5 cm);
	\draw (-1,0) -- (-1,-3);
	\draw (0,0)--(0,-3);
	\node at (1.75,-1.5) {$\to$};
	\begin{pgfonlayer}{background}
		\fill [black!20] (-2,0)--(-2,-3)--(1,-3)--(1,0) -- cycle;
		\fill [black!10] (-1,0)--(-1,-3)--(0,-3)--(0,0)--cycle;
		\fill [black!20] (2.5,0)--(2.5,-3)--(5.5,-3)--(5.5,0)--cycle;
		\fill [black!10] (3.5,0) arc (180:360:0.5 cm);
		\fill [black!10] (3.5,-3) arc (180:0:0.5cm);
	\end{pgfonlayer}	
\end{scope}
\end{tikzpicture}
\end{center}

Now in order to for $f^R$ to be left dual to $f$ we need the following diagrams

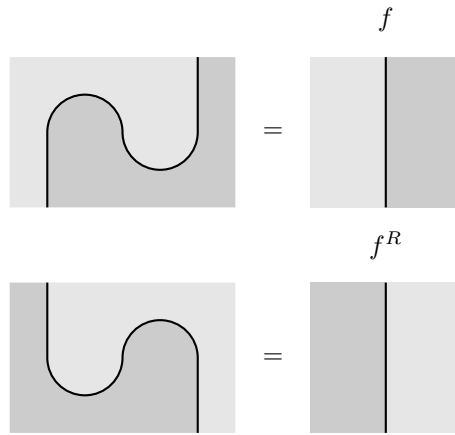
\begin{center}
\begin{tikzpicture}[thick]
\begin{scope}
	\node (LB) at (0,0) {};
	\node (f) at (.5, 1) {};
	\node (g) at (2.5, 2) {};
	\node (RT) at (3, 2) {};
	\draw (f.center |- LB.center) -- (f.center) arc (180:0: 0.5 cm) arc (180: 360: 0.5 cm) -- (g.center);

	\node at (3.5, 1) {=};
	
	\fill [black!10] (4, 2) rectangle (5, 0);
	\fill [black!20] (5, 2) rectangle (6, 0);
	\draw (5, 2) -- (5, 0);
	\node at (5,2.5) {$f$};
	
	\begin{pgfonlayer}{background}
		\fill [black!10] (LB.center) -- (f.center |- LB.center) -- (f.center) arc (180:0: 0.5 cm) arc (180: 360: 0.5 cm) -- (g.center) -- (g.center -| LB.center) -- cycle;
		\fill [black!20] (f.center |- LB.center) -- (f.center) arc (180:0: 0.5 cm) arc (180: 360: 0.5 cm) -- (g.center) -- (RT.center) -- (RT.center |- LB.center) -- cycle;	
	\end{pgfonlayer}	

\end{scope}

\begin{scope}[yshift = -3cm]
	\node (LT) at (0,2) {};
	\node (f) at (.5, 1) {};
	\node (g) at (2.5, 0) {};
	\node (RB) at (3, 0) {};
	\draw (f.center |- LT.center) -- (f.center) arc (180:360: 0.5 cm) arc (180: 0: 0.5 cm) -- (g.center);

	\node at (3.5, 1) {=};
	
	\fill [black!20] (4, 2) rectangle (5, 0);
	\fill [black!10] (5, 2) rectangle (6, 0);
	\draw (5, 2) -- (5, 0);
	\node at (5,2.5) {$f^R$};
	
	\begin{pgfonlayer}{background}
		\fill [black!20] (LT.center) -- (f.center |- LT.center) -- (f.center) arc (180:360: 0.5 cm) arc (180: 0: 0.5 cm) -- (g.center) -- (g.center -| LT.center) -- cycle;
		\fill [black!10] (f.center |- LT.center) -- (f.center) arc (180:360: 0.5 cm) arc (180: 0: 0.5 cm) -- (g.center) -- (RB.center) -- (RB.center |- LT.center) -- cycle;	
	\end{pgfonlayer}	
\end{scope}

\end{tikzpicture}

\end{center}

\noindent
Let us first consider the first identity, as the second follows similarly. First we have a comparison map:
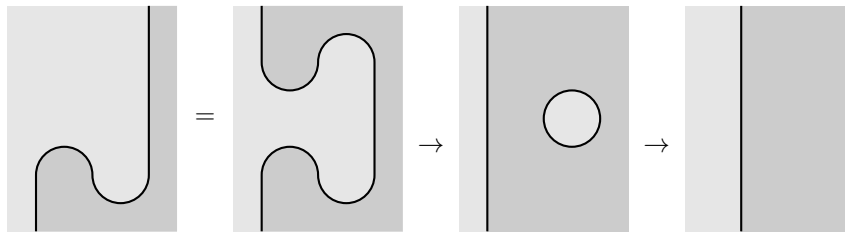
\begin{center}
\begin{tikzpicture}[thick,scale=0.75]
\draw (.5,-4) -- (.5,-3) arc (180:0: 0.5 cm) arc (180: 360: 0.5 cm) -- (2.5,0);
%
		
	\begin{pgfonlayer}{background}
		\fill [black!10] (0,0) --(0,-4)--(3,-4)--(3,0)--cycle;
		\fill [black!20] (.5,-4) -- (.5,-3) arc (180:0: 0.5 cm) arc (180: 360: 0.5 cm) -- (2.5,0)--(3,0)--(3,-4)-- cycle;	
	\end{pgfonlayer}	

\node at (3.5,-2) {$=$};
\draw (4.5,-4) -- (4.5,-3) arc (180:0: 0.5 cm) arc (180: 360: 0.5 cm) -- (6.5,-1) arc (0:180:0.5cm) arc (0:-180:0.5cm)--(4.5,0);

\begin{pgfonlayer}{background}
		\fill [black!10] (4,0) --(4,-4)--(7,-4)--(7,0)--cycle;
		\fill [black!20] (4.5,-4) -- (4.5,-3) arc (180:0: 0.5 cm) arc (180: 360: 0.5 cm) -- (6.5,-1) arc (0:180:0.5cm) arc (0:-180:0.5cm)--(4.5,0) -- (7,0)--(7,-4)--cycle;	
	\end{pgfonlayer}	
	
\draw (8.5,0)--(8.5,-4);
\draw (10,-2) circle (0.5 cm);	

\begin{pgfonlayer}{background}
		\fill [black!10] (8,0) --(8,-4)--(11,-4)--(11,0)--cycle;
		\fill [black!20] (8.5,0)--(8.5,-4)--(11,-4)--(11,0)--cycle;	
		\fill [black!10] (10,-2) circle (0.5cm);
	\end{pgfonlayer}	
	
\node at (11.5,-2.5) {$\to$};
\node at (7.5,-2.5) {$\to$};

\draw (13,0)--(13,-4);

\begin{pgfonlayer}{background}
		\fill [black!10] (12,0) --(12,-4)--(13,-4)--(13,0)--cycle;
		\fill [black!20] (13,0)--(13,-4)--(15,-4)--(15,0)--cycle;	
	\end{pgfonlayer}	

\end{tikzpicture}
\end{center}
Now these arrows are not isomorphisms individually, but the composite is an isomorphism. A map going the other way can be built in nearly the identical fashion. Specifically, take the above diagram, rotate each figure 180 degrees and reverse the colors. This will be a new sequence of operations where the natural map goes in the other direction. In fact this is precisely the inverse of original map. We leave the details as an exercise. 
\end{proof}

\section{The data of an $SO(3)$ action}

Using the notation of the previous few sections, we have the following characterization.

\begin{theorem}[{\cite{DSPS_DTC2}}]
Let $\cC$ be a symmetric monoidal 3-category.  To give an action of $SO(3)$ on $\cC$ is to specify the following data
\begin{itemize}
\item $S: \Id_\cC \xrightarrow{\cong} \Id_\cC$;
\item $\sigma : q(S) \xrightarrow{\cong} \Id_{\Id_{\Id_\cC}}$;
\item $R : S^2 \xrightarrow{\cong} \Id_{\Id_\cC}$;
\end{itemize}
subject to a condition:
\begin{equation*}
	\frac{\eta q}{2}  (R,S) = 0.
\end{equation*}
\end{theorem}
\noindent We will explain how this last condition arises in the proof sketched below.

\begin{proof}[Proof Sketch.]
	The following is only a partial sketch of the proof. The full account, including a higher categorical interpretation of the condition $\frac{\eta q}{2}  (R,S) = 0$, may be found in \cite{DSPS_DTC2}.

Let $B \mathrm{Orp} (3)$ be the homotopy fiber in the fibration
\[
B \mathrm{Orp} (3) \to BSO(3) \xrightarrow{p_1} K(\ZZ,4).
\]
This is the structure group corresponding to an orientation and a {\em $p_1$-structure}, i.e. a lift of the classifying map of the tangent bundle to $B\mathrm{Orp}(3)$ is the same as an orientation and trivialization of $p_1$. On an 3-manifold $p_1$ is always trivializable, but there are different trivializations. Such a lift is also called a `2-framing' by Atiyah \cite{MR1046621}. 

We then have that
\[
\pi_2 B \mathrm{Orp} (3) = \ZZ/2, \quad \pi_3 B\mathrm{Orp} (3) = \ZZ/4 = \Gamma ( \ZZ/2), \text{ and } \pi_4B\mathrm{Orp} (3) = 0.
\]
One may construct a minimal $CW$-structure for $B\mathrm{Orp} (3)$, which begins
\begin{equation*}
	S^2 \cup_{2} e^3 \cup_\phi e^5 \cup \textrm{... higher cells ...}
\end{equation*}
Thus we can study $\mathrm{Orp}(3)$-actions just as we studied $SO(3)$-actions. As before the $S^2$-part of the action consists of giving $S: \Id_\cC \xrightarrow{\cong} \Id_\cC$. Next the effect of the 3-cell is to trivialize $S\circ S$, hence this corresponds to $R : S^2 \xrightarrow{\cong} \Id_{\Id_\cC}$.

This part of the action corresponds to an $\Omega \Sigma \RP^2$-action. The loop space $\Omega \Sigma \RP^2$ is the `free $A_\infty$-group generated by $\RP^2$'. The homotopy groups of $B\Omega \Sigma \RP^2 \simeq \Sigma \RP^2$ are:
\begin{equation*}
	\pi_2 = \ZZ/2, \quad \pi_3 = \ZZ/4, \textrm{ and } \pi_4 = \ZZ/4.
\end{equation*}
The effect of the last cell is to trivialize the generator of $\pi_4$, which is the attaching map $\phi$. In other words given $S$ and $R$, there exists a canonical element that we can construct. It is an automorphism of the identity of the identity of the identity functor, which is always 4-torsion. The construction of this element is analogous to the construction of $q(S)$ from $S$, but is more complicated and uses both $R$ and $S$. We name this element 
$\frac{\eta q}{2}  (R,S) $ for reasons we won't go into here. See \cite{DSPS_DTC2} for details.  

The $\Omega \Sigma \RP^2$-action extends to an $\mathrm{Orp}(3)$-action precisely if the equation:
\begin{equation*}
	\frac{\eta q}{2}  (R,S) = 0
\end{equation*}
holds. The proof of the theorem then follows by establishing that the following  square is a pushout square of 4-types:
\begin{center}
\begin{tikzpicture}[thick]
\node (A) at (0,0) {$S^2$};
\node (B) at (3,0) {$BSO(2)$};
\node (C) at (0,-2) {$B \mathrm{Orp} (3)$};
\node (D) at (3,-2) {$BSO(3)$};
\node (E) at (6,-4) {$B \mathrm{Aut}^0 (\cC)$};

\draw[->] (A) to (B);
\draw[->] (A) to (C);
\draw[->] (B) to (D);
\draw[->] (C) to (D);
\draw[->] (B) to [bend left =20] node [above right] {$(S,\sigma)$}  (E);
\draw[->] (C) to [bend right =20] node [below left] {$(S,R)$} (E);
\draw[->,dashed] (D) to (E);
\node at (2.5, -1.5) {$\lrcorner$};
\end{tikzpicture}
\end{center}

\end{proof}

What about $\pi_3 SO(3)$? This group is non-trivial, how can we see its image?
We will call the image of the generator $a$ (for `anomaly') which is a map
\[
a: \Id_{\Id_{\Id_\cC}} \to \Id_{\Id_{\Id_\cC}}.
\]
Let us apply the quadratic map $q$ to the Radford
\[
q(R) : q(S)^4 = q(S^2)  \xrightarrow{\cong} q(\Id) = \Id_{\Id_\Id}.
\]
Then the map $a$ is given as the following composition
\[
a: \Id_{\Id_\Id} \overset{\sigma^{-4}}{\Rightarrow} q(S)^4 \overset{q(R)}{\Rightarrow} \Id_{\Id_\Id}.
\]

\section{An application to fusion categories}

Here we apply our work above in the setting of fusion categories. 

\begin{definition}\footnote{Actually what we call {\em fusion categories} above are more commonly called {\em multi-fusion categories}. Fusion categories are traditionally required to satisfy the additional requirement that the monoidal unit object is simple.}
A {\it fusion category} is a monoidal $\KK$-linear abelian category $\cF$ satisfying two additional properties:
\begin{enumerate}
\item $\cF$ is semi-simple with a finite number of isomorphism classes of simple objects and finite dimensional hom sets;
\item $\cF$ is rigid, i.e. every object has both left and right duals.
\end{enumerate}
\end{definition}

Fusion categories are well-known in the world of representation theory. We recall a few examples:
\begin{itemize}
\item The category of representations of a finite quantum group (a.k.a. finite dimensional semisimple Hopf algebra) is a fusion category.
\item The category of level $\ell$ positive energy representations of a loop group is fusion.
\item Given $A \subseteq B$ a finite depth finite index subfactor, so $A$ and $B$ are von Neumann algebras, then the {\em planar algebra} or {\em standard invariant} associated to the subfactor is essentially a fancy version of a fusion category.
\item Fusion categories also arise in many approaches to conformal field theory.

\end{itemize}

\noindent Given an object $X \in F$, let $X^\ast$ denote the (right) dual.

\begin{theorem}[Etingof-Nikshych-Ostrik \cite{MR2183279}]
Let $\cF$ be a fusion category, then the endofunctor which sends an object $X \in \cF$ to $X^{\ast \ast \ast \ast}$ is canonically and naturally monoidally isomorphic to the identity functor.
\end{theorem}

The usual proof begins by non-canonically realizing $\cF$ as the representation category of a weak Hopf algebra. Then we apply an algebraic analog of {\it Radford's} $S^4$ {\it formula}, where $S$ is the antipode in the Hopf algebra.  We outline a more canonical proof which utilizes the notion of full-dualizability to re-prove this theorem.  

To begin we must first find a home for fusion categories in a higher categorical setting, where it will be possible to discuss higher dualizability. This is established in \cite{DSPS_TC3} where a symmetric monoidal 3-category of tensor categories is constructed. {\em Tensor categories} are analogous to Fusion categories, but where the semi-simplicity assumption is dropped. This symmetric monoidal 3-category is a categorification of the 2-category of algebras, bimodules, and bimodule maps, and is given as follows:
\begin{center}
\begin{tabular}{rcl}
Objects && Tensor categories \\
1-morphisms && Bimodule categories\\
2-morphisms && Functors\\
3-morphisms && Natural transformations
\end{tabular}
\end{center}
The monoidal structure and composition of morphisms are given by the (relative) Deligne tensor product.

One example of a bimodule category is the identity bimodule category $\prescript{}{F}F_{F}$. We can also twist one of the actions by any tensor autoequivalence $\alpha: F \to F$ to get a new bimodule category $\prescript{}{F}F_{F^{\alpha}}$.

\begin{theorem}[\cite{DSPS_DTC1}]
The fully-dualizable objects of the above symmetric monoidal 3-category are precisely the separable fusion categories. In characteristic zero every fusion category is separable. In positive characteristic over a perfect field, if the unit object is simple, then separability is equivalent to global dimension non-zero.
In all cases for an object $\cF$ the Serre automorphism is given by the bimodule
\[
S_\cF = \prescript{}{F}F_{F^{\ast \ast}}.
\]
\end{theorem}

From the Radford isomorphism $R: S_\cF \circ S_\cF \simeq id$ we immediately deduce the following. (Passing from the bimodule to the endofunctor uses semisimplicity.)

\begin{cor}[ \cite{DSPS_DTC1}]
Let $\cF$ be a fusion category, then the endofunctor which sends an object $X \in \cF$ to $X^{\ast \ast \ast \ast}$ is canonically and naturally monoidally isomorphic to the identity functor.
\end{cor}

\section{Exercises} \label{sec:exercises_4}

\begin{exercise}
	Fill in the details for the proof of Lemma \ref{3catlemma}.
\end{exercise}

\specialsection*{The Unicity Theorem}





\section{Introduction to the Unicity Theorem} \label{sec:unicity_intro}

The {\em Unicity Theorem} and the accompanying machinery of \cite{BarSch1112} give a axiomatization of the theory of $(\infty,n)$-categories as well as several tools for verifying these axioms and producing comparisons in specific cases. I would like to explain the background and ideas behind the Unicity Theorem, as well as some of its consequences, but before we get to the business of addressing what exactly $(\infty,n)$-categories ought to be and how we should go about axiomatizing them, I would like us to begin by pondering two seemingly unrelated questions:

\vskip 1ex
\hskip 2cm What is a homotopy theory?

\vskip 1ex
\hskip 2cm What is an $(\infty,1)$-category?

\vskip 1ex
\noindent
We will discover that these questions are not really independent and that in fact $(\infty,1)$-categories and homotopy theories are really two sides of the same coin. 

In what follows we begin in \S\ref{sec:HomotopyTheory} by surveying a few popular approaches to the notion of abstract homotopy theory. In \S\ref{sec:Infty1Cats} we continue by surveying some notions of $(\infty,1)$-categories. This naturally leads us in \S\ref{sec:Comparions} to some of the known comparisons of these models, as well as the known comparisons of analogous theories of $(\infty,n)$-categories. We find that there are two fundamental problems. Briefly, in the $(\infty,1)$-case there are too many such comparisons, while for the higher $(\infty,n)$-case there are too few. 

After this we will turn in \S\ref{sec:Unicity}  to the resolution of these problems via the unicity theorem and related results. We will discuss various properties that the theory of $(\infty,n)$-categories ought to have, and how these ultimately lead us to the axiomatization of \cite{BarSch1112}.

\section{Homotopy theories} \label{sec:HomotopyTheory}

\subsection{Quillen Model Categories}

It is not surprising that there are several different notions of what an abstract homotopy theory should be. Let us recall some of these notions and how they compare with each other. The first and probably the most popular notion is 
Quillen model categories, $(\cM,\cC,\cF,\cW)$; where $\cM$ is a category and $\cW$ is the class of weak equivalences.  The classes $\cF$ and $\cC$ are the fibrations and cofibrations respectively.  The quadruple $(\cM,\cC,\cF,\cW)$ must satisfy a number of axioms expressing various lifting properties and closure properties of the classes $\cW$, $\cF$, and $\cC$.  

The structure of Quillen model categories allows one to closely mimic many constructions in classical homotopy theory. These include, under some assumptions, the ability to form {\em mapping  spaces} between objects, to form homotopically meaningful and invariant notions of limit and colimit (cleverly named {\em homotopy limits} and {\em homotopy colimits}), and to construct various {\em derived functors}, functors which preserve weak equivalences and approximate functors which don't. 

With these computational benefits and the existence of a multitude of important and useful examples, it is little wonder that Quillen model categories have become a standard tool. By now it is clear that they are a useful concept, however the notion of Quillen model category also suffers from a few defects. The first is that the required structure of a Quillen model category is quite stringent, and for some examples which one would hope to include as abstract homotopy theories it is not possible to produce the desired Quillen model structure. 

Secondly, the notion of Quillen model category suffers from an excess of structure. The first hint of this comes from the notion of equivalence of homotopy theory which is called {\it Quillen equivalence}. A Quillen equivalence between model categories consists of a pair of adjoint functors which, contrary to what one might initially suspect, only preserve a portion of the structure of a Quillen model category. They preserve only about half of this structure.

I learned of the following example from Tom Goodwillie (via the question and answer website MathOverflow \cite{Goodwillie_MO}):
\begin{prop} On the category of sets there exist exactly nine model structures;  there are three Quillen equivalence classes.\footnote{These three Quillen equivalence classes correspond to the theories of 0-types, $(-1)$-types, and $(-2)$-types.}
\end{prop}
\noindent What this illustrates is the failure of Quillen model structures to closely align with the notion of homotopy theory they are supposed to represent. Even on the category of sets, which is very simple compared to most examples which arise as Quillen model categories, there are three times as many Quillen model structures as there are corresponding homotopy theories. Quillen model categories give a good model of abstract homotopy theories; they certainly get the job done,  but they are wearing a tuxedo while they are doing it. We will see some other notions which are more general and have fewer cufflinks and bow-ties, which give a more revealing description of a homotopy theory.   


\subsection{Simplicial and Relative Categories}

There are several alternative approaches to abstract homotopy theories. Here are two more:
\begin{itemize}	
\item Relative categories, $(\cC,\cW)$; where $\cC$ is a category and $\cW$ is the class of weak equivalences.  The only condition is that $\cW$ contains all the identities of $\cC$.

\item Simplicial categories, by which we mean categories enriched in simplicial sets.	


\end{itemize}
The notion of abstract homotopy theory was clarified by the work of Dwyer and Kan \cite{MR584566}.
They considered a very minimal notion consisting of a category together with a subcategory  of ``weak equivalences'' which is only required to contain the identities (such a pair is now called a {\em relative category} \cite{MR2877401}). From this they constructed a functor called the {\it hammock localization}: 
\[
L^H : \mathrm{RelCat} \to \mathrm{Cat}_\Delta .
\]
which takes a relative category $(\cM, \cW)$ and produces a simplicial category $L^H(\cM, \cW)$. 
Dwyer and Kan were able to show that, in principle, much of the structure provided by a model structure, such as the mapping spaces, can be recovered from this simplicial category, and hence from the weak equivalences alone. However without additional assumptions on the weak equivalences, extracting such information is usually impractical. 

It is often said that the Quillen equivalences between model categories behave, themselves, something like the weak equivalences of a homotopy theory. This is not literally true, as the 2-categorical structure would also have to be incorporated, however it raises the question of whether there could be a `homotopy theory of homotopy theories'? 

In their work, Dwyer and Kan also provided a definition of weak equivalence between simplicial categories. As this makes the collection of simplicial categories itself into a relative category\footnote{Of course there are size issues which must be addressed to interpret this rigorously. These can be handled in one of the standard ways \cite{0810.1279}, and we will continue to ignore these issues in this expository account.},   this can be viewed as the first construction a {\em homotopy theory of homotopy theories}. This was later improved by the work of Julie Bergner \cite{MR2276611} who constructed a Quillen model structure on the category of simplicial categories in which the weak equivalences are precisely the Dwyer-Kan equivalences.

\subsection{Rezk's homotopy theory of homotopy theories} Another model of abstract homotopy theory was introduced by Rezk in \cite{MR1804411} with the express purpose of developing the homotopy theory of homotopy theories more fully. 
The rough idea is that to a homotopy theory $X$ of some unspecified sort we should be able to associate several `moduli spaces' or classifying spaces. First we can let $X_0$ be the classifying space of objects. This will look essentially like a disjoint union of spaces of the form $B\Aut^h(x)$, where $x$ is an object of the homotopy theory and $\Aut^h(x)$ is the {\em derived mapping space of automorphisms} of that object. The disjoint union will be over all isomorphism classes of objects. Next we form the space $X_1$, which is a similarly constructed classifying space for arrows in the homotopy theory $X$. The space $X_2$ will be the classifying space for pairs of composable arrows, the space $X_3$ will be the classifying space of triples of composable arrows, and so on, giving rise to a simplical space $X_\bullet$. 

The collection of these classifying spaces, which together form a simplicial space, forms the basis of Rezk's model of homotopy theories. In fact he constructs a model category structure on the category of all simplicial spaces, thereby giving another model of the homotopy theory of all homotopy theories. The fibrant objects of that model structure are now known as {\em complete Segal spaces} (CSS). One of Rezk's insights was that this model category is much better behaved that other previous attempts. The complete Segal space model category is both a simplicial and cartesian model category. In particular given two homotopy theories in Rezk's sense it is easy to construct a mapping object between these, which will again be a homotopy theory. 

Thanks to the work of Bergner, we now know that Rezk's model contains essentially the same information as the simplicial categories studied by Dwyer and Kan (they are Quillen equivalent model categories). We will come back to this point later.   Let us just remark that in his original work, Rezk also constructed the {\em classification diagram} functor
\[
cd:\mathrm{RelCat} \to \mathrm{CSS}.
\]
which takes a relative category and produces a complete Segal space, thereby giving a direct way to compare these notions as well.

\section[(infty,1)-categories]{$(\infty,1)$-categories} \label{sec:Infty1Cats}

Let's temporarily leave the world of homotopy theory and discuss higher category theory. 
An ordinary category has objects and morphisms between the objects. These morphisms compose associatively and there are identities. Similarly, a higher category is supposed to have objects, morphisms between the objects (called 1-morphisms), morphisms between the 1-morphisms (called 2-morphisms), and so on. In addition there should be various means of composing these various morphisms, as well as identities. 

The phrase `and so on' is ambiguous. It can mean that we continue until some finite stage $n$, where we have $n$-morphisms between the $(n-1)$-morphisms. This gives us $n$-categories. For $n=\infty$ (i.e. for $\infty$-categories) we should allow $n$-morphisms of arbitrary dimension. Of course this is not so much a definition as an informal philosophy about what sort of ingredients should enter in a definition of higher category. 

Making this amorphous philosophical idea into a precise mathematical definition is not an easy undertaking, nor is there a single clear route for obtaining such a definition. On the contrary, there are numerous competing definitions of higher category and the search for a useful comparison of these various notions has been an elusive and long-standing goal \cite{MR1883478}. In their proposal for the 2004 IMA program on $n$-categories, Baez and May underlined the difficulties of this \emph{Comparison Problem}:
\begin{quote}
	It is not a question as to whether or not a good definition exists. Not one, but many, good definitions already do exist 
	[\dots].
	There is growing general agreement on the basic desiderata of a good definition of $n$-category, but there does not yet exist an axiomatization, and there are grounds for believing that only a partial axiomatization may be in the cards.
\end{quote}

\noindent One of these basic desiderata is that the theory of $n$-categories should satisfy the {\em homotopy hypothesis}, to which we now turn. 

\subsection{The Homotopy Hypothesis and Simplicial Categories}

One possible litmus test for any proposed theory of higher categories is the {\it homotopy hypothesis}. A baby version of this is that equivalence classes of $n$-groupoids should be in natural bijection with equivalence classes of homotopy $n$-types. A stronger version would require an equivalence of homotopy categories, and a still stronger version would require that there is a homotopy theory of higher categories inducing an equivalence of homotopy theories between the $n$-groupoids and $n$-types. 
The most well-known instance of this is the equivalence of the theory of 1-groupoids with the theory of homotopy 1-types, which is implemented by the {\em fundamental groupoid} functor (and its weak inverse the classifying space functor). Letting $n$ pass to $\infty$ the homotopy hypothesis asserts that the homotopy theory of $(\infty,0)$-categories (a. k. a. $\infty$-groupoids) is the same as the homotopy theory of topological spaces. 

In an $(\infty,1)$-category we are supposed to have $(\infty,0)$-categories of morphisms between any two given objects, that is, by the preceding paradigm, we should have mapping spaces between any two objects. Thus an $(\infty,1)$-category should be something similar to a category enriched in spaces. We have such a homotopy theory for the category of simplicially enriched categories, namely Bergner's model category \cite{MR2276611}. 
But then we see immediately that, in this instance, the homotopy theory of $(\infty,1)$-categories is the same as the homotopy theory of homotopy theories.

\subsection{Quasicategories and Segal Categories} A central theme of higher category theory is that composition of morphisms should not necessarily be strictly associative, but only associative up to higher coherent morphisms. The model of $(\infty,1)$-categories as (strict) simplicial categories goes against this and is fairly rigid. Rezk's model of complete Segal spaces can be viewed as allowing for weaker compositions, but there are also other models. I will mention two more,  based on generalizing the nerve construction. The first notion, that of {\em Segal categories}, was already described in section \ref{sec:SegalCats}, but we will review the definition here for the convenience of the reader.

Let $\cC$ be an ordinary category and recall the nerve functor from categories to simplicial sets
\[
N : \mathrm{Cat} \to \mathrm{sSet}.
\]
The 0-simplices of $N\cC$ are the objects of $\cC$ and the $n$-simplices are given by $n$-tuples of composable morphisms in $\cC$. We saw in section \ref{sec:SegalCats} that the nerve functor is fully-faithful and that we can characterise its image terms of the {\em Segal maps}.

Recall that the {\em spine} $S_n$ of the simplex $\Delta[n]$ is the sub-simplicial set consisting of the union of all the consecutive 1-simplices, The inclusion
\begin{equation*}
	s_n: S_n = \Delta^{\{0,1\}} \cup^{\Delta^{\{1\}}} \Delta^{\{1,2\}} \cup^{\Delta^{\{2\}}} \cdots  \cup^{\Delta^{\{n-1\}}}  \Delta^{\{n-1,n\}} \to \Delta[n],
\end{equation*}
 corepresents the $n^\textrm{th}$ {\em Segal map}:
\begin{equation*}
	s_n: X_n \to X(S_n) = X_1 \times_{X_0}  X_1 \times_{X_0} \cdots \times_{X_0}  X_1.
\end{equation*}
As we saw in section \ref{sec:categorical_nerve}, 
a simplicial set is isomorphic to the nerve of a category precisely when each Segal map is a bijection for $n \geq 1$.

It is helpful to understand how a category arises from such a simplicial set. The 0-simplices of the simplicial set form the objects of the corresponding category and the 1-simplices form the morphisms (the two face maps $X_1 \rightrightarrows X_0$ give the source and target of morphisms). The composition of composable morphisms is given by considering the following diagram:
\begin{equation*}
	X_1 \times_{X_0} X_1 \stackrel{(d_0,d_2)}{\longleftarrow} X_2 \stackrel{d_1}{\rightarrow} X_1.
\end{equation*}
The leftward map, the Segal map, is an isomorphism, and replacing the leftward map with its inverse we obtain the composition map. The simplicial identities ensure there are identity morphisms, and the associativity of composition is ensured by considering the Segal map for the 3-simplices. 

A simplicial category also has a nerve which is a simplicial space (i.e., a bisimplicial set). Again we can characterize those simplicial spaces which are the nerves of simplicial categories: they are precisely those for which the Segal maps are isomorphisms and for which the space of 0-simplices is discrete (i.e. a constant simplicial set). 

A {\em Segal category} \cite{MR81h:55018, math.AG/9807049, MR2276611} (see~section \ref{sec:Segal_cats_subsection}), like a simplicial category, consists of a simplicial space for which the space of 0-simplices is discrete. However instead of requiring the Segal maps to be isomorphisms, we only require them to be weak equivalences of spaces, that is for each $n$ we have a homotopy equivalence of simplicial sets:
\[
\cC_n \xrightarrow{\simeq} \underbrace{\cC_1 \times_{\cC_0} \cC_1 \times_{\cC_0} \dotsb \times_{\cC_0} \cC_1}_{n \text{ factors}}  .
\]
We still have a diagram:
\begin{equation*}
	X_1 \times_{X_0} X_1 \stackrel{\simeq}{\longleftarrow} X_2 \stackrel{d_1}{\rightarrow} X_1.
\end{equation*}
By choosing a homotopy inverse to the first map\footnote{This requires that the spaces involved are fibrant.} we obtain a composition map. Even better, the space of homotopy inverses (including the homotopies $s_2 \circ s_2^{-1} \simeq id$ and $s_2^{-1} \circ s_2 \simeq id$)  is a contractible space parameterizing the potential composition maps. These will generally fail to be associative on the nose, but will be associative up to homotopy. Better, the space of homotopy inverses to the third Segal map $s_3$ (another contractible space parametrizing `triple compositions') can be used to obtain a contractible space of homotopies witnessing the coherent associativity of the composition maps. Similarly the remaining spaces in the Segal category provide still higher homotopical coherence data. 



We would be remiss if we didn't mention a final theory of $(\infty,1)$-categories, the {\em quasicategories}, which were first introduced by Boardman and Vogt \cite{MR0420609} in their work on homotopy coherent diagrams. This model has become especially important partly for its ease of use and largely because of the extensive theory developed by Joyal \cite{Joyal_CRM, Joyal_notes} and Lurie \cite{MR2522659}. This body of work includes $(\infty,1)$-categorical notions of limit, colimit, Kan extension, localizations, and many other constructions. 

Recall the {\em horn} $\Lambda^i[n]$ ($0 \leq i \leq n$) which is a subcomplex of $\Delta[n]$ obtained by removing the single non-degenerate $n$-simplex and the $i^\textrm{th}$ face. A simplicial set is a {\em Kan complex} if `all horns have fillers', that is for every map $\Lambda^i[n] \to X$ the dashed arrow in the diagram below exists and makes it a commutative diagram.
\begin{center}
\begin{tikzpicture}
	\node (LT) at (0, 1.5) {$\Lambda^i[n]$};
	\node (LB) at (0, 0) {$\Delta[n]$};
	\node (RT) at (2, 1.5) {$X$};
	\draw [right hook->] (LT) -- node [left] {$$} (LB);
	\draw [->] (LT) -- node [above] {$$} (RT);
	\draw [->, dashed] (LB) -- node [below] {$$} (RT);
\end{tikzpicture}
\end{center}
Those simplicial sets which are isomorphic to the nerve of a category can equivalently be characterized as those simplicial set which have {\em unique filler} for the {\em inner horns}. That is they only are guaranteed to have fillers for the inner horns ($0< i < n$), and in this case the dashed arrow above is unique. A {\em quasicategory} is a simplicial set which has fillers (possibly non-unique) for every inner horn. This notion generalizes both Kan complexes and the nerves of categories.


\subsection{Previous comparisons of theories of $(\infty,1)$-categories} \label{sec:Comparions}

The first substantial comparison of homotopy theories of $(\infty,1)$-categories was the work of Julia Bergner \cite{MR2276611}. In 2005 she first constructed a model structure on the category of simplicial categories with the weak equivalences those given by Dwyer and Kan. The following year she constructed the following zig-zag of Quillen equivalences (only right Quillen functors are shown):
\begin{equation*}
	\cat_\Delta \to \Seg_{proj} \leftarrow \Seg_{inj} \leftarrow \CSS
\end{equation*}
These connect the theory of simplicial categories, two versions of Segal categories, and Rezk's complete Segal spaces. Since then the flood gates were released and many more comparisons have come pouring through. Shortly after Bergner's comparison Joyal and Teirney \cite{MR2342834} produced two distinct Quillen equivalences between Segal categories and quasicategories (with adjunctions in opposite directions) and two distinct Quillen equivalences between Rezk's complete Segal spaces and quasicategories. Around the same time it was observed \cite{MR2522659} that the homotopy coherent nerve of Porter and Cordier \cite{MR838654} provided a Quillen equivalence between quasicategories and simplicial categories. 

In 2010 Dugger and Spivak generalized this providing a plethora of variations on the homotopy coherent nerve \cite{DSa, DSb}. More recently Barwick and Kan have constructed a model category structure on the category of relative categories and have provided Quillen equivalences with both quasicategories and complete Segal spaces \cite{MR2877401}. They have also shown that the Hammock localization of Dwyer and Kan, while not part of a Quillen adjunction, is still an equivalence of relative categories \cite{MR2877402}.  
So we find that the various theories of $(\infty,1)$-categories are connected by an intricate web of Quillen equivalences. Figure \ref{fig:ComparisonsN=1} shows a diagram of some of these Quillen equivalences (only the right Quillen functors are shown)\footnote{The dashed arrow \tikz[baseline=-3pt]{\draw [thick, loosely dotted] (0,0) -- (0.75,0);} represents the hammock localization. It is not a Quillen equivalence, but is an equivalence of relative categories.}. 

%
%
%

\begin{figure}[htbp]
\begin{center}
	\begin{tikzpicture}[thick, baseline=1cm]
		\node (B) at (0, 2) [minimum width=0.5cm, rectangle, draw,  label=right:Bergner] {};
		\node (JT) at (0, 1.5) [minimum width=0.5cm, densely dashed, rectangle, draw, label=right:{Joyal, Teirney}] {};
		\node (J) at (0, 1) [minimum width=0.5cm, loosely dashed, rectangle, draw, label=right:Joyal] {};
		\node (DS) at (0, 0.5) [minimum width=0.5cm, densely dotted, rectangle, draw, label=right:{Dugger, Spivak}] {};
		\node (BK) at (0, 0) [minimum width=0.5cm, dashdotdotted, rectangle, draw, label=right:{Barwick, Kan}] {};
		\node (BDK) at (0, -0.5) [minimum width=0.5cm, loosely dotted, rectangle, draw, label=right:{Barwick, Dwyer, Kan}] {};
	\end{tikzpicture}
\begin{tikzpicture}[thick, baseline = 0cm]
	\node (Qcat) at (0,0) {$\mathrm{QCat}$};
	\node (RelCat) at (18:2.5cm) {$\mathrm{RelCat}$};
	\node (CSS) at (306:2.5cm) {$\CSS$};
	\node (Segcatc) at (234:2.5cm) {$\Seg^{inj}$};
	\node (Segcatf) at (162:2.5cm) {$\Seg^{proj}$};
	\node (sCat) at (90:2.5cm) {$\cat_\Delta$};
	\draw [->] (sCat) to [bend right=28] node [above left] { } (Segcatf);
	\draw [->] (Segcatc) to [bend left] node [left] { } (Segcatf);
	\draw [->] (CSS) to [bend left=20] node [below] {} (Segcatc);
	\draw [->, dashdotdotted] (RelCat) to [bend left] node [right] {} (CSS);
	\draw [->, densely dotted] (sCat) to [bend left] node [above right] {} (Qcat); 
	\draw [->, loosely dashed] (sCat) to [bend right] node [left] {} (Qcat);
	\draw [->, densely dashed] (Segcatc) to [bend right=10] node [right] {} (Qcat);
	\draw [<-, densely dashed] (Segcatc) to [bend left=10] node [above left] {} (Qcat);
	\draw [->, densely dashed] (CSS) to [bend right=10] node [right] {} (Qcat);
	\draw [<-,  densely dashed] (CSS) to [bend left=10] node [left] {} (Qcat);
	\draw [->, dashdotdotted] (RelCat) to node [above=0.1cm] {} (Qcat);
	\draw [<->, loosely dotted] (RelCat) to [bend right=28] node [above right] {} (sCat);
\end{tikzpicture} 
\end{center}

\caption{Right Quillen equivalences between some homotopy theories of $(\infty,1)$-categories.}
\label{fig:ComparisonsN=1}
\end{figure}
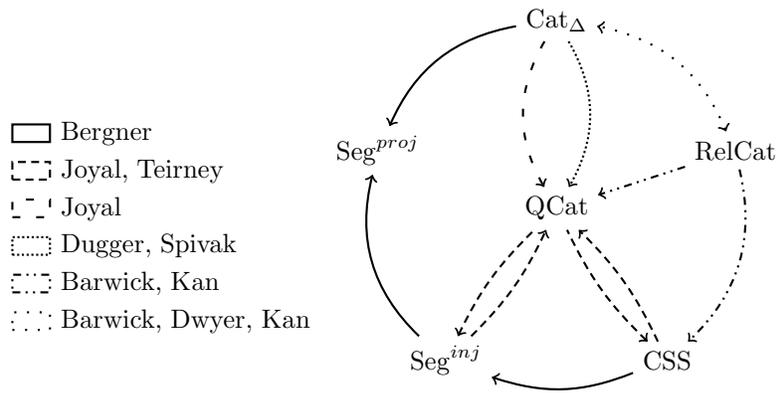

Thus we see that all the previously mentioned models for the theory of $(\infty,1)$-categories or the theory of homotopy theories are in fact equivalent. However, the situation is a bit troubling as there appear to be many possible ways in which they are equivalent. How do we know that it doesn't matter {\em how} we pass around the diagram, say from simplicial categories to Segal categories?

In fact the diagram in Figure~\ref{fig:ComparisonsN=1} is not commutative, even up to natural weak equivalence! There is a monodromy problem.  In one case, a difficult result of Bergner \cite{MR2439415} shows that two important paths in the Figure~\ref{fig:ComparisonsN=1} can be connected by a zig-zag of natural weak equivalences. We will see later, as a consequence of To\"en's theorems, that in fact the same holds for every pair of paths.

\subsection{Comparisons of theories of $(\infty,n)$-categories}

There are higher dimensional analogs of each of these theories and we can similarly ask to compare these homotopy theories of $(\infty,n)$-categories. In this case, however, the situation is quite different. Instead of having a wealth of equivalences, we find ourselves with a dearth. Let $\cM$ be a model category modeling $(\infty,n-1)$-categories. Some of these higher generalizations include:
\begin{itemize}
	\item $\cat_\cM$, categories enriched in $\cM$, if $\cM$ is sufficiently nice.
	\item $\Seg_\cM^{inj}$ and $\Seg_\cM^{proj}$,  Segal categories enriched in $\cM$, if $\cM$ is sufficiently nice. Iterating this starting with Segal categories produces {\em Segal $n$-categories} $\Seg_n$.
	\item $\CSS_\cM$, complete Segal spaces enriched in $\cM$, if $\cM$ is sufficiently nice (this notion of nice is different from the last one). Iterating this starting from Rezk's complete Segal spaces yield's Barwick's {\em $n$-fold complete Segal spaces} which featured in Lurie's work on the Cobordism Hypothesis \cite{MR2555928}. 
	\item Other models based on localizing presheaves of spaces:
	\begin{itemize}
		\item $\Theta_nSp$ Rezk's complete $\Theta_n$-spaces. 
		\item $\pitchfork$-$Sh$, Ayala and Rozenblyum's {\em transversality sheaves}.
	\end{itemize}
	\item Models similar to quasicategories:
	\begin{itemize}
		\item $\mathrm{Comp}_n$ Verity's weak complicial sets \cite{Verity20081081, MR2342841}.
		\item $\mathrm{QCat}_n$ Ara's $n$-quasicategories (based on $\Theta_n$-sets) \cite{Ara1206}.
	\end{itemize}
	\item $\mathrm{RelCat}_n$ the $n$-relative categories of Barwick and Kan \cite{BarKan1102}. 
\end{itemize}

\noindent There are two previous general existence/comparison results:
\begin{theorem}[\cite{Lur0905} Pr.~1.5.4 and Pr.~2.3.1]
	Let $\cS$ be a model category which is combinatorial, left proper, where every monomorphism is a cofibration, where filtered colimits are left exact in the underlying quasicategory, and where the underlying quasicategory is an `absolute distributor' \cite{Lur0905}. If $\cS$ is also simplicial, then the model structures $\CSS_\cS$ and $\Seg^{inj}_\cS$ exist and there is a Quillen equivalence:
	\begin{equation*}
		\CSS_\cS \rightleftarrows \Seg^{inj}_\cS.
	\end{equation*}
	Moreover $\CSS_\cS$ will again be simplicial and satisfy the above properties. 
\end{theorem}

\begin{theorem}[ \cite{Lur0905} Th.~2.2.16 and \cite{MR2883823} Th.~21.3.2]
	Let $\cK$ be a model category which is combinatorial, where every monomorphism is a cofibration, and where the class of weak equivalences is closed under filtered colimits. Then, if $\cK$ is cartesian, the model structures $\cat_\cK$, $\Seg^{proj}_\cK$, and $\Seg^{inj}_\cK$ exist and we have natural Quillen equivalences:
	\begin{equation*}
		\cat_\cK \to \Seg^{proj}_\cK \leftarrow \Seg^{inj}_\cK.
	\end{equation*}
	Moreover $\Seg^{inj}_\cK$ is again cartesian and satisfies these properties. 
\end{theorem}
\noindent There are also a spattering of other specific comparisons, including a recent and lucid treatment by Bergner-Rezk \cite{1204.2013} of this last result in the case $\cK = \Theta_{n-1}Sp$. These previous comparisons are summarized in Figure~\ref{fig:ComparisonsN}.


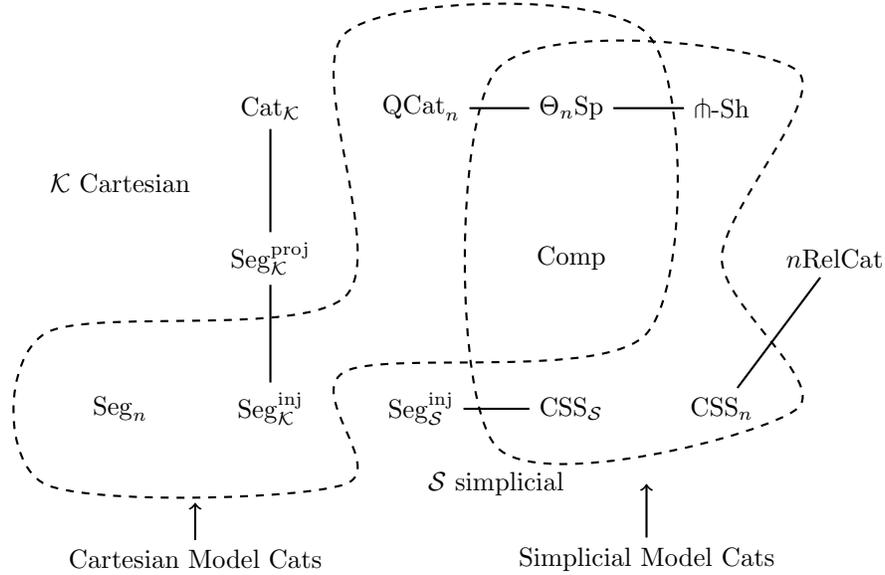
\begin{figure}[htbp]

\begin{center}
\begin{tikzpicture}[thick]

\node (K) at (0, 3) {$\cK$ Cartesian};
\node (Segn) at (0,0) {$\mathrm{Seg}_n$};

\node (C) at (2,4) {$\mathrm{Cat}_\cK$};
\node (SegP) at (2,2) {$\mathrm{Seg}_\cK^\textrm{proj}$};
\node (SegIk) at (2,0) {$\mathrm{Seg}_\cK^\textrm{inj}$};
\node (QCatn) at (4,4) {$\mathrm{QCat}_n$};
\node (SegIs) at (4,0) {$\mathrm{Seg}_\cS^\textrm{inj}$};
\node (T) at (6,4) {$\Theta_n \mathrm{Sp}$};
\node (Comp) at (6,2) {Comp};
\node (CSS) at (6,0) {$\mathrm{CSS}_\cS$};
\node (S) at (5,-1) {$\cS$ simplicial};
\node (TS) at (8,4) {$\pitchfork$-Sh};
\node (CSSn) at (8,0) {$\mathrm{CSS}_n$};
\node (RC) at (9.5,2) {$n\mathrm{RelCat}$};

\draw (C) -- (SegP) -- (SegIk);
\draw (QCatn) -- (T) -- (TS);
\draw (SegIs) -- (CSS);
\draw (CSSn) -- (RC);
\draw [dashed] plot [smooth, smooth cycle] coordinates {(-1,1) (3,1.5) (3,5) (7,5) (7,1) (3, 0.5) (3, -1) (-1,-1)} -- cycle;
\draw [dashed] plot [smooth, smooth cycle] coordinates {(5,-0.5) (5,4.5) (9,4.5) (8,2) (9, 0)};

\node (A) at (1,-2) {Cartesian Model Cats};
\node (B) at (7,-2) {Simplicial Model Cats};
\draw [->] (A) -- +(0,0.75);
\draw [->] (B) -- +(0,1);

\end{tikzpicture}
\end{center}
\caption{Comparison of theories of $(\infty,n)$-categories, before unicity. The lines denote equivalences of homotopy theories. $\cS$ and $\cK$ represent nice model categories of $(\infty,n-1)$-categories which are, respectively, simplicial or Cartesian. Note: $n\mathrm{RelCat}$ is not a model cateogry, merely a relative category.}
\label{fig:ComparisonsN}
\end{figure}

%
%
%
%

These comparisons, however, are not sufficient.
The above theorems can only be applied in conjunction if the model category $\cM$ is simultaneously simplicial and cartesian. However the category of complete Segal spaces enriched in $\cM$ is usually not cartesian, while Segal categories enriched in $\cM$ is almost never simplicial. Thus we are left at an impasse.

\section{The Unicity Theorem} \label{sec:Unicity}
These problems, the equivalence-monodromy problem when $n=1$ and the existence of equivalences for higher $n$, are solved by the unicity theorems, which were first proven by To\"en in the case $n=1$. To\"en stated his results as a pair of theorems, the first one providing axioms which characterized the homotopy theory of homotopy theories up to equivalence, and a second theorem which settled the ambiguity of this equivalence. 

Each of the theories proposed above includes, at a minimum, a category and a notion of weak equivalence, hence a relative category (a.k.a. a homotopy theory). Thus, for example, in the case $n=1$ we may regard Figure~\ref{fig:ComparisonsN=1} as a diagram in the model category $\mathrm{RelCat}$. If we prefer one of the other models of homotopy theories, such as quasicategories, then we may regard Figure~\ref{fig:ComparisonsN=1} as a diagram in that category using our favorite preferred comparison functor. 

In any case, as a diagram of equivalences in a model category, the question of what extent it commutes or can be made to commute up to higher homotopical data is completely governed by the derived (topological) automorphism group of any one of the objects. This group, which a priori has interesting higher topological information, was computed in the $n=1$ case by To\"en to be $\mathrm{Aut}^h \left ( L^H \mathrm{CSS} \right ) \simeq \ZZ/2$. So in fact this is a discrete group.

 Moreover this involution is the one which sends a category to its opposite, and hence is completely detected by its restriction to a certain full subcategory consisting of two objects, the {\em $0$-cell} $C_0 = pt$ and the {\em $1$-cell} $C_1$. This later is the `free walking arrow', it has two objects, $0$ and $1$, and a unique non-identity arrow $0 \to 1$. It is easy to check that all of the comparison functors of Figure~\ref{fig:ComparisonsN=1} induce the identity action on this subcategory, and as a corollary, the diagram of Figure~\ref{fig:ComparisonsN=1} commutes when considered as a diagram of homotopy theories.

More generally we have:

\begin{theorem*}
There exist four axioms which characterize the quasicategory of $(\infty,n)$-categories up to equivalence. Moreover the space of quasicategories satisfying these axioms is $\left (B\ZZ/2 \right )^n$.
\end{theorem*}

\noindent Here quasicategories could be replaced with any of the equivalent notions listed in Figure~\ref{fig:ComparisonsN=1}. 
We will make the details of this axiomatization more precise below. Having an axiomatization is only useful if there are examples which can be shown to satisfy these axioms. Fortunately we have:

\begin{theorem} \label{thm:All_Models_Sat_Axioms}
With the exception of complicial sets, all the models of $(\infty,n)$-categories appearing in the diagram of Figure~\ref{fig:ComparisonsN} above satisfy the four axioms.  
\end{theorem}

\noindent It is an open problem as to whether a variant of Verity's complicial sets satisfies the axioms.

\subsection{First Properties}

There are many desiderata one could imagine for the homotopy theory of $(\infty,n)$-categories. 
Before stating the axioms let us describe a few of these. First, strict $n$-categories should be {examples} of weak $n$-categories and hence $(\infty,n)$-categories. We should expect that there is a functor from the category of strict $n$-categories to our potential theory $\cC$. However we don't expect it to be fully-faithful in any sense as there should be many more weak functors than just the strict ones, and moreover there should be weak natural isomorphisms between these functors, and higher morphisms between those, etc. 

One of the most important examples of strict $n$-categories are the {\em cells} $C_i$ $0 \leq i \leq n$. The $i$-cell is the free walking $i$-morphism. $C_0 = pt$ is the terminal category, $C_1 = \{ 0 \to 1 \}$ has two objects and a single non-trivial morphism between them. $C_2$ looks as follows:
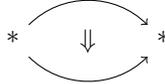
\begin{center}
	\begin{tikzpicture}
	\node (A) at (0,0) {*};
	\node (B) at (2,0) {*};
	\node at (1,0) {$\Downarrow$};
	\draw [->] (A) to [out=45, in=135] (B);
	\draw [->] (A) to [out=-45, in=-135] (B);
\end{tikzpicture}
\end{center}
\noindent They can be defined inductively as follows: $C_i$ has exactly two objects, $0$ and $1$. The only non-identity morphisms occur from $0$ to $1$, and we have $\hom_{C_i}(0,1) = C_{i-1}$. 
The cells are important as they form the basic building blocks with which we can obtain any higher category. By gluing cells together using homotopy colimits we expect to be able to build any possible $(\infty,n)$-category. 

\vskip 1ex
\begin{itemize}
	\item [(P1)] The quasicategory $\cC$ is generated under (homotopy) colimits by the cells, that is the smallest full sub-quasicategory containing the cells and closed under colimits is $\cC$ itself.
\end{itemize}

In any homotopy theory $\cD$, there exists a distinguished subcategory $\tau_{\leq 0} \cD$ of {\em 0-truncated objects}. This is the full subcategory of objects $X \in \cD$ such that the derived mapping spaces $\cD(D,X)$ are (homotopically) discrete for any $D$. The category $\tau_{\leq 0} \cD$ is an ordinary category and it consists of that part of $\cD$ which has a trivial homotopy theory. For example  when $\cD = Top$, then $\tau_{\leq 0} \cD \simeq Set$, the category of sets, realized as the discrete spaces. 

Recall that the homotopy hypothesis is the statement that the homotopy theory of $n$-groupoids is equivalent to the homotopy theory of $n$-types. Thus we see that the homotopy theory of ordinary 1-categories must contain a non-trivial, albeit simple, portion. Namely it contains a portion which is equivalent to the theory of 1-types. For a category $X$ to be in $\tau_{\leq 0} \mathrm{Cat}$, it is necessary that for every category $D$ the groupoid of fucntors $D \to X$ and natural isomorphisms between them must be `homotopically discrete'. Taking the case $D=pt$, we see that, in particular, the maximal subgroupoid of $X$ must be homotopically discrete. That is to say, any two isomorphic objects of $X$ are uniquely isomorphic. 

In fact, up to equivalence, we may take $\tau_{\leq 0} \mathrm{Cat}$ to consist precisely of those categories which have no non-trivial isomorphisms, i.e., every isomorphism is an identity map. 
Two functors between such categories are equivalent if and only if they are identical, in which case the equivalence is the identity. Thus the groupoid of functors and natural isomorphisms between such categories is discrete. 

The 2-dimensional case is also quite instructive. We can similarly consider $\tau_{\leq 0} \mathrm{Bicat}$. Up to equivalence this consists of those bicategories for which the only invertible 2-morphisms are identities and for which the only weakly invertible 1-morphisms are also identities. In this case we learn several things. First, since the coherence isomorphisms are necessarily identities, such a bicategory is automatically a {\em strict} 2-category. Moreover, instead of considering 1-morphisms which were weakly invertible, we could equivalently have considered those 1-morphism with strict inverses. There is no difference as long as the only invertible 2-morphisms are identities. This motivates the following definition:

\begin{definition}
	A strict $n$-category is {\em gaunt} if for all $k\geq 1$ the only invertible $k$-morphisms are identities.  
\end{definition}

\noindent Moreover we see that the homotopy theory of $(\infty,n)$-categories should satisfy the following additional property:

\vskip 1ex 
\begin{itemize}
	\item [(P2)] The category of 0-truncted objects of $\cC$, $\tau_{\leq 0} \cC$, is equivalent to the category of gaunt $n$-categories. 
\end{itemize}

\noindent This property is satisfied by all of the equivalent notions of $(\infty,1)$-category, as well as Rezk's $\Theta_n$-spaces \cite{MR2578310, MR2740648}. In fact these two simple properties are enough to recover part of the unicity results:

\begin{quote_prop} \label{Prop:automorphisms}
	If $\cC$ is a quasicategory satisfying properties (P1) (plus a strong generation property described in the next section) and (P2), then $\Aut^h(\cC)$ is equivalent to a subgroup of the discrete group $ (\ZZ/2)^n$. 
\end{quote_prop}

\begin{proof}[Proof Sketch]
	Any equivalence of $\cC$ must preserve the subcategory $\tau_{\leq 0} \cC$, hence restricts to a self-equivalence of the category of gaunt $n$-categories. Also, any equivalence must preserve homotopy colimits, so by (P1) we see that the value of any such functor on objects is in fact completely determined by the restriction to the cells. A direct calculation \cite{BarSch1112} shows that there are at most $(\ZZ/2)^n$ such equivalences, and moreover that they preserve the cells up to isomorphism (though they permute the maps between cells). However to pin down the value of the automorphism on morphisms as well, we will need a strong generation property, as described in the next section. 
\end{proof}

Thus we see that these two properties, desirable for any homotopy theory of $(\infty,n)$-categories, are enough to solve the monodromy problem. However these properties alone do not determine the theory. 

\subsection{Strong Generation}

One of the advantages of using the language of quasicategories to express the notion of homotopy theory is that it allows one formulate universal properties of the homotopy theory that could be difficult to formulate otherwise. As an example consider the following statement: the quasicategory $Top$ of spaces is freely generated under homotopy colimits by the singleton space $pt$. 

Such a statement corresponds to a universal property for the theory of spaces. To begin with, every functor $F: Top \to \cD$ which preserves homotopy colimits is determined up to equivalence by its restriction along $i:\{ pt \} \subseteq Top$. In fact $F$ is its own (homotopy) left Kan extension of its restriction to $\{ pt \}$. 

\begin{equation*}
	F(X)  \simeq  \colim_{pt \to X} F(pt).
\end{equation*}

\noindent This (homotopy) colimit is taken over the $(\infty,1)$-category of maps $pt \to X$. In particular we can apply this to the case when $D = Top$ and $F$ is the identity functor. This gives a universal formula for how to build any space as a homotopy colimit of contractible spaces. 
\begin{center}
\begin{tikzpicture}
	\node (LT) at (0, 1.5) {$\{ pt \}$};
	\node (LB) at (0, 0) {$Top$};
	\node (RT) at (2, 1.5) {$Top$};
	\draw [right hook->] (LT) -- node [left] {$i$} (LB);
	\draw [right hook->] (LT) -- node [above] {$i$} (RT);
	\draw [->, dashed] (LB) -- node [below right] {$Lan_i i = id_{Top}$} (RT);
\end{tikzpicture}
\end{center}
\begin{definition}
	Let $f:\cD' \to \cD$ be map of (presentable) quasicategories, then we say that $f$ {\em strongly generates} $\cD$ if the homotopy left Kan extension of $f$ along $f$ is the identity functor of $\cD$. 
\end{definition}

We can equivalently write this by saying that for all $D \in \cD$, the following canonical map is an equivalence:
\begin{equation*}
	D \simeq \colim_{D' \in \cD', \; f(D') \to D} f(D')
\end{equation*}
\noindent For example the homotopy theory of spaces is strongly generated by the inclusion of the terminal object $\{pt\} \hookrightarrow Top$. 

The category of spaces is {\em universal} with this property in the following sense. If $\cD$ is any presentable quasicategory which is strongly generated by its terminal object, then $\cD$ is a {\em localization} of $Top$, i.e., there exists a adjunction $L: Top \leftrightarrows \cD: R$ with $R$ fully-faithful. The functor $L$ is determined by the image of $\{pt\}$, which is the terminal object in $\cD$. 

Since the theory of quasicategories builds the homotopy theory of spaces into its framework, this is perhaps not terribly surprising, but it leads us to ask whether the theory of $(\infty,n)$-categories might have an analogous universal property? 
Indeed property (P1) above tells us that that every object can be generated under (possibly iterated) homotopy colimits by the cells. While we certainly want property (P1) to hold, it does not lend itself to a universal property as there is no mention of {\em how} the objects are built from the cells. In particular if we look at the full subcategory of cells $\GG$ and consider the canonical colimit
\begin{equation*}
	\colim_{C_i \in \GG, \; C_i \to X} C_i
\end{equation*}
this will almost never be equivalent to the object $X$, even in the $n=1$ case. 

\begin{example*}
	 If $X = \Delta^2$ is the `free-walking composition', i.e., the three element totally ordered set $(0 < 1 < 2)$, then the above colimit reproduces the category generated by $\partial \Delta^2$ instead of $\Delta^2$ itself. The former is the category which has three objects $0,1$ and $2$, no non-identity automorphisms, a single morphism $f_{01}: 0 \to 1$, a single morphism $f_{12}: 1 \to 2$, and two distinct morphisms from $0$ to $2$, $f_{02}$ and $f_{12} \circ f_{01}$. 
\end{example*}

It is natural to suppose that the theory of $(\infty,n)$-categories is strongly generated by some subcategory $\cR$ with contains the cells, but which is large enough so that the canonical homotopy colimit
\begin{equation*}
	\colim_{r \in \cR, \; r \to X} r
\end{equation*}
does reproduce the object $X$, for every object $X \in \cC$. Indeed in an extreme case we could take $\cR$ to be the entirety of all $(\infty,n)$-categories, although this would be circular as part of a definition of the theory of $(\infty,n)$-categories. We expect that there should be a much smaller $\cR$ which will work. 

This is closely related to Dan Dugger's notion of a {\em presentation} for a homotopy theory \cite{Dugger_CMHP, MR1870515}. 
If $\cC$ is a presentable quasicategory which is strongly generated by the subcategory $\cR$, then $\cC \simeq S^{-1}\Pre(\cR)$ is a localization of $\Pre(\cR)$, the quasicategory of presheaves of spaces on $\cR$, by a saturated class of morphisms $S$. 
This also gives rise to a universal property. If $\cD$ is any other presentable quasicategory which is strongly generated by a functor $f:\cR \to \cD$ (which induces a functor $\Pre(\cR) \to \cD$) and for which the morphisms of $S$ become equivalence, then $\cD$ is a localization of $\cC \simeq S^{-1}\Pre(\cR)$.

For the theory of $(\infty,n)$-categories, we suppose that $\cR$ must contain the cells, but the choice of $\cR$ is not unique. We may always enlarge it, for if we have containments $\cR \subseteq \cR' \subseteq \cC$ and $\cR$ strongly generates $\cC$, then so does $\cR'$ \cite[Rk. 4.4.7]{Lur0905}. 

A consequence of the techniques of \cite{BarSch1112} is that in fact we may obtain many  equivalent axiomizations of the theory of $(\infty,n)$-categories by allowing the category $\cR$ to vary. Some of the competing factors include:
\begin{itemize}
	\item The larger the subcategory $\cR$, the weaker the assumption that $\cR$ strongly generates $\cC$, and
	\item the larger the subcategory $\cR$, the easier it is to build comparison maps to theories strongly generated by smaller subcategories. However, 
	\item the larger the the subcategory $\cR$ is the larger the localizing class $S$ must be. For judicious choices of $\cR$, this class might have a simple description. 
	\item Finally, and perhaps most importantly, in order for an argument similar to the proof sketch of Proposition \ref{Prop:automorphisms} to hold, we must be able to show that $\cR$ is preserved by any automorphism and also that the automorphisms of $\cR$ are discrete and determined by the cells.  
\end{itemize}

This last item is essential for applications of the unicity theorem. If $\cR$ is too large, for example if $\cR$ is not an ordinary category but an full fledged $(\infty,1)$-category, then it might be difficult to compute its automorphisms explicitly. We will have gained nothing. 

This is the problem with trying to take $\cR$ to be, say, all strict $n$-categories. We would have to compute its automorphisms as a full subcategory of the theory of $(\infty,n)$-categories, which is tantamount to knowing precisely what the weak functors between strict $n$-categories should be. 

Another obvious candidate is to take $\cR$ to be the category of all gaunt $n$-categories. This should coincide with $\tau_{\leq 0} \cC$, and so is an ordinary category. It is easy to compute that automorphisms of $\cR$ in this case are the discrete group $(\ZZ/2)^n$, and so this seems like a great starting point. However in this case it is, as of this writing, unclear how to describe an appropriate class $S$ in any sort of explicit fashion. 

In a different direction, the quasicategory version of Rezk's theory of $\Theta_n$-spaces has a description of this form. In this case $\cR = \Theta_n$, and the class $S$ is generated by a countable collection of maps which corepresent higher versions of the Segal maps together with maps corepresenting certain `completeness maps' \cite{MR2578310, MR2740648}. In other words we may characterize Rezk's theory of $\Theta_n$-spaces by the following properties:
\begin{itemize}
	\item It is strongly generated by the subcategory $\Theta_n$.
	\item The countable collection of maps corepresenting the higher Segal and completeness maps are equivalences. 
	\item It is universal with respect to these first two properties.   
\end{itemize}
This last condition means that any presentable quasicategory satisfying the first two properties is a localization of Rezk's theory. Of course the above is not so much an axiomatization as the definition Rezk's theory. 

In the next section we will describe some additional properties that the theory of $(\infty,n)$-categories should possess
 that will enable us to reduce the description of $S$ to a finite amount of data. It will also lead us to consider a larger and more general category for $\cR$.

\subsection{Inner Homs} A fundamental property of the theory of $(\infty,n)$-categories is that it should have Cartesian products and internal homs. That is for any pair of $(\infty,n)$-cateogries $X$ and $Y$, there should be an $(\infty,n)$-category $\Fun(X,Y)$ whose objects are the (weak) functors from $X$ to $Y$, whose morphisms are the weak transformations between these, etc. 
If $\cC$ is a presentable quasicategory, which we are tacitly assuming, this is equivalent to the statement that for all $X \in \cC$ the functor
\begin{equation*}
	X \times (-): \cC \to \cC/X
\end{equation*}
preserves (homotopy) colimits. 

In fact there is a stronger property that the theory of $(\infty,n)$-categories $\cC$ satisfies. Not only does $\cC$ have internal homs, but also
\begin{itemize}
	\item [(P3)] For each $k$, the overcategories $\cC / C_k$ (over the $k$-cell $C_k$) admit internal homs.
\end{itemize}
 This is equivalent to the statement that for all $X \to C_k$ the functor
\begin{equation*}
	X \times_{C_k} (-): \cC/{C_k} \to C/X
\end{equation*}
preserves (homotopy) colimits. In the case $n=1$, this was proven for quasicategories by Joyal \cite[Th. 7.9]{Joyal_CRM}. For higher $n$ this is a calculation which must be carried out in each model. For Rezk's $\Theta_n$-spaces, the case of fiber products over the $0$-cell (i.e. ordinary products) follows from the main results of \cite{MR2578310, MR2740648}.

Note that even when $n=1$, for general $Z$ the over categories $\cat_{(\infty,1)}/Z$ do {\em not} possess internal homs. This is one of the reasons that $\cat_{(\infty,n)}$ is {\em not} an $\infty$-topos. For example consider the follow square, which is a pushout square in $\cat_{(\infty,n)}$ (and also $\cat$):
\begin{center}
\begin{tikzpicture}
	\node (LT) at (0, 1.5) {$\Delta^{\{1\}}$};
	\node (LB) at (0, 0) {$\Delta^{\{0,1\}}$};
	\node (RT) at (2, 1.5) {$\Delta^{\{1,2\}}$};
	\node (RB) at (2, 0) {$\Delta^{\{0,1,2\}}$};
	\draw [->] (LT) -- node [left] {$$} (LB);
	\draw [->] (LT) -- node [above] {$$} (RT);
	\draw [->] (RT) -- node [right] {$$} (RB);
	\draw [->] (LB) -- node [below] {$$} (RB);
	\node at (1.5, 0.5) {$\lrcorner$};
\end{tikzpicture}
\end{center}
This may be considered as a diagram of objects over $Z = \Delta^{\{0,1,2\}}$, where it remains a pushout square. If the quasi-category $\cat_{(\infty,1)}/Z$ had internal homs, then this would remain a pushout square after taking fiber products $(-) \times_Z Y$ for any $Y \to Z = \Delta^{\{0,1,2\}}$. Let $Y = \Delta^{\{0,2\}}$, with its inclusion into $Z$. Then the square of fiber products is as follows:
\begin{center}
\begin{tikzpicture}
	\node (LT) at (0, 1.5) {$\emptyset$};
	\node (LB) at (0, 0) {$\Delta^{\{0\}}$};
	\node (RT) at (2, 1.5) {$\Delta^{\{2\}}$};
	\node (RB) at (2, 0) {$\Delta^{\{0,2\}}$};
	\draw [->] (LT) -- node [left] {$$} (LB);
	\draw [->] (LT) -- node [above] {$$} (RT);
	\draw [->] (RT) -- node [right] {$$} (RB);
	\draw [->] (LB) -- node [below] {$$} (RB);
\end{tikzpicture}
\end{center}
Since $\Delta^{\{0,2\}} \not\simeq \Delta^{\{0\}} \sqcup \Delta^{\{2\}}$, this fails to be a pushout square. 

The $k$-cells are special. The quasicategories of objects over the cells do admit internal homs. One reason for this is that a map to a cell can be understood as an analog of a {\em correspondence} or {\em distributor} of higher categories. For example suppose that $M \to \Delta^{[1]}$ is a map of $(\infty,n)$-categories. The fibers $M_0$ and $M_1$ over $0$ and $1$ will be $(\infty,n)$-categories. The rest of the data of $M$ consists of a functor
\begin{equation*}
	M_1 \times M_0^\op \to \cat_{(\infty,n-1)}.
\end{equation*}
as was already well-known for quasicategories in the $n=1$ case from Joyal's work \cite{Joyal_CRM} (see also \cite[Rk. 2.3.1.4]{MR2522659}). 

In the higher categorical situation one expects maps to higher cells to similarly be described as higher correspondences between lower correspondences. If one was able to prove such a translation, then the existence of internal homs for the over categories $\cat_{(\infty,n)}/C_k$ would follow formally from the existence of ordinary internal homs in $\cat_{(\infty,n)}$. Moreover these higher correspondences should eventually help in the construction of a higher version of a `pro-arrow equipment' and a theory of `formal' higher category theory \cite{Shu0911_NCC}. This point of view is still conjectural and any more discussion would take us much too far afield, but in any case it is certainly expected that the over categories $\cat_{(\infty,n)}/C_k$ should have internal homs.

With property (P3) the infinite number of Segal maps used, for example, to construct Rezk's localization defining $\Theta_n$-spaces can be recovered from a finite number. For example the length-three Segal map
\begin{equation*}
	\Delta^{\{0,1\}} \cup^{\Delta^{\{1\}}} \Delta^{\{1,2,3\}} \to \Delta^{\{0,1, 2,3\}}
\end{equation*}
can be obtained from the length-two Segal map
\begin{equation*}
	\Delta^{\{0,1\}} \cup^{\Delta^{\{1\}}} \Delta^{\{1,2\}} \to \Delta^{\{0,1, 2\}}
\end{equation*}
by taking the fiber product with $\Delta^{\{1,2,3\}}$ over $\Delta^{\{1,2\}}$, where the maps
\begin{align*}
	\Delta^{\{0,1, 2\}} & \to \Delta^{\{1, 2\}} \\
	\Delta^{\{1, 2, 3\}} & \to \Delta^{\{1, 2\}}
\end{align*}
are define by sending $0$ and $1$ to $1$, and $2$ and $3$ to $2$. 
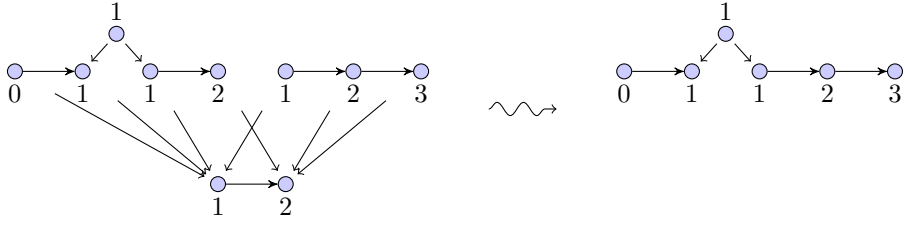
\begin{center}
\begin{tikzpicture}[inner sep=2pt, decoration=snake, xscale=0.9]
	\node (A) at (0, 1.5) [circle,draw,fill=blue!20] {};
		\node [below] at (A.south) {$0$};
	\node (B) at (1, 1.5 ) [circle,draw,fill=blue!20] {};
		\node [below] at (B.south) {$1$};
	\node (C) at (1.5, 2) [circle,draw,fill=blue!20] {};
		\node [above] at (C.north) {$1$};
	\node (D) at (2, 1.5) [circle,draw,fill=blue!20] {};
		\node [below] at (D.south) {$1$};
	\node (E) at (3, 1.5) [circle,draw,fill=blue!20] {};
		\node [below] at (E.south) {$2$};
	\draw [-stealth'] (A) -- (B);
	\draw [-stealth'] (D) -- (E);
	\draw [->, shorten <=2pt, shorten >=2pt] (C) -- (B);
	\draw [->, shorten <=2pt, shorten >=2pt] (C) -- (D);
	
	\node (F) at (3, 0) [circle,draw,fill=blue!20] {};
		\node [below] at (F.south) {$1$};
	\node (G) at (4, 0) [circle,draw,fill=blue!20] {};
		\node [below] at (G.south) {$2$};
	\draw [-stealth'] (F) -- (G);
	
	\node (H) at (4, 1.5) [circle,draw,fill=blue!20] {};
		\node [below] at (H.south) {$1$};
	\node (I) at (5, 1.5) [circle,draw,fill=blue!20] {};
		\node [below] at (I.south) {$2$};
	\node (J) at (6, 1.5) [circle,draw,fill=blue!20] {};
		\node [below] at (J.south) {$3$};
	\draw [-stealth'] (H) -- (I);
	\draw [-stealth'] (I) -- (J);
	
	\draw [->, shorten <=5mm, shorten >=1mm] (E) -- (G);
	\draw [->, shorten <=5mm, shorten >=1mm] (D) -- (F);
	\draw [->, shorten <=5mm, shorten >=1mm] (B) -- (F);
	\draw [->, shorten <=5mm, shorten >=1mm] (A) -- (F);
	
	\draw [->, shorten <=5mm, shorten >=1mm] (H) -- (F);
	\draw [->, shorten <=5mm, shorten >=1mm] (I) -- (G);
	\draw [->, shorten <=5mm, shorten >=1mm] (J) -- (G);
	
	\draw [decorate, ->] (7,1) -- (8,1);
	
	\node (AA) at (9, 1.5) [circle,draw,fill=blue!20] {};
		\node [below] at (AA.south) {$0$};
	\node (BB) at (10, 1.5 ) [circle,draw,fill=blue!20] {};
		\node [below] at (BB.south) {$1$};
	\node (CC) at (10.5, 2) [circle,draw,fill=blue!20] {};
		\node [above] at (CC.north) {$1$};
	\node (DD) at (11, 1.5) [circle,draw,fill=blue!20] {};
		\node [below] at (DD.south) {$1$};
	\node (EE) at (12, 1.5) [circle,draw,fill=blue!20] {};
		\node [below] at (EE.south) {$2$};
	\node (FF) at (13, 1.5) [circle,draw,fill=blue!20] {};
		\node [below] at (FF.south) {$3$};
	
	\draw [-stealth'] (AA) -- (BB);
	\draw [-stealth'] (DD) -- (EE);
	\draw [-stealth'] (EE) -- (FF);
	\draw [->, shorten <=2pt, shorten >=2pt] (CC) -- (BB);
	\draw [->, shorten <=2pt, shorten >=2pt] (CC) -- (DD);

\end{tikzpicture}
\end{center}

\subsection{The Axioms} 
At last we are ready to give the axiomatization. In light of the above discussion it is natural to introduce the category $\Upsilon_n$ which is the smallest subcategory of the gaunt $n$-categories which contains the cells and is closed under retracts and pullbacks over cells $(-) \times_{C_i} (-)$. 

The four axioms are the following.  Let $\cC$ be a presentable quasicategory equipped with a fully-faithful functor $\Upsilon_n \to \cC$ (we will later see that up to the automorphisms of $\Upsilon_n$ this functor is uniquely determined from $\cC$ alone). 
\begin{enumerate}
\item [(A1)] $\cC$ is strongly generated from $\Upsilon_n$, so that the canonical map 
\begin{equation*}
	\colim_{r \in \Upsilon_n, \; r \to X} r  \to X
\end{equation*}
is an equivalence for all $X \in \cC$;
\item [(A2)] For each cell $C_k$, $0 \leq k \leq n$, the over category $\cC/C_k$ has internal homs. Equivalently, for each object $X$, the functor 
\[
X \times_{C_i} (-): \cC/C_i \to \cC
\]
preserves homotopy colimits;
\item [(A3)] In $\cC$ a certain finite list of colimit equations is satisfied;
\item [(A4)] $\cC$ is universal with respect to the preceding three axioms.  That is if $\cD$ also satisfies (A1)-(A3), then we have a pair of adjoint functors $\cC \leftrightarrows \cD$ with the right adjoint an inclusion. This pair is unique if we require it to preserve the inclusion of $\Upsilon_n$. 
\end{enumerate}

\noindent The colimit equations in (A3) are all of the form: a certain colimit of objects in the image of $\Upsilon_n$ is equivalent (via the canonical map) to another object in the image of $\Upsilon_n$. In the case $n=1$, we need exactly four such colimit equations:
\begin{align*}
	\emptyset & \stackrel{\sim}{\to} f(\emptyset) \\
	f(C_1) \cup^{f(C_0)} f(C_1) & \stackrel{\sim}{\to} f(\Delta^{[2]}) \\
	f(\Delta^{\{0,1,2\}}) \cup^{f(\Delta^{\{0,2\}})} f(\Delta^{\{0,1,2\}}) & \stackrel{\sim}{\to} f(C_1 \times C_1) \\
	\left( f(C_0) \sqcup f(C_0) \right) \cup^{\left( f(\Delta^{\{0,2\}}) \sqcup f(\Delta^{\{1,3\}}) \right)} f(\Delta^{[3]}) & \stackrel{\sim}{\to} C_0
\end{align*}
where $f: \Upsilon_n \to \cC$ denotes the inclusion. The first three of these have a nice conceptual interpretation: they are precisely the relations needed to write $C_1 \times C_1$ and $C_0 \times_{C_1} C_0$ as iterated colimits of $C_0$ and $C_1$. The last equation implements completeness. 

For general $n$, the same pattern persists. There will be three families of equations which, conceptually, are needed to write $C_i \times_{C_j} C_k$ as an iterated colimit of cells, and there will be a fourth family expressing completeness. In total the number of equations grows approximately as $n^2$. 

There are several useful consequences of these axioms \cite{BarSch1112}:
\begin{itemize}
\item The cells do, in fact, generate everything under homotopy colimits;
\item $\tau_{\leq 0} \cat_{(\infty,n)}$ does consist of exactly the gaunt $n$-categories;
\item The derived topological group of automorphisms of $\cat_{(\infty,n)}$, or indeed an $\cC$ satisfying (A1)-(A3), is the discrete group $(\ZZ/2)^n$.
\item If $\cM$ and $\cN$ are combinatorial model categories whose homotopy theories satisfy the axioms, then $\cM$ and $\cN$ are connected by a zig-zag of Quillen equivalences (\cite[Rk. A.3.7.7]{MR2522659}, \cite{Dugger_CMHP}).
\item If $\cM$ and $\cN$ are model categories whose homotopy theories satisfy the axioms, if $L: \cM \rightleftarrows \cN :R$ is a Quillen adjunction, and  if $L$ preserves cells (up to weak equivalences), then $(L,R)$ is a Quillen equivalence.
\end{itemize}
\noindent Theorem \ref{thm:All_Models_Sat_Axioms} now applies and says that nearly all the models described so far satisfy these axioms, and hence are equivalent. 
Moreover we have some immediate consequences:

\begin{theorem*}[Rezk \cite{MR2578310}]
 $\Theta_n$-spaces, $\Theta_n \mathrm{Sp}$, satisfy the homotopy hypothesis.
\end{theorem*}

\begin{cor*}
All of the equivalent models satisfy the homotopy hypothesis.
\end{cor*}

\begin{theorem}[Simpson \cite{MR2883823}]
	The Segal $n$-categories satisfy the stabilization hypothesis.
\end{theorem}

\begin{cor*}
All of the equivalent models satisfy the stabilization hypothesis.
\end{cor*}

\begin{theorem*}[Lurie \cite{MR2555928}]
Barwick's $n$-fold complete Segal spaces, $\mathrm{CSS}_n$, satisfy the cobordism hypothesis.
\end{theorem*}

\begin{cor*}
All of the equivalent models satisfy the cobordism hypothesis.
\end{cor*}

While the unicity theorems are initially about the homotopy theory or $(\infty,1)$-category of $(\infty,n)$-categories, part of the axiomatization includes the existence of internal homs. Thus different theories of $(\infty,n)$-categories, which are equivalent as $(\infty,1)$-categories, will also give rise to equivalent `categories enriched in $(\infty,n)$-categories'. (The quotation marks indicate that this will probably be a weak enrichment which we will not make precise). As such categories will also be a model of $(\infty,n+1)$-categories, we may also deduce the uniqueness of the $(\infty,n+1)$-category of $(\infty,n)$-categories.

\bibliographystyle{amsalpha}
\bibliography{Bibliography_CSP}

\end{document}